\documentclass[times]{amsart}
\usepackage[utf8]{inputenc}
\usepackage[left=2.5cm,right=2.5cm, top=1cm,bottom=2.3cm,bindingoffset=0.5cm]{geometry}
\usepackage{ytableau}
\usepackage{xcolor}
\usepackage{wrapfig}

\usepackage{amscd, amsmath, amsfonts, amssymb, amsthm, fancyhdr, epsfig}
\usepackage[all,cmtip]{xy}
\usepackage{tikz}
\usetikzlibrary{decorations.pathreplacing}
\usepackage{bbm}
\usepackage{mathtools}
\usetikzlibrary{patterns}

\newcommand{\h}{{\mathfrak h}}
\newcommand{\gl}{{\mathfrak {gl}}}
\newcommand{\ssl}{{\mathfrak {sl}}}
\newcommand{\Z}{{\mathbb Z}}
\newcommand{\N}{{\mathbb N}}
\newcommand{\Q}{{\mathbb Q}}
\newcommand{\C}{{\mathbb C}}

\newcommand{\cf}{\mathcal F}
\newcommand{\cc}{\mathcal C}
\newcommand{\cp}{\mathcal P}
\newcommand{\ck}{\mathbbm{k}}
\newcommand{\one}{\mathbbm{1}}
\renewcommand{\bar}{\overline}
\newcommand{\Rep}{\operatorname{Rep}}

\newcommand{\Ker}{\operatorname{Ker}}
\newcommand{\colim}{\operatorname{colim}}

\renewcommand{\dim}{\operatorname{dim}}
\renewcommand{\deg}{\operatorname{deg}~}

\newcommand{\End}{\operatorname{End}}
\newcommand{\Hom}{\operatorname{Hom}}
\newcommand{\Ind}{\operatorname{Ind}}
\newcommand{\hc}{{HC}}
\newcommand{\nn}{{(n)}}
\newcommand{\tens}[1]{%
  \mathbin{\mathop{\otimes}\displaylimits_{#1}}%
}

\renewcommand{\l}{\lambda}
\newcommand{\bl}{{\boldsymbol\l}}
\newcommand{\bm}{{\boldsymbol\mu}}
\newcommand{\bn}{{\boldsymbol\nu}}
\newcommand{\bF}{{\boldsymbol F}}

\def\blacksquare{\hbox{\vrule width 5pt height 5pt depth 0pt}}

\theoremstyle{plain}
\newtheorem{theorem}{Theorem}[section]

\newtheorem{claim}[theorem]{Claim}
\newtheorem{hyp}[theorem]{Conjecture}
\newtheorem{lemma}[theorem]{Lemma}
\newtheorem{definition}[theorem]{Definition}
\newtheorem{corollary}[theorem]{Corollary}
\newtheorem{proposition}[theorem]{Proposition}

\theoremstyle{definition}
\newtheorem{remark}[theorem]{Remark}

\newtheorem{ex}[theorem]{Example}

\numberwithin{equation}{theorem}

\begin{document}

\title{Harish-Chandra bimodules of finite K-type in Deligne categories}

\author{Alexandra Utiralova\\ with an Appendix by Serina Hu and Alexandra Utiralova}

\begin{abstract}
	We continue the study of Harish-Chandra bimodules in the setting of the Deligne categories $\Rep(G_t)$, that we started in \cite{U}. In this work we construct a family of Harish-Chandra bimodules that generalize simple finite dimensional bimodules in the classical case. It turns out that they have finite $K$-type, which is a non-vacuous condition for the Harish-Chandra bimodules in $\Rep(G_t)$. The full classification of (simple) finite $K$-type bimodules is yet unknown.
	
	This construction also yields some examples of central characters $\chi$ of the universal enveloping algebra $U(\mathfrak{g}_t)$ for which the quotient $U_\chi$ is not simple, and, thereby, it allows us to partially solve Problem 3.23 posed in \cite{E}.
\end{abstract}	
	
\maketitle

\tableofcontents
\section{Introduction}

In this work we provide a construction of a family of Harish-Chandra bimodules in the setting of Deligne categories $\Rep(G_t)$, that generalize simple finite dimensional bimodules $\Hom_\ck(V_\bm,V_\bl)$ in the classical case.

\subsection{Deligne categories}
The Deligne categories are the interpolation of the categories of representations of classical groups (e.g. general linear group $GL_n,$ orthogonal group $O_n$, symplectic group $Sp_{2n},$ symmetric group $S_n$) to complex rank, i.e. to non-integer values of $n$.
They were first introduced as an example in the paper by Deligne and Milne \cite{DM}, where the category $\Rep(GL_t)$ for $t$ not necessarily integer was defined. In his later papers Deligne also introduced  the categories $\Rep(O_t), \Rep(Sp_{2t})$ and $\Rep(S_t)$, interpolating the categories of representations of the groups $O_n, Sp_{2n}$ and $S_n$ correspondingly, and also suggested the ultraproduct realization of these categories \cite{D2}, \cite{D1}. We give the definition of the categories $\Rep(G_t)$ for $G = GL, O, Sp$ in Section $2$, we will not work with the categories $\Rep(S_t)$ in this paper.

The Deligne categories provide means to study the stable behavior of representation theoretic constructions for classical groups $GL_n, O_n, Sp_{2n}$ as $n$ goes to infinity. In particular, one can transfer many classical notions and constructions to the complex rank and try to prove some analogues of classical theorems. This direction of study of the Deligne categories was suggested by Pavel Etingof in \cite{E1} and \cite{E}. 

\subsection{Harish-Chandra bimodules}

In this work we continue our study of Harish-Chandra bimodules in $\Rep(G_t)$ started in \cite{U}.

Let $\mathfrak{g}_t$ be the Lie algebra in $\Rep(G_t)$ defined in Subsection \ref{suniv}, it acts \textbf{naturally} on objects of $\Rep(G_t)$ (see Definition \ref{defnatact}). Let  $\mathfrak{g}_t^{\rm diag}\subset \mathfrak{g}_t\oplus\mathfrak{g}_t^{op}$ denote the diagonal copy of $\mathfrak{g}_t$. Let $U(\mathfrak{g}_t)$ be the universal enveloping algebra of $\mathfrak g_t$ (see \ref{defunivenvalg}).

Harish-Chandra bimodules for $G_t$ were defined in \cite{E} as follows. {A $(\mathfrak{g}_t,\mathfrak{g}_t)$-bimodule $M\in\Ind(\Rep(G_t))$ is a Harish-Chandra bimodule if it is finitely generated (that is to say, it is a quotient of $U(\mathfrak{g}_t)\otimes U(\mathfrak{g}_t^{\mathrm{op}})\otimes X$ for some $X\in\Rep(G_t)$), the action of $\mathfrak{g}_t^{\rm diag}$ on $M$ coincides with the \textbf{natural} action of $\mathfrak{g}_t$ and both copies of $Z(U(\mathfrak{g}_t))$ act finitely on $M$.} For more details see Definition \ref{defhcbimod}.

We say that a Harish-Chandra bimodule $M$ has finite $K$-type if for any simple object $X\in\Rep(G_t)$ the multiplicity
$$
\dim \Hom_{\mathfrak{g}_t}(X, M|_{\mathfrak g_t^{\rm diag}})
$$
is finite.

Every classical Harish-Chandra bimodule has finite $K$-type (see \cite{BG}, Proposition 5.3). However, in complex rank the situation is drastically different. It turns out that the finite $K$-type condition doesn't hold even for one of the easiest examples of Harish-Chandra bimodules, the quotient of $U(\mathfrak g_t)$ by the kernel of some central character $\chi$, which we denote by   $U_\chi$ (see Example \ref{exuniv}).

In this paper we construct a family of finite $K$-type Harish-Chandra bimodules $\underline\Hom(\bm, \bl)$ interpolating the finite-dimensional bimodules $\Hom_\ck(V_\bm, V_\bl)$ (see Section $3$ and Appendix).

When we put $\mu = \l$, the resulting bimodule $\underline\End(\l)$ is a quotient of $U_\chi$ for some central character $\chi$ computed in Section $4$. This allows us to construct a family of central characters $\chi$ such that $U_\chi$ is not a simple algebra, and thus, partially answer Problem 3.23 posed in \cite{E}.

\subsection{Categorical $\ssl_\Z$-action}

Following the work of Inna Entova-Aizenbud \cite{EA}, in which she introduced the categorical action of $\ssl_\Z$ on the Deligne category $\Rep(GL_t)$, we define two commuting categorical type $A$ actions on the category of Harish-Chandra bimodules $\hc(\gl_t)$. 

Let $\cc$ be some rigid symmetric monoidal category and $V\in {ob}~ \cc$. As defined in Section $6.2$ in \cite{EA}, the category $\Rep_\cc(\gl(V))$ of objects in $\cc$ with the action of $\gl(V) = V\otimes V^*$ enjoys the categorical type $A$ action given by the following data:
\begin{itemize}
    \item a pair ob biadjoint functors $(E,F)$, where $F= V\otimes (-), E=V^*\otimes (-)$;
    \item an operator $x\in \End(F)$, where $x_M\in \End_\cc(V\otimes M) \cong \Hom_\cc(V\otimes V^*\otimes M, M)$ is given by the action map $\gl_V\otimes M\to M$ for every $M\in {ob}~\cc$;
    \item an operator $\tau = \sigma_{V,V}\otimes 1 \in \End(F^2)$, where $\sigma_{-,-}$ is the symmetric braiding in $\cc$.
\end{itemize}

Bimodules over $\gl_t$ enjoy two commuting (left and right) actions of $\gl(V)=\gl_t$, which produce  two commuting categorical type $A$ actions on the category of Harish-Chandra bimodules.

In Section $6$ we consider the left  action on the subcategory $\hc_\bm$ generated (in a suitable sense) by simple bimodules $\underline\Hom(\bm,\bl)$ with a fixed $\bm$. We show that it induces commuting actions of multiple copies of $\ssl_\Z$ (the number of copies $m$ depends on $\bm$) on the complexified Grothendieck group $\mathbb C\tens{\mathbb Z} G(\hc_\bm)$. Theorems \ref{tgrmod1} and \ref{tgrmod2} provide a description of $\C\tens{\Z}Gr(\hc_\bm)$ as a module over $(\mathfrak{sl}_\Z)^{\otimes m}$.

\subsection{Structure of the paper}

This paper is structured as follows:
\begin{itemize}
    \item In Section 2 we give the definitions of the Deligne categories $\Rep(G_t)$, ultraproduct construction for them, universal enveloping algebra $U(\mathfrak{g}_t)$, its central characters, and Harish-Chandra bimodules.
    \item In Section 3 we prove a couple of technical lemmas allowing us to construct objects in $\Ind\Rep(G_t)$ by taking ultraproducts of classical objects.
    \item In Section 4 we study the stable behavior of finite dimensional Harish-Chandra bimodules $\Hom_\ck(V_\bm, V_\bl)$ and construct the bimodules $\underline\Hom(\bm, \bl)\in \Ind\Rep(G_t)$.
    \item In Section 5 we compute the central characters with which $U(\mathfrak{g}_t)$ acts on the bimodules $\underline\Hom(\bm,\bl)$.
    \item In Section 6 we construct nontrivial ideals in $U_\chi$ and provide a different construction of the bimodules $\underline\End(\bl)$.
    \item In Section 7 we introduce a categorical $\ssl_\Z$-action on the subcategory of bimodules of the form $\underline\Hom(\bm,\bl)$.
    \item In the Appendix (written jointly with Serina Hu) we generalize the above constructions to the categories $\Rep(O_t)$ and $\Rep(Sp_t)$.
    
\end{itemize}

\textbf{Remark.}
In Sections \ref{s3} -- \ref{sslz} we work with the category $\Rep(GL_t)$. 

\begin{section}*{Acknowledgements}
 We are thankful to Pavel Etingof for suggesting this problem and for many helpful discussions.

\end{section}

\section{Main definitions}

First, let us fix some notations.
From now on let $\mathbbm k := \bar\Q$. 

We will give a rather brief description of the Deligne categories below. For a more detailed explanation we refer the reader to the paper by Pavel Etingof \cite{E}. 

\begin{definition}
The Deligne category $\Rep(GL_t)$ is the Karoubian envelope of the rigid symmetric $\C$-linear tensor category generated by a single object $V$ (with dual $V^*$) of dimension $t\in \C$, such that for all $m\ge 1$ we have $\End(V^{\otimes m})=\C[S_m]$, where the symmetric group $S_m$ acts on $V^{\otimes m}$ by permuting the tensor factors; and $\Hom(V^{\otimes m}, V^{\otimes l}) = 0$ when $m\neq l$. 
\end{definition}

\begin{definition}
The Deligne category $\Rep(O_t)$ (resp. $\Rep(Sp_{t})$) is the Karoubian envelope of the rigid symmetric $\C$-linear tensor category generated by the single self-dual object $V$ of dimension $t$, such that $\End(V^{\otimes m})$ is isomorphic to the Brauer algebra $B_m(t)$ for all $m\ge 1$ and the isomorphism $V\to V^*$ is symmetric (resp. skew-symmetric). 
\end{definition} 

It turns out that for non-integer values of $t$ the categories $\Rep(GL_t),  \Rep(O_t)$ and $\Rep(Sp_t)$ are abelian and semisimple (see for example \cite{EGNO}, Subsection 9.12). 

\subsection{The ultraproduct construction}

Let $\cc_t$ denote $\Rep(GL_t), \Rep(O_t)$ or $\Rep(Sp_t)$. We denote by $G_n$ the group $GL_n, O_{2n+1}$, $Sp_{2n}$ respectively and let $\cc_n$ be the category of representations of $G_n$.  We denote by $V^{(n)}$ the defining representation of $G_n$ (it is $n$-dimensional for $GL_n$, $2n$-dimensional for $Sp_{2n}$, and $({2n+1})$-dimensional for $O_{2n+1}$). 

We will now state the result showing that $\cc_t$ can be constructed as a subcategory in the ultraproduct of the categories $\cc_n$. For more details on the ultrafilters and ulraproducts  see \cite{HK} or \cite{S}. A very detailed and nice explanation of the following construction can be found in \cite{K}.  The original statement for transcendental $t$ is due to Pierre Deligne \cite{D1} (but is left without proof). And the similar statement for all values of $t$ (requiring passing to positive characteristics) was proved in \cite{H} by Nate Harman.

Let $\cf$ be a nonprincipal ultrafilter on $\mathbb N$. We will fix some isomorphism of fields $ \prod_\cf \mathbbm k\simeq\mathbb C$.

\begin{theorem}{\cite{D1}}
\label{t1}
The category $\cc_t$ is equivalent to the full subcategory $\widetilde\cc$ in $\prod_{\mathcal F} \cc_n$ generated by $\widetilde V= \prod_\cf V^{(n)}$ under the operations of taking duals, tensor products, direct  sums, and direct summands, if $t$ is the image of $\prod_\cf  n$  under the isomorphism $\prod_\cf  \mathbbm{k} \simeq \C$. This equivalence is an equivalence of symmetric tensor categories over $\C$ and it sends $V$ to $\widetilde V$.
\end{theorem}

\begin{proof}
See, for example, Theorem 2.3. in \cite{U} for the proof.
\end{proof}

\begin{remark}
Note that in the case of $O_t$, one could take the ultraproduct exclusively over even $n$ or over odd $n$, which would produce equivalent categories after a sufficient choice of $\ck$-linear automorphism of $\prod_\cf \ck$ sending $\prod_\cf 2n$ to $\prod_\cf (2n+1)$.  
\end{remark}

\begin{remark}
Clearly, one can only obtain transcendental numbers $t$ as the image of $\prod_{\cf} n$ (or $\prod_{\cf} 2n$). To get this construction for algebraic $t$, one would need to consider the representations of $G_n$ over some fields of positive characteristic \cite{H}.

However, applying automorphisms of $\C$ over $\mathbbm k$, one can show that any transcendental $t$ can be obtained in this manner.
\end{remark}

So, from now on we assume for simplicity that $t$ is non-algebraic. We expect, however, that similar results hold for all non-integer values of $t$.

\subsection{The simple objects}

 In this subsection we give a description of simple objects in $\cc_t$, which coincide with indecomposable objects, as $t$ is assumed to be non-algebraic (thus, $\cc_t$ is semisimple). See \cite{E}, subsections 2.5 and 2.6 for a reference.

Let $\mathcal P_n$ denote the set of partitions of $n$. Let $\mathcal P = \bigcup \mathcal P_n$. For $\l\in \mathcal P$ let $\l_1\ge \l_2\ge\ldots\ge \l_k\ge\ldots$ denote its parts and let $\ell(\l)= \max(i~|~\l_i\neq 0)$ be the length of $\l$. For any $i>\ell(\l)$ we put $\l_i=0$. We will also use the notation $|\l|$ for the number $n$ such that $\l\in\mathcal P_n$.

Let $\rho$ denote the half-sum of all positive roots in the root system $X_n$, $X\in\{A,B,C,D\}$ and  let $W$ denote the corresponding Weyl group. Let $E_n$ be an $n$-dimensional Euclidean vector space with a choice of orthonormal basis $\{e_i\}_{i=1}^n$. Then we choose the following standard realization of root systems:
$$
A_{n-1}=\{e_i-e_j\}_{i\neq j}\subset E_n,
$$
$$
B_n=\{e_i\pm e_j\}_{i\neq j}\cup\{e_i\}_i\subset E_n,
$$
$$
C_n=\{e_i\pm e_j\}_{i\neq j}\cup\{2e_i\}_i\subset E_n,
$$
$$
D_n=\{e_i\pm e_j\}_{i \neq j}\subset E_n.
$$

Let $\mathfrak{g}_n$ denote the Lie algebra associated to $G_n$, and let $\mathfrak h_n\subset \mathfrak{g}_n$ be the Cartan subalgebra. We identify $E_n$ with $\simeq \h_n^*$.

\begin{theorem}\label{tsimpobj}(\cite{E}, Subsection 2.5)
The simple objects of $\Rep(GL_t)$ are labeled by the elements of $\mathcal P\times \mathcal P$. Moreover, if we let $V_\bl$ be the simple object corresponding to some $\bl=(\l,\bar\l)\in \mathcal P\times \mathcal P$, then
$$
V_{\bl} = \prod_\cf V^{(n)}_{\bl},
$$
where for big enough $n$ we define $V^{(n)}_{\bl}$ to be the simple $GL_n$-module of the highest weight $[\bl]_n$ with
$$
[\bl]_n= \underbrace{(\l_1, \ldots,\l_k, 0, \ldots, 0, -\bar\l_l,\ldots, -\bar\l_1)}_n = \sum_{i=1}^n \l_i e_i - \sum_{j=1}^n \bar\l_j e_{n-j+1}\in E_n.
$$

\begin{remark}
We will sometimes abuse the notation and omit the superscript $\nn$ in $V_\bl^\nn$.
\end{remark}

\end{theorem}

\begin{theorem}{(\cite{E}, Subsection 2.6)}
The simple objects of $\Rep(O_t)$ and $\Rep(Sp_t)$ are labeled by the elements of $\mathcal P$. Moreover, if $V_\l$ is the simple object corresponding to some $\l\in\mathcal P$, we have
$$
V_\l=\prod_\cf V^{(n)}_\l,
$$
where for big enough $n$ we define $V_\l^{(n)}$ to be the simple $O_{2n+1}$- (or $Sp_{2n}$-) module with the highest weight $[\l_n]$, with
$$
[\l]_n = \sum_{i=1}^n \l_ie_i\in E_n.
$$

Note that for big enough $n$ the weight $[\l]_n$ is indeed integral and dominant. Also, $V_\l^{(n)}$ is a well defined irreducible submodule of $\mathbb S_\l V^{(n)}$ (where $\mathbb S_\l$ is the Schur functor corresponding to the partition $\l$).
\end{theorem}

\begin{remark}
The categories $\Rep(O_t)$ and $\Rep(Sp_{-t})$ are equivalent as tensor categories and differ only by a change of the commutativity isomorphism. See the remark after Definition $2.10$ in \cite{E}. 
\end{remark}

\subsection{The universal enveloping algebra}\label{suniv}

\begin{definition}
We define
$$
\mathfrak{g}_t= \prod_\cf \mathfrak{g}_n.
$$
It is a Lie algebra object in the category $\cc_t$. 
\end{definition}

\begin{definition}\label{defnatact}
The Lie algebra $\mathfrak{g}_t$ acts on all objects of $\cc_t$ (as $\mathfrak{g}_n$ acts on objects of $\cc_n$). This action can be extended to the action on ind-objects. We will refer to it as the \textbf{natural} action of $\mathfrak{g}_t$.
\end{definition}

 When $G = GL$ we have
$$
\gl_t = V\otimes V^*
$$
as objects of $\cc_t$. 

Similarly, we have
$$
\mathfrak o_t = \Lambda^2 V,
$$
$$
\mathfrak{sp}_t = {\mathrm S}^2V
$$
as objects of their corresponding categories.

\begin{definition}\label{defunivenvalg}

We define the universal enveloping algebra $U(\mathfrak{g}_t)\in \Ind\cc_t$ to be the filtered ultraproduct of the universal enveloping algebras $U(\mathfrak{g}_n)$ with respect to the Poincar\'e-Birkhoff-Witt (PBW for short) filtration $F$. That is
$$
F^kU(\mathfrak{g}_t) = \prod_\cf F^kU(\mathfrak{g}_n)~~~~~~(*)
$$
and 
$$
U(\mathfrak{g}_t) = \colim_k F^kU(\mathfrak{g}_t).~~~~~~(**)
$$
\end{definition}

\textbf{Notation.} Whenever we have an ind-object $X$ in $\cc_t$ with filtration $F$ (i.e. $X  = \colim_k F^k X$ with $F^kX\in \cc_t$) and 
$$
F^k X = \prod_\cf X^{(n)}_k,
$$
we write 
$$
X=\prod_\cf^F X^{(n)},
$$
where $X^{(n)} = \colim_k X^{(n)}_k \in \Ind\cc_n$, and we put $F^kX^{(n)}= X^{(n)}_k$. We say $X$ is the filtered ultraproduct of $X^{(n)}$ with respect to the filtration $F$. (See Section \ref{s1} for details).

Thus, we have
$$
U(\mathfrak{g}_t) = \prod_\cf^F U(\mathfrak{g}_n),
$$
which is just a short notation for $(*)$ and $(**)$. As in the case of $\mathfrak{g}_t$, we have a natural action of $U(\mathfrak{g}_t)$ on all objects of $\cc_t$.

\begin{definition}
Let $Z(U(\mathfrak{g}_t))$ be the center of $U(\mathfrak g_t)$. That is, for every $k$ we define $F^k Z(U(\mathfrak g_t)) \subset F^k U(\mathfrak g_t)$ to be the maximal subobject $Z_k$, such that for any other subobject $X\subset U(
\mathfrak g_t)$ the commutator morphism $(m - m\circ \tau)$ restricted to $X\otimes Z_k$ is zero (here $m$ denotes the multiplication map $U(\mathfrak g_t)\otimes U(\mathfrak g_t)\to U(\mathfrak g_t)$ and $\tau$ is the braiding morphism for $U(\mathfrak g_t)\otimes U(\mathfrak g_t)$). We put $Z(U(\mathfrak g_t))= \colim_k F^k Z(U(\mathfrak{g}_t))$. It is naturally a filtered subalgebra of $U(\mathfrak g_t)$.
\end{definition}

Since $U(\mathfrak g_t)$ is generated by $\mathfrak g_t$, we have   $Z(U(\mathfrak{g}_t)) =  \Hom_{\mathfrak{g}_t}(\one, U(\mathfrak{g}_t)^{\rm {ad}}) = \Hom_{\Ind \cc_t}(\one, U(\mathfrak{g}_t))$ (where $\one$ denotes the unit object and we consider $U(\mathfrak g_t)$ as a $\mathfrak g_t$-module via the adjoint action), so it is an algebra over $\C$ (that is, it is an algebra in the category of vector spaces, which we consider as a full tensor subcategory of $\cc_t$ generated by $\one$). The embedding of $Z(U(\mathfrak g_t))$ into $U(\mathfrak g_t)$, then coincides with the natural map $\one \otimes \Hom_{\Ind\cc_t}(\one, U)\to U$ for $U= U(\mathfrak g_t)$.
We have
$$
Z(U(\mathfrak{g}_t))=\prod_\cf^F Z(U(\mathfrak{g}_n)).
$$

For any $n$ we have the Harish-Chandra isomorphism
$$
Z(U(\mathfrak{g}_n))\simeq\ck[\h_n^*]^{W},
$$
with the induced filtration by the degree of $W$-invariant polynomials. By the Chevalley–Shephard–Todd theorem, the ring of invariants is a polynomial ring. Let us consider two cases:

\textit{Case 1} ($G_n=GL_n$).

In this case we have $W=S_n$ and $Z(U(\mathfrak{g}_n))$ is isomorphic to the ring of symmetric polynomials in $n$ variables. It is generated by the power sum polynomials $p_k = \sum_i x_i^k$, which on the left hand side correspond to some central elements $C_k$, such that for any weight $\nu \in \h_n^*$ and the corresponding irreducible $\mathfrak{g}_n$-module $L(\nu)$ with the highest weight $\nu$ we have
$$
C_k|_{L(\nu)} = \sum_{i=1}^n \left((\nu_i+\rho_i)^k - \rho_i^k\right).
$$
(Note that this is just our choice of generators with particular normalization, there are many ways to do that). Now 
$$
Z(U(\mathfrak{g}_t))=\prod^F_\cf \ck[C_1,\ldots,C_n] \simeq \C[C_1,C_2,\ldots]
$$
- the polynomial algebra with generators $C_k, k\in\N$ with $\deg C_k = k$. 

\textit{Case 2} ($G_n=O_{2n+1}$ or $G_n=Sp_{2n}$).

In this case we have $W=(\Z/2\Z)^n\rtimes S_n$ and $Z(U(\mathfrak{g}_n))$ is isomorphic to the ring of symmetric polynomials in $n$ variables invariant under multiplying any of these variables by $-1$. That is, it is a polynomial algebra generated by the even power sum polynomials $p_{2k} = \sum_i x_i^{2k}$. On the left hand side we get central elements $C_{2k}$, such that for any weight $\nu\in E_{n}$  and the corresponding irreducible $\mathfrak{g}_n$-module $L(\nu)$ with the highest weight $\nu$ we have
\[
C_{2k}|_{L(\nu)} = \begin{cases}\sum \left( (\nu_i+\rho_i)^{2k}-\rho_i^{2k}\right), \text{ if } G_n = O_{2n+1},\\
\sum \left( (\nu_i+\rho_i-1)^{2k}-(\rho_i-1)^{2k}\right), \text{ if } G_n = Sp_{2n}.
\end{cases}
\]
(Note that it is again just a specific choice of normalization).

Then
$$
Z(U(\mathfrak{g}_t)) = \prod_\cf^F \ck[C_2,C_4,\ldots, C_{2n}] = \C[C_2,C_4,\ldots]
$$
- the polynomial algebra with generators $C_k, k\in 2\N$ with $\deg C_k=k$.

By central character $\chi$ of $U(\mathfrak{g}_t)$ we will mean any algebra homomorphism
$$
\chi:Z(U(\mathfrak{g}_t))\to \C.
$$
We will say $\chi$ is the ultraproduct of $\chi^{(n)}$ with 
$$
\chi^{(n)}: Z(U(\mathfrak{g}_n))\to \ck,
$$
if for any $k$
$$
\chi(C_k) = \prod_\cf \chi^{(n)}(C_k).
$$

\subsection{Harish-Chandra bimodules}

Let $\mathfrak g_t^{op}$ denote the object $\mathfrak g_t$ with the opposite Lie algebra structure. Let us denote by $\mathfrak{g}_t^{\rm diag}$ the diagonal copy of $\mathfrak{g}_t$ inside $\mathfrak{g}_t\oplus \mathfrak{g}_t^{op}$ via the embedding $(\text{id}_{\mathfrak g_t}, -\text{id}_{\mathfrak g_t})$.

\begin{definition}\label{defhcbimod}
A Harish-Chandra bimodule in $\cc_t$ is an object $M\in \Ind \cc_t$ with a structure of an $U(\mathfrak g_t)$-bimodule (or, in other words, a left $(U(\mathfrak{g}_t)\otimes U(\mathfrak{g}_t^{op})$-module), such that
\begin{itemize}
    \item it is a finitely generated $U(\mathfrak{g}_t)\otimes U(\mathfrak{g}_t^{op})$-module, i.e. it is a quotient of $(U(\mathfrak{g})\otimes U(\mathfrak{g}^{op}))\otimes X$ for some $X\in\cc_t$, 
    
    \item $\mathfrak{g}_t^{\rm diag}$ acts \textbf{naturally} on $M$ (in the sense of Definition \ref{defnatact}),
    
    \item both copies of $Z(U(\mathfrak{g}_t))$ act finitely on $M$ (that is the annihilator of $M$ in $Z((U(\mathfrak{g}_t)\otimes U(\mathfrak{g}_t^{op}))) = Z(U(\mathfrak{g}_t))\otimes Z(U(\mathfrak{g}_t))$ is an ideal of finite codimension).
\end{itemize}

\end{definition}

We denote the category of Harish-Chandra bimodules $\hc(\mathfrak{g}_t)$. 

For any $n$ we can similarly define a Harish-Chandra bimodule over $\mathfrak{g}_n$. By the \textbf{natural} action of $\mathfrak{g}_n^{\rm diag}$ we will mean the action of $\mathfrak{g}_n$ on ind-objects of $\cc_n$ that integrates to the action of $G_n$ (i.e. coincides with the derivative of the action of $G_n$), so that $M$ is a well-defined object in $\Ind\cc_n$. We will denote the corresponding category $\hc(\mathfrak{g}_n)$.

\begin{definition}
The category $\hc_{\chi,\psi}(\mathfrak{g}_t)$ is the full subcategory in $\hc(\mathfrak{g}_t)$ consisting of bimodules $M$ on which the left copy of $Z(U(\mathfrak{g}_t))$ acts via central character $\chi$ and the right copy of $Z(U(\mathfrak{g}_t))$ acts via central character $\psi$.
\end{definition}

\begin{ex}\label{exuniv}
Let $$U_\chi = U(\mathfrak{g}_t)/U(\mathfrak g_t) \Ker(\chi)$$ be the quotient of the universal enveloping algebra by the two-sided ideal generated by the kernel of a central character $\chi$. We equip $U_\chi$ with the structure of $U(\mathfrak{g}_t)$-bimodule via left and right multiplication. Then $U_{\chi}\in \hc_{\chi,\chi}(\mathfrak{g}_t)$.
\end{ex}

We similarly define the categories $\hc_{\chi,\psi}(\mathfrak{g}_n)$.

\begin{definition}

We say that a Harish-Chandra bimodule $M$ has \textbf{finite $\mathbf{K}$-type} (morally we let $K$  be the group corresponding to the Lie algebra $\mathfrak{g}_t^{\rm diag})$ if for any simple object $X\in \cc_t$ 
$$
\dim \Hom_{\Ind\cc_t}(X,M) <\infty.
$$

\end{definition}

\begin{remark}

The classical counterpart of this definition gives an empty condition on a Harish-Chandra bimodule $M\in\hc(\mathfrak{g}_n)$ (see \cite{BG}, Proposition 5.3). However, in the setting of the Deligne categories, things get more complicated. For example,  $U_\chi$ will not have a finite $K$-type (see \cite{E} Remark 3.21).
\end{remark}

The goal of this paper is to provide a construction of a family of Harish-Chandra bimodules of finite $K$-type, that will be a generalization, in some sense, of finite dimensional Harish-Chandra bimodules. That is, it will be a filtered ultraproduct of bimodules of the form
$$
\Hom_\ck (V_\bm^\nn, V_\bl^\nn) = V_\bl^\nn\otimes (V_\bm^\nn)^*.
$$

\section{Filtered ultraproducts}\label{s1}

In this section we will prove a few technical theorems, that will allow us to construct objects of $\Ind\cc_t$ as the filtered ultraproduct of objects of $\Ind\cc_n$.

\subsection{Definitions and notations}

\begin{definition}
For a collection $X_n$ of objects of $\Ind\cc_n$  with filtration $F$ we define the filtered ultraproduct as
$$
\prod^F_\cf X_n = \colim_k \prod_\cf F^kX_n.
$$
It is an object of $\Ind(\prod_\cf \cc_n)$.
\end{definition}

For any category $\cc$ we define the ind-completion $\Ind\cc$ as the full subcategory of presheaves on $\cc$ that are filtered colimits of representable objects. Then $\cc$ is a subcategory of $\Ind\cc$ via the Yoneda embedding $X\mapsto \Hom_{\cc}(\cdot, X)$.

Let $S$ denote the set of isomorphism classes of simple objects of $\cc_t$. That is, we identify $S$ with $\cp\times\cp$ if $G=GL$ and with $\cp$ when $G=O$ or $Sp$. We write $V_\bl$ for the simple module corresponding to $\bl\in S$.

Let us fix some filtration (by finite subsets) $F$ on $S$. It induces a filtration on each $\cc_n$ and on $\cc_t$. Since our category is semisimple, we have
$$
F^kX = \bigoplus_{\bl\in F^k S} V_{\bl}\otimes \Hom(V_\bl,X)
$$
for any $X\in\cc_n$ or $\cc_t$ (where we take $\Hom$ in the corresponding category).

Moreover, for any object $X= \prod_\cf X_n\in\cc_t$
$$
X\in F^k\cc_t \text{ if and only if } X_n\in F^k\cc_n \text{ for almost all } n, \text{ i.e. for all } n \text{ in some } U\in \cf.
$$

\begin{ex}
\label{exfiltr}
Let $G=GL$ and thus, $S=\cp\times \cp$.

Consider $F^k(\cp\times\cp) =\{\bl = (\l,\bar\l), \text{ such that } |\l|,|\bar\l|\le k\}$. Then if $l,m\le k$ 
$$
V^{\otimes l} \otimes (V^*)^{\otimes m} \in F^k\cc_t
$$
\end{ex}

\subsection{Some technical lemmas}

\begin{lemma}
\label{lwelldefup}
Let $X_n\in \Ind\cc_n$ be some filtered objects with filtration $\widetilde F$ and a function $c:S \times \Z_{\ge 0}\to \Z_{\ge 0}$, such that for almost all $n$, i.e. for all $n$ in some $U\in\cf$:
\begin{itemize}
    \item $\forall \bl \in S$ we have $[\widetilde F^k X_n: V^\nn_{\bl}]\le c(\bl,k)$,
    \item $\widetilde F^k X_n \in F^k \cc_n$.
\end{itemize}
Then 
$$
\prod_\cf^{\widetilde F} X_n\in \Ind\cc_t,
$$
that is, for all $k$
$$
\prod_\cf \widetilde F^kX_n\in \cc_t.
$$
\end{lemma}
\begin{proof}
It is enough to show that for all $k$ there exists $U_k\in \cf$, such that 
$$
\forall n\in U_k~\widetilde F^k X_n = \bigoplus_{\bl\in S} (V^\nn_\bl)^{\oplus c'(\bl,k)}
$$
for some $c':S\times\Z_{\ge 0}\to \Z_{\ge 0}$ independent of $n$. Let  $$f:U\to \{0,1,\ldots,c(\bl,k)\},$$ be defined by 
$$
f(n)=[\widetilde F^k X_n:V^\nn_\bl].
$$
Then, $U$ is the disjoint union of the sets $f^{-1}(a)$ for $0\le a\le c(\bl,k)$. So, there exists  $c'(\bl,k)\le c(\bl,k)$, such that $f^{-1}(c'(\bl, k)) = U_{\bl,k}\in\cf$ (otherwise, the compliment of $U$ is the finite intersection of sets in $\cf$, so it must be in $\cf$). Hence,
$$
\forall n\in U_{\bl,k}~ [\widetilde F^k X_n:V^\nn_\bl] = c'(\bl, k).
$$
Now we take $U_k = \bigcap_{\bl\in F^kS} U_{\bl,k}$.
\end{proof}

\begin{lemma}
\label{lequivup}
Let $X_n\in \cc_n$ be some filtered objects with filtration $\widetilde F$ and a function $c: S\to \Z_{\ge 0}$, such that for almost all $n$:
\begin{itemize}
    \item $\forall \bl\in S$ we have $[X_n:V^\nn_\bl]\le c(\bl)$,
    \item $\widetilde F^k X_n\subset F^kX_n$.
\end{itemize}
Then we have a natural isomorphism
$$
\prod_\cf^{\widetilde F} X_n \simeq \prod_\cf^F X_n
$$
as ind-objects.
\end{lemma}

{Note} that unlike Lemma \ref{lwelldefup}, here we must require that $X_n\in \cc_n$ and $[X_n:V^\nn_\bl]$ is bounded for almost all $n$, so that the ultraproduct $\prod_\cf^F X_n$ is a well-defined object in $\Ind\cc_t$.

\begin{proof}
Let 
$$
\boldsymbol F^k = \Hom_{\cc_t}(~\cdot~, \prod_\cf F^k X_n),
$$
$$
\widetilde \bF^k= \Hom_{\cc_t}(~\cdot~, \prod_\cf \widetilde F^k X_n).
$$
We want to show that $\colim \bF^k = \colim\widetilde\bF^k$. We can check this object-wise, so let us fix some $Y\in\cc_t$. Then there exists some $m$ such that $Y\in F^m\cc_t$, and thus, for any $k\ge m$
$$
\Hom_{\cc_t}(Y,\prod_\cf F^k X_n) = \Hom_{\cc_t}(Y, \prod_\cf F^mX_n). 
$$
Now, if $Y = \prod_\cf Y_n$, we have $Y_n\in F^m\cc_n$ for almost all $n$ and, therefore, 
$$
\Hom_{\cc_n}(Y_n,X_n) = \Hom_{\cc_n}(Y_n,F^mX_n).
$$
Thus, 
$$
\colim \bF^k(Y) = \Hom_{\cc_t}(Y,\prod_\cf F^mX_n) = \prod_\cf \Hom_{\cc_n}(Y_n,F^mX_n) = \prod_\cf \Hom_{\cc_n}(Y_n,X_n).~~~(\ref{lequivup}.1)
$$
On the other hand, since $\widetilde F^kX_n\subset F^k X_n$, there exists a natural map
$$
\colim {\widetilde \bF}^k(Y)\to \colim \bF^k(Y).
$$
It is left to check that $\colim \bF^k(Y)$ is indeed the colimit of $\widetilde \bF^k(Y)$ and satisfies the corresponding universal property.

Let $S$ be a $\C$-vector space with maps $\phi_k: \widetilde \bF^k(Y)\to S$, commuting with the natural inclusions $\widetilde \bF^k(Y)\to \widetilde \bF^{k+1}(Y)$. Since, $\mathrm{Vect}_\C = \prod_\cf \mathrm{Vect}_\ck$, we have 
$$
S = \prod_\cf S_n,
$$
for some $S_n\in \mathrm{Vect}_\ck$. Moreover, $\phi_k = \prod_\cf \phi_k^\nn$ with 
$$
\phi_k^\nn: \widetilde \bF^k(Y_n)\to S_n.
$$

For a fixed $n$ the collection of maps $\phi_k^{(n)}$ gives a map from the direct system of $\widetilde\bF^k(Y_n) = \Hom_{\cc_n}(Y_n, \widetilde F^k X_n)$ to $S_n$. Thus, there exists a map
$$
\phi^\nn: \colim \widetilde\bF^k(Y_n) = \Hom_{\cc_n} (Y_n, X_n)\to S_n.
$$
The equality coming from the fact that on each $X_n$ the filtration is finite.

Finally, define $\phi= \prod_\cf \phi^\nn: \prod_\cf \Hom_{\cc_n}(Y_n,X_n) \to S$. Thus, by $(\ref{lequivup}.1)$,
$$
\phi: \colim \bF^k(Y) \to S.
$$
Obviously, by construction, $\phi$ commutes with all appropriate  maps, showing that $\colim \bF^k(Y)$ indeed satisfies the universal property of the colimit of $\widetilde \bF^k(Y)$.
\end{proof}

\begin{ex}\label{expbw}
Let $\widetilde F$ be the PBW-filtration on $U(\gl_t)$. Then $\widetilde F^k U(\gl_t)$ is the image of $\bigoplus_{i=0}^k (V\otimes V^*)^{\otimes i} =  \bigoplus_{i=0}^k \gl_t^{\otimes i} \subset T(\gl_t)$ under the natural projection from the tensor algebra on $\mathfrak{g}_t$ to the universal enveloping algebra. So, if $F$ is the filtration from  Example \ref{exfiltr}, we have 
$$
\widetilde F^kU(\gl_t)\in F^k\cc_t.
$$
So, Lemma \ref{lequivup} will allow us to take ultraproducts of sub- and quotient-bimodules of $U_\chi$ with the induced filtration.
\end{ex}

\section{Finite dimensional Harish-Chandra bimodules}\label{s3}

For any $n$ and any integral dominant weights $\bl,\bm$ the space of homomorphisms
$$
\Hom_\ck(V_{\bm}^{(n)}, V^{(n)}_{\bl})
$$
is a simple finite-dimensional Harish-Chandra bimodule for $\mathfrak{g}_n$. As an $\mathfrak{g}_n^{\rm diag}$-module it is isomorphic to 
$$
V_{\bl}^\nn \otimes (V_{\bm}^\nn)^*.
$$
 
For any weight $\boldsymbol\nu\in \h_n^*$ let $\chi_{\boldsymbol\nu}$ denote the central character with which $Z(U(\mathfrak{g}_n))$ acts on the simple module $L(\nu)$. Now if $\chi= \chi_\bl$ and $\psi = \chi_\bm$, we have
$$
\Hom_\ck(V_\bm^{(n)}, V^{(n)}_\bl)\in\hc_{\chi,\psi}(\mathfrak{g}_n).
$$

We would like to take an ultraproduct of such modules to construct an HC-bimodule in $\Ind\cc_t$. However, we will need to impose some conditions on $\bl$ and $\bm$ and introduce some filtration in order to obtain a well-defined object in $\Ind\cc_t$.

From now on let us concentrate on the case $G = GL$. We will provide the corresponding constructions and theorems for $G=O$ and $Sp$ in the appendix.

\subsection{Classical case and multiplicities}

The dominant integral weights of $GL_n$ are in bijection with bipartions $\boldsymbol\l=(\l,\bar \l)\in \mathcal P\times \mathcal P$ with $\ell(\l)+\ell(\bar \l)\le n$ via the map
$$
\bl\mapsto \sum_{i=1}^n \l_ie_i - \sum_{j=1}^n \bar\l_j e_{n+1-j}.
$$

\begin{definition}
For $\l\in \cp$ let $d(\l)$ be the Dufree size (or rank) of $\l$, that is, the length of the main diagonal of the corresponding Young diagram.
\end{definition}

\begin{definition}

Let $\l'$ be the partition conjugate to $\l$ (that is, the Young diagram of $\l'$ is obtained from the diagram of $\l$ by reflecting about the main diagonal).
\end{definition}

For any partition $\l$ and some numbers $k,l\le d(\l)$ let us divide the diagram of $\l$ into four parts in the following way: we make a horizontal cut under the $k$-th row of $\l$ and a vertical cut to the right of the $l$-th column of $\l$.

The resulting four parts consist of a $k$-by-$l$ rectangle on the top left, diagram $\alpha$ on the top right, diagram $\beta'$ on the bottom left, and diagram $\gamma$ on the bottom right. See Figure \ref{fig1}.

\begin{figure}\caption{Cutting the diagram of a partition into four parts}\label{fig1}
\begin{tikzpicture}
\filldraw[fill=red!20!white, draw=red!20!white] (2,0) rectangle (3.5,-0.5);
\filldraw[fill=red!20!white, draw=red!20!white] (2,-0.5) rectangle (3,-1);
\filldraw[fill=red!20!white, draw=red!20!white] (2,-1) rectangle (2.5,-1.5);
\filldraw[fill=green!20!white, draw=green!20!white] (2,2.5) rectangle (7,2);
\filldraw[fill=green!20!white, draw=green!20!white] (2,2) rectangle (6,1.5);
\filldraw[fill=green!20!white, draw=green!20!white] (2,1.5) rectangle (5.5,1);
\filldraw[fill=green!20!white, draw=green!20!white] (2,1) rectangle (4.5,0.5);
\filldraw[fill=green!20!white, draw=green!20!white] (2,0.5) rectangle (4,0);
\filldraw[fill=yellow!20!white, draw=yellow!20!white] (0,0) rectangle (2,-2);
\filldraw[fill=yellow!20!white, draw=yellow!20!white] (0,-2) rectangle (1.5,-4.5);
\filldraw[fill=yellow!20!white, draw=yellow!20!white] (0,-4.5) rectangle (1,-5);
\filldraw[fill=yellow!20!white, draw=yellow!20!white] (0,-5) rectangle (0.5,-5.5);
\filldraw[fill=blue!20!white, draw=blue!20!white] (0,0) rectangle (2,2.5);
\draw (2,2.5) -- (2,0);
\draw (0,0) -- (2,0);
\draw (0,0) -- (0,2.5);
\draw (0,2.5) -- (7,2.5);
\draw (0,0) -- (4,0);
\draw (6,2) -- (7,2);
\draw (5.5,1.5) -- (6,1.5);
\draw (4.5,1) -- (5.5,1);
\draw (4,0.5) -- (4.5,0.5);
\draw (7,2.5) -- (7,2);
\draw (6,2) -- (6,1.5);
\draw (5.5,1.5) -- (5.5,1);
\draw (4.5,1) -- (4.5,0.5);
\draw (4,0.5) -- (4,0);
\node at (3.3,1.4){$\alpha$};
\node at (1,2.8){$l$};
\node at (-0.3,1.3){$k$};
\draw (0,0) -- (0,-5.5);
\draw (0,-5.5) -- (0.5,-5.5);
\draw (0.5,-5.5) -- (0.5,-5);
\draw (0.5,-5) -- (1,-5);
\draw (1,-5) -- (1,-4.5);
\draw (1,-4.5) -- (1.5,-4.5);
\draw (1.5,-4.5) -- (1.5,-2);
\draw (1.5,-2) -- (2,-2);
\draw (2,-2) -- (2,0);
\node at (0.8,-2.5){$\beta'$};
\draw (3.5,0) -- (3.5,-0.5);
\draw (3.5,-0.5) -- (3,-0.5);
\draw (3,-0.5) -- (3,-1);
\draw (3,-1) -- (2.5,-1);
\draw (2.5,-1) -- (2.5,-1.5);
\draw (2.5,-1.5) -- (2,-1.5);
\node at (2.5,-0.5){$\gamma$};
\node at (13,0){The resulting four parts consist of};
\node at (13, -1){a $k$-by-$l$ rectangle on the top left,};
\node at (13, -2){diagram $\alpha$ on the top right,};
\node at (13,-3){diagram $\beta'$ on the bottom left,};
\node at (13, -4){and diagram $\gamma$ on the bottom right.};
\end{tikzpicture}
\end{figure}
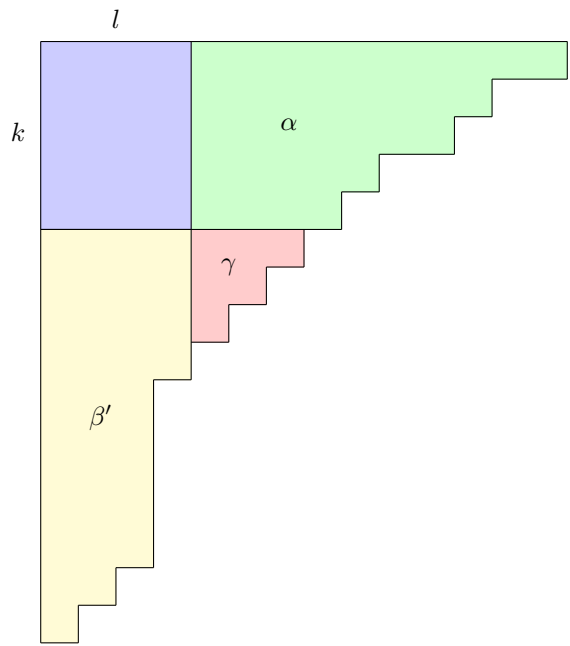

\noindent
We have for $i\le k$ and $j\le l$
$$
\alpha_i = \l_i - l,
$$
$$
\beta_j = \l_j'-k,
$$
and for $i\le \ell(\gamma) = \l'_{l+1}-k$
$$
\gamma_i = \l_{k+i} -l.
$$
Clearly, $\l$ is uniquely determined by the partitions $\alpha,\beta,\gamma$ (we recover $k$ and $l$ as $\ell(\alpha)$ and $\ell(\beta)$ correspondingly).

\begin{definition}\label{defcutdiag}
We write $\l=[\alpha, \beta, \gamma]$ if the partitions $\alpha, \beta$ and $\gamma$ are obtained from $\lambda$ by the cutting procedure described above for $k = \ell(\alpha), l = \ell(\beta)$. See Figure \ref{fig1}.
\end{definition}

 Note that not any triple of partitions produces a well-defined $\l$: we should also require that $\gamma_1\le \alpha_k$ and $\gamma'_1\le \beta_l$.

We will now show how one can construct two sequences of highest weights $\bl^\nn$ and $\bm^\nn$ such that for any $\bn\in\cp\times\cp$ the multiplicity
$$
\dim\Hom_{GL_n}(V^\nn_\bn,V^\nn_{\bl^\nn}\otimes(V^\nn_{\bm^\nn})^*)
$$
is constant for almost all $n$.

\begin{definition}
Let $\l,\mu$ be some partitions. Then we write $\l \subset \mu$ if $\l_i \le \mu_i$ for all $i$.
\end{definition}

Let us first consider the case of polynomial $GL_n$-representations. That is, we let $\bl = (\l, \emptyset)$ and $\bm=(\mu,\emptyset)$.

The following theorem is due to R.K. Brylinski \cite{B}. It was originally proved with the assumption $|\l|=|\mu|$. We will repeat the proof here, omitting this assumption. 

Let $s_\l$ denote the symmetric Schur polynomial corresponding to $\l\in\cp$, and let $(\cdot,\cdot)$ be the symmetric form making $\{s_\l\}$ into an orthonormal basis in the space of symmetric polynomials. Let $c_{\mu,\nu}^\l = c_{\nu, \mu}^\l$ denote the Littlewood-Richardson coefficient, i.e. the coefficient of $s_\l$ in the product $s_\mu s_\nu$:
$$
s_\mu s_\nu = \sum_\l c_{\mu,\nu}^\l  s_\l.
$$
Let $s_{\l/\mu}$ denote the skew Schur polynomial associated with the skew diagram $\l/\mu$. We have
$$
(s_\l,s_\mu s_\nu) = (s_{\l/\mu}, s_\nu).
$$

\begin{theorem}\cite{B} \label{tgupta} . Let $\bn=(\nu,\bar\nu)\in\cp\times\cp$ be some bipartition. Let $\bl=(\l,\emptyset), \bm=(\mu,\emptyset)$. \\ Suppose $n\ge \min(\ell(\l)+\ell(\bar\nu), \ell(\mu)+\ell(\nu))$. Then
$$
\dim\Hom_{Gl_n}(V_{\bn}^\nn, V^\nn_\bl\otimes (V_\bm^\nn)^*) = (s_{\l/\nu}, s_{\mu/\bar \nu}).
$$
\end{theorem}

\textit{Proof.}

For any bipartition $\bn=(\nu,\bar\nu)$ let  $^t\bn = (\bar\nu, \nu)$.  We have 
$$
(V_\bn^\nn)^* = V^\nn_{^t\bn}.
$$
Since 
\begin{equation}\label{eqdualmult}
[V_\bl^\nn\otimes (V^\nn_\bm)^*:V^\nn_\bn] = [V^\nn_\bm\otimes (V^\nn_\bl)^*: V^\nn_{^t\bn}],
\end{equation}
without loss of generality we can assume $n\ge \ell(\mu)+\ell(\nu).$

For an integer $a$, let $(a^n)$ denote the weight $(a,a,\ldots,a)\in \h_n^*$ and let $\boldsymbol\det^a$ denote the corresponding simple representation of $GL_n$. 
Let us take $a\ge \rm{max}(\bar\nu_1, \mu_1)$, then the weights $[\bn]_n+(a^n)$ and $[^t\bm]_n+(a^n)$ are dominant and all of their coordinates are non-negative, and thus $V_\bn^\nn\otimes \boldsymbol\det^a$ and $(V_\bm^\nn)^*\otimes \boldsymbol\det^a$ are polynomial representations of $GL_n$.

For any bipartition $\bn = (\nu, \bar\nu)$ and any $a\ge \bar\nu_1$ let  $\bn+(a^n)$ denote the partition corresponding to (the non-negative integral dominant weight)  $[\bn]_n + (a^n)$. 

Clearly,
$$
[V_\bl^\nn\otimes (V^\nn_\bm)^*:V^\nn_\bn] = [V_\bl^\nn\otimes (V^\nn_\bm)^*\otimes \boldsymbol\det^a:V^\nn_\bn\otimes \boldsymbol\det^a] = [V_\bl^\nn\otimes V^\nn_{^t\bm+(a^n)}:V^\nn_{\bn+(a^n)}],
$$
so we can use the Littlewood-Richardson rule to compute this multiplicity. 

One of the consequences of this is that $^t\bm + (a^n)\subset \bn + (a^n)$. Using equation \ref{eqdualmult}, we also see that $^t\bl+(b^n) \subset {^t\bn} + (b
^n)$ for $b\ge\rm{max}(\l_1, \nu_1)$. Thus, in order for the multiplicity $\dim\Hom_{Gl_n}(V_{\bn}^\nn, V^\nn_\bl\otimes (V_\bm^\nn)^*)$ to be nonzero, we must have $\nu\subset \l$ and $\bar\nu\subset \mu$ and $|\l|-|\mu|=|\nu|-|\bar\nu|$. So, we will assume these relations hold.

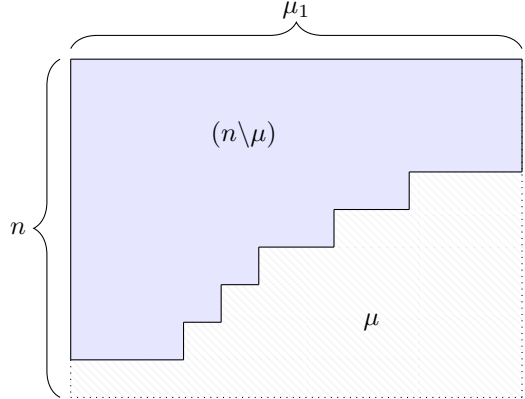
\begin{tikzpicture}
\draw[pattern=north west lines, pattern color = black!20!white!, draw=white, draw opacity=0.2] (0,1) rectangle (6,0.5);
\draw[pattern=north west lines, pattern color = black!20!white!, draw=white, draw opacity=0.2] (1.5,1.5) rectangle (6,1);
\draw[pattern=north west lines, pattern color = black!20!white!, draw=white, draw opacity=0.2] (2,2) rectangle (6,1.5);
\draw[pattern=north west lines, pattern color = black!20!white!, draw=white, draw opacity=0.2] (2.5,2.5) rectangle (6,2);
\draw[pattern=north west lines, pattern color = black!20!white!, draw=white, draw opacity=0.2] (3.5,3) rectangle (6,2.5);
\draw[pattern=north west lines, pattern color = black!20!white!, draw=white, draw opacity=0.2] (4.5,3.5) rectangle (6,3);
\filldraw[fill = blue!10!white!, draw = blue!10
!white!] (0,5) rectangle (6,3.5);
\filldraw[fill = blue!10!white!, draw = blue!10
!white!] (0,3.5) rectangle (4.5,3);
\filldraw[fill = blue!10!white!, draw = blue!10
!white!] (0,3) rectangle (3.5,2.5);
\filldraw[fill = blue!10!white!, draw = blue!10
!white!] (0,2.5) rectangle (2.5,2);
\filldraw[fill = blue!10!white!, draw = blue!10
!white!] (0,2) rectangle (2,1.5);
\filldraw[fill = blue!10!white!, draw = blue!10
!white!] (0,1.5) rectangle (1.5,1);
\draw (0,5)--(6,5);
\draw (6,5)--(6,3.5);
\draw (6,3.5)--(4.5,3.5);
\draw (4.5,3.5)--(4.5,3);
\draw (4.5,3)--(3.5,3);
\draw (3.5,3)--(3.5,2.5);
\draw (3.5,2.5)--(2.5,2.5);
\draw (2.5,2.5)--(2.5,2);
\draw (2.5,2)--(2,2);
\draw (2,2)--(2,1.5);
\draw (2,1.5)--(1.5,1.5);
\draw (1.5,1.5)--(1.5,1);
\draw (1.5,1)--(1,1);
\draw (1,1)--(0,1);
\draw (0,5)--(0,1);
\draw[dotted] (0,1)--(0,0.5);
\draw[dotted] (0,0.5)--(6,0.5);
\draw[dotted] (6,0.5)--(6,5);
\node at (3,5.65){$\mu_1$};
\node at (2.3,4){$(n\backslash\mu)$};
\node at (4,1.5){$\mu$};
\node at (11, 3){Let $(n\backslash \mu)$ denote the partition $^t\bm+(\mu_1^n)$.};

\draw[decorate,decoration={brace,amplitude=10pt, mirror, raise=4pt}](0,5)--(0,0.5);
\draw[decorate,decoration={brace,amplitude=10pt,raise=4pt}](0,5)--(6,5);
\node at (-0.7,2.75){$n$};
\end{tikzpicture}
$$
$$
We have
$$
\dim\Hom_{GL_n}(V^\nn_\bn,V^\nn_\bl\otimes V^\nn_{^t\bm}) = \dim\Hom_{GL_n}(V^\nn_{\bn+(\mu_1^n)}, V^\nn_\bl\otimes V^\nn_{(n\backslash \mu)}) = (s_{\bn+(\mu_1^n)}, s_\l s_{(n\backslash\mu)}).
$$
Let us denote the partition $\bn+(\mu_1^n)$ by $\theta$. Then
$$
(s_\theta, s_\l s_{(n\backslash \mu)})  = (s_{\theta/(n\backslash \mu)}, s_\l).
$$
Let us look at the diagrams of $\theta$ and $\theta/(n\backslash \mu)$. The latter is a well defined skew diagram, since the relation $(n\backslash \mu)\subset\theta$ is equivalent to $\bar\nu\subset \mu$.

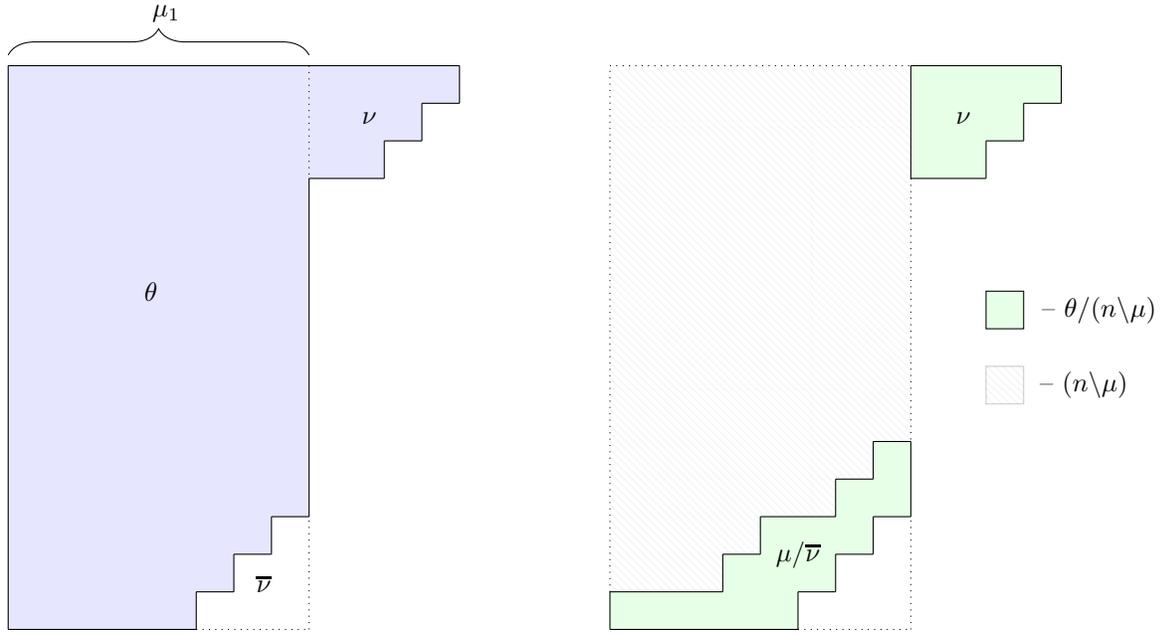
\begin{tikzpicture}
\filldraw[fill=blue!10!white!,draw=blue!10!white!] (2,10) rectangle (6,4);
\filldraw[fill=blue!10!white!,draw=blue!10!white!] (2,4) rectangle (5.5,3.5);
\filldraw[fill=blue!10!white!,draw=blue!10!white!] (2,3.5) rectangle (5,3);
\filldraw[fill=blue!10!white!,draw=blue!10!white!] (2,3) rectangle (4.5,2.5);
\filldraw[fill=blue!10!white!,draw=blue!10!white!] (6,10) rectangle (8,9.5);
\filldraw[fill=blue!10!white!,draw=blue!10!white!] (6,9.5) rectangle (7.5,9);
\filldraw[fill=blue!10!white!,draw=blue!10!white!] (6,9) rectangle (7,8.5);
\draw (2,10)--(8,10);
\draw (8,10)--(8,9.5);
\draw (8,9.5)--(7.5,9.5);
\draw (7.5,9.5)--(7.5,9);
\draw (7.5,9)--(7,9);
\draw (7,9)--(7,8.5);
\draw (7,8.5)--(6,8.5);
\draw[dotted] (6,10)--(6,8.5);
\node at (6.8,9.3){$\nu$};
\draw (6,8.5)--(6,4);
\draw (6,4)--(5.5,4);
\draw (5.5,4)--(5.5,3.5);
\draw (5.5,3.5)--(5,3.5);
\draw (5,3.5)--(5,3);
\draw (5,3)--(4.5,3);
\draw (4.5,3)--(4.5,2.5);
\draw (4.5,2.5)--(2,2.5);
\draw (2,2.5)--(2,10);
\node at (3.9,7){$\theta$};
\draw[dotted] (4.5,2.5)--(6,2.5);
\draw[dotted] (6,2.5)--(6,4);
\node at (5.4,3.1){$\bar\nu$};
\node at (4.1,10.7){$\mu_1$};
\draw[decorate,decoration={brace,amplitude=10pt,raise=4pt}](2,10)--(6,10);
\filldraw[fill=green!10!white!,draw=green!10!white!] (14,10) rectangle (16,9.5);
\filldraw[fill=green!10!white!,draw=green!10!white!] (14,9.5) rectangle (15.5,9);
\filldraw[fill=green!10!white!,draw=green!10!white!] (14,9) rectangle (15,8.5);
\filldraw[fill=green!10!white!,draw=green!10!white!] (10,3) rectangle (12.5,2.5);
\filldraw[fill=green!10!white!,draw=green!10!white!] (11.5,3.5) rectangle (13,3);
\filldraw[fill=green!10!white!,draw=green!10!white!] (12,4) rectangle (13.5,3.5);
\filldraw[fill=green!10!white!,draw=green!10!white!] (13,4.5) rectangle (14,4);
\filldraw[fill=green!10!white!,draw=green!10!white!] (13.5,5) rectangle (14,4.5);
\draw[pattern=north west lines, pattern color = black!20!white!, draw=white, draw opacity=0.2] (10,10) rectangle (14,5);
\draw[pattern=north west lines, pattern color = black!20!white!, draw=white, draw opacity=0.2] (10,5) rectangle (13.5, 4.5);
\draw[pattern=north west lines, pattern color = black!20!white!, draw=white, draw opacity=0.2] (10,4.5) rectangle (13,4);
\draw[pattern=north west lines, pattern color = black!20!white!, draw=white, draw opacity=0.2] (10,4) rectangle (12,3.5);
\draw[pattern=north west lines, pattern color = black!20!white!, draw=white, draw opacity=0.2] (10,3.5) rectangle (11.5,3);
\draw[pattern=north west lines, pattern color = black!20!white!, draw=black, draw opacity=0.2] (15,6) rectangle (15.5,5.5);
\node at (16.3,5.75){--$~ (n\backslash\mu$)};
\filldraw[fill=green!10!white!,draw=black] (15,7) rectangle (15.5,6.5);
\node at (16.5,6.75){--$~
\theta/(n\backslash\mu)$};
\draw (14,10)--(16,10);
\draw (16,10)--(16,9.5);
\draw (16,9.5)--(15.5,9.5);
\draw (15.5,9.5)--(15.5,9);
\draw (15.5,9)--(15,9);
\draw (15,9)--(15,8.5);
\draw (15,8.5)--(14,8.5);
\draw (14,8.5)-- (14,10);
\draw (14,5)--(14,4);
\draw (14,4)--(13.5,4);
\draw (13.5,4)--(13.5,3.5);
\draw (13.5,3.5)--(13,3.5);
\draw (13,3.5)--(13,3);
\draw (13,3)--(12.5,3);
\draw (12.5,3)--(12.5,2.5);
\draw (12.5,2.5)--(10,2.5);
\draw (10,2.5)--(10,3);
\node at (14.7,9.3){$\nu$};
\draw (14,5)--(13.5,5);
\draw (13.5,5)--(13.5,4.5);
\draw (13.5,4.5)--(13,4.5);
\draw (13,4.5)--(13,4);
\draw (13,4)--(12,4);
\draw (12,4)--(12,3.5);
\draw (12,3.5)--(11.5,3.5);
\draw (11.5,3.5)--(11.5,3);
\draw (11.5,3)--(10,3);
\draw[dotted](10,3)--(10,10);
\draw[dotted](10,10)--(14,10);
\draw[dotted](14,8.5)--(14,5);
\draw[dotted](12.5,2.5)--(14,2.5);
\draw[dotted](14,2.5)--(14,4);
\node at (12.5, 3.5){$\mu/\bar\nu$};  
\end{tikzpicture} 

$$
$$

It is easy to see that the diagram of $\theta/(n\backslash \mu)$ is actually a disjoint union of the diagram of $\nu$ and the diagram of $\mu/\bar\nu$ rotated $180$ degrees (because $n\ge \ell(\mu)+\ell(\nu)$). Therefore, the skew Schur polynomial $s_{\theta/(n\backslash \mu)}$ splits into the product of two other skew Schur polynomials:
$$
s_{\theta/(n\backslash \mu)} = s_\nu s_{\mu/\bar\nu}.
$$
Thus, we get that
$$
\dim\Hom_{GL_n}(V^\nn_\bn, V^\nn_\bl\otimes (V^\nn_\bm)^*) = (s_\nu s_{\mu/\bar\nu}, s_\l) = (s_{\mu/\bar\nu},s_{\l/\nu}).
$$
\blacksquare

\begin{definition}\label{deftheta}
We say that a sequence of partitions $\theta^\nn$ is \textbf{nice} if for every positive integer $K$ the set 
\begin{equation}
\{n\in\mathbb N|~\forall~i\le\ell(\theta):~\theta^\nn_i-\theta^\nn_{i+1}>K\} 
\end{equation}
 lies in $\cf$. We denote this set $U_K(\theta^\nn)$.
\end{definition}

\begin{definition}\label{defll}
Given some sequence of real numbers $a^\nn$, we write $a^\nn \ll n$ if for every positive integer $K$ the set $$\{n\in \mathbb N|~n-a^\nn> K\}$$
lies in $\cf$.
\end{definition}

\begin{definition}
We say that a sequence $a^\nn$ of real numbers is bounded if there exists $K>0$ such that the set 
$$\{n\in \mathbb N|~n-a^\nn< K\}$$
lies in $\cf$.
\end{definition}

\begin{definition}\label{dboldlambda}
Let us fix $k,l \in \mathbb Z_{\ge 0}, \gamma\in \mathcal P$. Let $\l^{(n)}$ be a sequence of partitions with $\ell(\l^\nn)\ll n$, such that $$\l^\nn =[\alpha^\nn,\beta^\nn,\gamma],$$
(in the sense of Definition \ref{defcutdiag}) with $\ell(\alpha^\nn)=k, \ell(\beta^\nn)= l$ for almost all $n$. Assume that the sequences of partitions $\alpha^\nn, \beta^\nn$ are \textbf{nice} (in the sense of Definition \ref{deftheta}).

Let us fix some $\mathbf {a}=(a_1,\ldots,a_k)\in\Z^k$ and $\mathbf b=(b_1,\ldots,b_l)\in\Z^l$. We denote $|\mathbf a|= \sum_{i=1}^k a_i$ and $|\mathbf b|=\sum_{j=1}^l b_j.$ Then for almost all $n$ and any $\delta\in \cp$
$$
\mu^\nn = [\alpha^\nn+\mathbf a,\beta^\nn +\mathbf b,\delta]
$$
is a well-defined partition. 

Define $\bl^\nn= (\l^\nn,\emptyset), \bm^\nn= (\mu^\nn,\emptyset)$.
\end{definition}

\begin{proposition}\label{cgupta} Let $\bl^\nn, \bm^\nn$ be as in Definition \ref{dboldlambda}.

Then for any $\bn = (\nu, \bar \nu)\in \mathcal P\times \mathcal P$ the multiplicity of $V_\bn^\nn$ in $\Hom_\ck(V^\nn_{\bm^\nn}, V^\nn_{\bl^\nn})$ is constant for almost all $n$.
\end{proposition}
\begin{proof}
By Theorem \ref{tgupta}, this multiplicity is equal to
$$
(s_{\l^\nn/\nu}, s_{\mu^\nn/\bar\nu}) = \sum_{\eta\subset \l^\nn,~\mu^\nn} c^{\l^\nn}_{\nu, \eta}c^{\mu^\nn}_{\bar\nu,\eta}.
$$
The sum is taken over $\eta$ inside both $\l^\nn$ and $\mu^\nn$ with 
$$
|\l^\nn|-|\eta|=|\nu|,
$$
$$
|\mu^\nn|-|\eta|=|\bar\nu|.
$$
Clearly, it is nonzero only for $\bn$ with $|\bar\nu|-|\nu| = |\mu^\nn|-|\l^\nn|=|\mathbf a|+|\mathbf b|+|\delta|-|\gamma|.$

Since $\mathbf a$ and $\mathbf b$ are constant, $\alpha^\nn+\mathbf a,\beta^\nn+\mathbf b$ are \textbf{nice} as well.  Let $\max(|\nu|,|\bar\nu|)= K$ us take $U = U_K(\alpha^\nn)\cap U_K(\beta^\nn)\cap U_K(\alpha^\nn + \mathbf a)\cap U_K(\beta^\nn+\mathbf b) $, where the sets $U_K(\cdot)$ are as in Definition \ref{deftheta}. Then, by definition of a \textbf{nice} sequence of partitions, $U\in \cf$. From now on assume $n\in U$.

As $\eta\subset \l^\nn, \mu^\nn$, we can find $k'\le k,l'\le l$, such that  for the corresponding decomposition  $\eta = [\sigma, \tau, \epsilon]$,  we have $\sigma\subset\alpha^\nn, \alpha^\nn+\mathbf a$, $\tau\subset \beta^\nn, \beta^\nn+\mathbf b$, and $\epsilon\subset \gamma,\delta$. Thus, for all $n\in U$ we have that $\l^\nn/\eta$ is the disjoint union of skew diagrams $\alpha^\nn/\sigma, (\beta^\nn)'/\tau'$ and $\gamma/\epsilon$ (with the similar statement for $\mu^\nn/\eta$). Moreover, $\alpha^\nn/\sigma$ is the disjoint union of (at most $k$) rows of some lengths $c_1,\ldots, c_k\in\Z_{\ge 0}$; and $(\beta^\nn)'/\tau'$ is the disjoint union of (at most $l$) columns of some lengths $d_1,\ldots, d_l\in \Z_{\ge 0}$. Let us denote $\mathbf c= (c_1,\ldots,c_k)\in \Z_{\ge 0}^k$ and $\mathbf d= (d_1,\ldots, d_l)\in \Z_{\ge 0}^l$. Then $(\alpha^\nn+\mathbf a)/\sigma$ is the disjoint union of rows of lengths $\mathbf c+\mathbf a$ and $(\beta^\nn+\mathbf b)'/\tau'$ is the disjoint union of columns with lengths $\mathbf d+\mathbf b$. Thus we must have
$$
\mathbf c\in(-\mathbf a+\Z^k_{\ge 0})\cap \Z^k_{\ge 0},
$$
$$
\mathbf d\in(-\mathbf b+\Z^l_{\ge 0})\cap \Z^l_{\ge 0}.
$$

Clearly, $\eta$ is uniquely defined by $\epsilon\subset\gamma,\delta$ and $\mathbf c,\mathbf d$ as above. 

Given some $\mathbf c\in \Z^k_{\ge 0}$ and $\mathbf d\in \Z^l_{\ge 0}$ and $\epsilon\subset \gamma\in \cp$ let us construct a skew diagram $\widetilde \l/\widetilde \eta(\mathbf c,\mathbf d,\gamma/\epsilon)$ in the following way:

We let $\widetilde\l(\mathbf c,\mathbf d,\gamma) = [\widetilde\alpha,\widetilde\beta, \gamma]$ with 
$$
\widetilde\alpha_i = \gamma_1 + \sum_{j=i}^{k} c_{j}, 
$$
$$
\widetilde\beta_i = \gamma_1'+\sum_{j=i}^l d_j.
$$

Let $\widetilde\eta(\mathbf c,\mathbf d,\epsilon) = [\widetilde\alpha - \mathbf c, \widetilde\beta - \mathbf d, \epsilon]$. Then we put
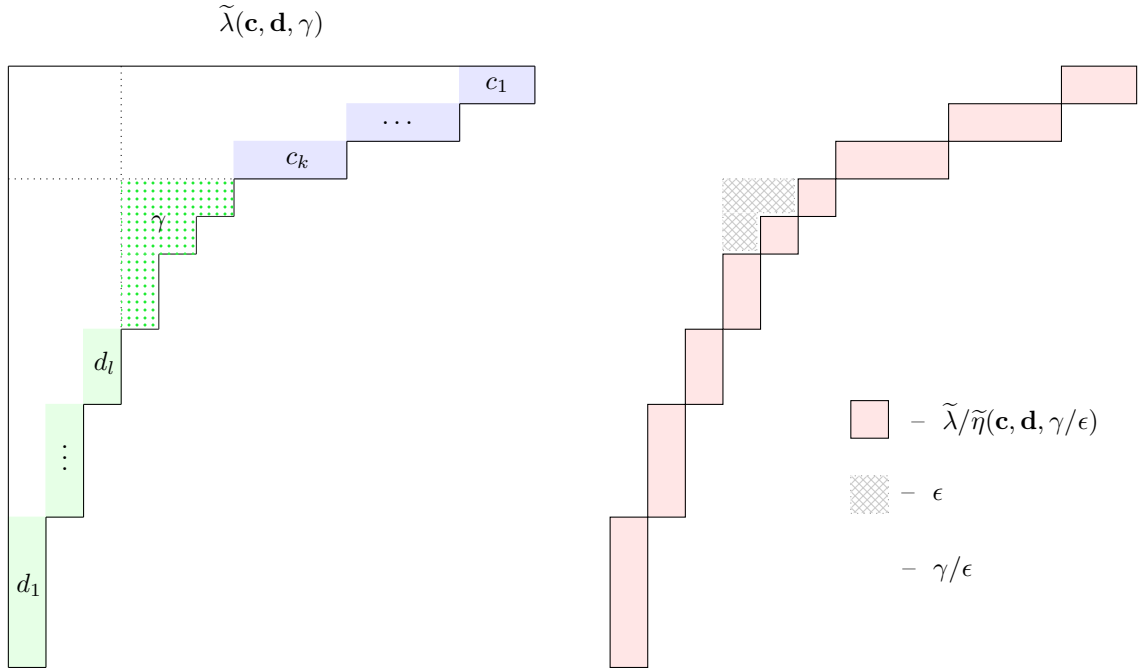
\begin{equation}\label{eqskew}
 \widetilde \l/\widetilde\eta(\mathbf c,\mathbf d,\gamma/\epsilon) = \widetilde\l(\mathbf c,\mathbf d,\gamma)/\widetilde\eta(\mathbf c,\mathbf d,\epsilon).   
\end{equation}

\begin{tikzpicture}

\filldraw[fill=blue!10!white!, draw=blue!10!white!] (6,8) rectangle (7,7.5);
\filldraw[fill=blue!10!white!, draw=blue!10!white!] (4.5,7.5) rectangle (6,7);
\filldraw[fill=blue!10!white!, draw=blue!10!white!] (3,7) rectangle (4.5,6.5);
\filldraw[fill=green!10!white!, draw=green!10!white!] (0,2) rectangle (0.5,0);
\filldraw[fill=green!10!white!, draw=green!10!white!] (0.5,3.5) rectangle (1,2);
\filldraw[fill=green!10!white!, draw=green!10!white!] (1,4.5) rectangle (1.5,3.5);
\draw (0,8)--(7,8);
\draw (7,8)--(7,7.5);
\draw (7,7.5)--(6,7.5);
\draw (6,7.5)--(6,7);
\draw (6,7)--(4.5,7);
\draw (4.5,7)--(4.5,6.5);
\draw (4.5,6.5)--(3,6.5);
\draw[dotted] (0,6.5)--(3,6.5);
\draw (3,6.5)--(3,6);
\draw (3,6)--(2.5,6);
\draw (2.5,6)--(2.5,5.5);
\draw (2.5,5.5)--(2,5.5);
\draw (2,5.5)--(2,4.5);
\draw (2,4.5)--(1.5,4.5);
\draw (1.5,4.5)--(1.5,3.5);
\draw (1.5,3.5)--(1,3.5);
\draw (1,3.5)--(1,2);
\draw (1,2)--(0.5,2);
\draw (0.5,2)--(0.5,0);
\draw (0.5,0)--(0,0);
\draw (0,0)--(0,8);
\draw[dotted] (1.5,8)--(1.5,4.5);
\draw[pattern = dots, pattern color = blue!40!white!20!green!, draw=white, draw opacity=0.2] (1.5,6.5) rectangle (3,6);
\draw[pattern = dots, pattern color = blue!40!white!20!green!, draw=white, draw opacity=0.2] (1.5,6) rectangle (2.5,5.5);
\draw[pattern = dots, pattern color = blue!40!white!20!green!, draw=white, draw opacity=0.2] (1.5,5.5) rectangle (2,4.5);
\node at (6.5,7.75){$c_1$};
\node at (5.2,7.25){$\dots$};
\node at (3.85, 6.75){$c_k$};
\node at (2,5.9){$\gamma$};
\node at (1.28,4.05){$d_l$};
\node at (0.75,2.9){$\vdots$};
\node at (0.28,1.1){$d_1$};
\node at (3.5,8.6){$\Large{\widetilde\l(\mathbf c,\mathbf d,\gamma)}$};
\filldraw[fill=red!10!white!, draw=black] (14,8) rectangle (15,7.5);
\filldraw[fill=red!10!white!, draw=black] (12.5,7.5) rectangle (14,7);
\filldraw[fill=red!10!white!, draw=black] (11,7) rectangle (12.5,6.5);
\filldraw[fill=red!10!white!, draw=black] (8,2) rectangle (8.5,0);
\filldraw[fill=red!10!white!, draw=black] (8.5,3.5) rectangle (9,2);
\filldraw[fill=red!10!white!, draw=black] (9,4.5) rectangle (9.5,3.5);
\filldraw[fill=red!10!white!, draw=black] (10.5,6.5) rectangle (11,6);
\filldraw[fill=red!10!white!, draw=black] (10,6) rectangle (10.5,5.5);
\filldraw[fill=red!10!white!, draw=black] (9.5,5.5) rectangle (10,4.5);
\draw[pattern = crosshatch,  pattern color = black!20!white!, draw = white] (9.5,6.5) rectangle (10.45,6.05);
\draw[pattern = crosshatch,  pattern color = black!20!white!, draw = white] (9.5,6.05) rectangle (9.95,5.55);
\draw[dotted][draw = black!40!white] (9.5,6.5)--(10.45,6.5);
\draw[dotted][draw = black!40!white] (9.5,6.5)--(9.5,5.55);
\draw[dotted][draw = black!40!white] (10.45,6.5)--(10.45,6.05);
\draw[dotted][draw = black!40!white] (9.5,5.55)--(9.95,5.55);
\draw[dotted][draw = black!40!white] (9.95,5.55)--(9.95,6.05);
\draw[dotted][draw = black!40!white] (9.95,6.05)--(10.45,6.05);
\draw[pattern = vertical lines,  pattern color = blue!40!white!20!green!, draw = blue!20!white!, draw opacity = 0] (10.52,6.48) rectangle (10.98,6.02);
\draw[pattern = vertical lines,  pattern color = blue!40!white!20!green!, draw = blue!20!white!, draw opacity = 0] (10.02,5.98) rectangle (10.48,5.52);
\draw[pattern = vertical lines,  pattern color = blue!40!white!20!green!, draw = blue!20!white!, draw opacity = 0] (9.52,4.52) rectangle (9.98,5.48);
\node at (13.25,3.3){-- $~ \Large{\widetilde\l/\widetilde\eta(\mathbf c,\mathbf d,\gamma/\epsilon)}$};
\filldraw[fill = red!10!white!, draw = black] (11.2,3.55) rectangle (11.7,3.05);
\draw[pattern=crosshatch, pattern color = black!20!white!, draw = white] (11.2,2.55) rectangle (11.7,2.05);
\node at (12.15,2.3){-- $~\Large{\epsilon}$};
\draw[dotted][draw = black!40!white](11.2,2.55) rectangle (11.7,2.05);
\draw[pattern = vertical lines,  pattern color = blue!40!white!20!green!, draw = blue!20!white!, draw opacity = 0] (11.2,1.55) rectangle (11.7,1.05);
\node at (12.35,1.3){-- $~\Large{\gamma/\epsilon}$};
\end{tikzpicture}
$$
$$
Now, $\l^\nn/\eta$ and $\widetilde\l/\widetilde\eta (\mathbf c,\mathbf d,\gamma/\epsilon)$ have the same disjoint parts in the same order. Thus, the number of Littlewood-Richardson tableaux of weight $\nu$ on them is the same. Let us call this number $c_\nu({\mathbf c,\mathbf d,\gamma/\epsilon})$. We have
$$
\sum_{\eta\subset \l^\nn,~\mu^\nn} c^{\l^\nn}_{\nu, \eta}c^{\mu^\nn}_{\bar\nu,\eta} = \sum_{\mathbf c,\mathbf d,\epsilon} c_\nu({\mathbf c,\mathbf d,\gamma/\epsilon}) c_{\bar \nu}({\mathbf c+\mathbf a,\mathbf d+\mathbf b,\delta/\epsilon}),
$$
where the sum is taken over 
$$
\mathbf c\in(-\mathbf a+\Z^k_{\ge 0})\cap \Z^k_{\ge 0},
$$
$$
\mathbf d\in(-\mathbf b+\Z^l_{\ge 0})\cap \Z^l_{\ge 0}.
$$
and $\epsilon \subset \gamma,\delta$ with
$$
|\mathbf c|+|\mathbf d|+|\gamma|-|\epsilon|=|\nu|,
$$
$$
|\mathbf c|+|\mathbf d|+|\mathbf a|+|\mathbf b|+|\delta|-|\epsilon| =|\bar\nu|.
$$
Thus, this sum depends only on $k,l,\mathbf a,\mathbf b, \gamma, \delta$.

\end{proof}

Note that in the proof of Proposition \ref{cgupta} not only have we proved that certain multiplicities do not depend on $n$, but also provided the algorithm for computing them. We would like to use this result further in the paper, which motivates the following definition.

\begin{definition}\label{defskew}
For fixed $\mathbf c\in \mathbb Z^k_{\ge 0}, \mathbf d\in\mathbb Z^l_{\ge 0}$ and partitions $\epsilon, \gamma, \nu$ with $\epsilon\subset \gamma$ define the constant $c_\nu(\mathbf c, \mathbf d, \gamma/\epsilon)$ to be the number of Littlewood-Richardson tableaux of weight $\nu$ on the skew-diagram $\widetilde \l/\widetilde\eta(\mathbf c, \mathbf d, \gamma/\epsilon)$ defined in \ref{eqskew} in the proof of Proposition \ref{cgupta}.
\end{definition}

\begin{corollary}\label{cskew}
Let $\bl^\nn, \bm^\nn$ be as in Definition \ref{dboldlambda}.

Then for any $\bn = (\nu, \bar \nu)\in \mathcal P\times \mathcal P$ the multiplicity of $V_\bn^\nn$ in $\Hom_\ck(V^\nn_{\bm^\nn}, V^\nn_{\bl^\nn})$ is constant for almost all $n$ and equal to
$$
\sum_{\eta\subset \l^\nn,~\mu^\nn} c^{\l^\nn}_{\nu, \eta}c^{\mu^\nn}_{\bar\nu,\eta} = \sum_{\mathbf c,\mathbf d,\epsilon} c_\nu({\mathbf c,\mathbf d,\gamma/\epsilon}) c_{\bar \nu}({\mathbf c+\mathbf a,\mathbf d+\mathbf b,\delta/\epsilon}),
$$
where the sum is taken over 
$$
\mathbf c\in(-\mathbf a+\Z^k_{\ge 0})\cap \Z^k_{\ge 0},
$$
$$
\mathbf d\in(-\mathbf b+\Z^l_{\ge 0})\cap \Z^l_{\ge 0}.
$$
and $\epsilon \subset \gamma,\delta$ with
$$
|\mathbf c|+|\mathbf d|+|\gamma|-|\epsilon|=|\nu|,
$$
$$
|\mathbf c|+|\mathbf d|+|\mathbf a|+|\mathbf b|+|\delta|-|\epsilon| =|\bar\nu|.
$$

\end{corollary}

\subsection{Constructing the bimodules $\underline{\mathrm{Hom}}(\bm, \bl)$}

Our goal now is to construct an object of $\Ind\cc_t$ out of $\Hom_\ck(V^\nn_{\bm^\nn}, V^\nn_{\bl^\nn})\in \cc_n$ with $\bl^\nn,\bm^\nn$ as in Definition \ref{dboldlambda}. For this purpose we would like to introduce some filtration on those bimodules and then take the filtered ultraproduct. Note that if a $(\mathfrak g, \mathfrak g)$ bimodule is obtained as a quotient of $U(\mathfrak g)$ by some two-sided ideal (that is,  it is spherical as in Definition \ref{dspherical}), then it naturally inherits the PBW-filtration from $U(\mathfrak g)$. So, if we let $\bl^\nn = \bm^\nn$, we will get a natural filtration $\widetilde F$ on $\End_\ck(V^\nn_{\bl^\nn})$ induced from the PBW-filtration on $U(\mathfrak{g}_n)$ (as the action map $U(\mathfrak g_n)\to \End_\ck(V^\nn_{\bl^\nn})$ is surjective, since $V^\nn_{\bl^\nn}$ is simple). 

\begin{definition}
\label{dend}

Let $\bl^\nn$ be as in Definition \ref{dboldlambda}.   

Let $\widetilde F^k \End_\ck(V^\nn_{\bl^\nn})$ be the image of $\widetilde F^k U(\mathfrak{g}_n)$ under the action map
$$
U(\mathfrak{g}_n)\to \End_\ck(V^\nn_{\bl^\nn}),
$$
where $\widetilde F$ is the PBW-filtration on $U(\mathfrak{g}_n)$.

Define $\bl$ to be a triple $(\alpha, \beta, \gamma)$, where $\alpha\in \C^k, \beta\in \C^l$ with
$$
\alpha_i = \prod_\cf \alpha_i^\nn,
$$
$$
\beta_j = \prod_\cf \beta_j^\nn,
$$
for all $i\le k, j\le l$.

We write
$$
\bl = \prod_\cf \bl^\nn.
$$

We define
$$
\underline \End(\bl)= \prod_\cf^{\widetilde F} \End_\ck (V^\nn_{\bl^\nn})
$$
\end{definition}

\begin{theorem}\label{tend} 
Let $\underline\End(\bl)$ be the filtered ultraproduct of bimodules as in Definition \ref{dend}. Then $\underline\End(\bl)$ lies in $\Ind(\cc_t)$ and is a spherical (in the sense of Definition \ref{dspherical})  Harish-Chandra bimodule of finite K-type with
$$
\dim \Hom_{\Ind\cc_t}(V_\bn, \underline \End(\bl)) =  \sum_{\mathbf c\in \Z_{\ge 0}^k} \sum_{\mathbf d\in\Z^l_{\ge 0}} \sum_{\epsilon\subset \gamma} c_\nu(\mathbf c,\mathbf d,\gamma/\epsilon)c_{\bar\nu}(\mathbf c,\mathbf d,\gamma/\epsilon). 
$$
\end{theorem}
\begin{proof}
By the remark about the PBW-filtration in Example \ref{expbw}, $\widetilde F^k \End_\ck (V^\nn_{\bl^\nn}) \in F^k\cc_n$, where $F$ is the filtration from Example \ref{exfiltr}. Moreover, by Corollary \ref{cskew}, for any $\nu\in\cp\times\cp$ and almost all $n$ the multiplicity 
$$
[\End_\ck(V^\nn_{\bl^\nn}): V^\nn_\nu] = \sum_{\mathbf c\in \Z_{\ge 0}^k} \sum_{\mathbf d\in\Z^l_{\ge 0}} \sum_{\epsilon\subset \gamma} c_\nu(\mathbf c,\mathbf d,\gamma/\epsilon)c_{\bar\nu}(\mathbf c,\mathbf d,\gamma/\epsilon) 
$$
doesn't depend on $n$. Thus, we can use Lemma \ref{lwelldefup} to show that $\underline \End(\bl)$ is a well-defined Harish-Chandra bimodule in $\Rep(GL_t)$. It follows from Lemma \ref{lequivup} that 
$$
\underline \End(\bl) = \prod_\cf^F \End_\ck(V^\nn_{\bl^\nn}).
$$
Hence, for any $\nu\in F^k(\cp\times\cp)$
$$
[\underline \End(\bl): V_\nu] = [F^k\underline \End(\bl): V_\nu] = [F^k\End_\ck(V_{\bl^\nn}):V^\nn_\nu]=[\End_\ck(V^\nn_{\bl^\nn}):V^\nn_\nu]
$$
for almost all $n$.
\end{proof}

Now, Lemma \ref{lequivup} shows that it doesn't really matter which filtration we choose. Using the filtration $F$ from Example \ref{exfiltr} we give the following definition:

\begin{definition}
\label{dhom}
Let $\bl^\nn, \bm^\nn,\mathbf a, \mathbf b$ be as in Definition \ref{dboldlambda}.

That is, to remind you, we put $\bl^\nn=(\l^\nn,\emptyset), \bm^\nn = (\mu^\nn,\emptyset)$.

We fix $k,l\in \mathbb Z_{\ge 0}, \gamma\in\mathcal P$ and assume  that $$\l^\nn =[\alpha^\nn,\beta^\nn,\gamma],$$
with $\ell(\alpha^\nn)=k, \ell(\beta^\nn)= l$ and $\ell(\l^\nn)\ll n$. We further assume that $\alpha^\nn$ and $\beta^\nn$ are \textbf{nice} sequences of partitions in the sense of Definition \ref{deftheta}.

Finally, fix some $\mathbf a=(a_1,\ldots,a_k)\in\Z^k$ and $\mathbf b=(b_1,\ldots,b_l)\in\Z^l$ and let
$$
\mu^\nn = [\alpha^\nn+\mathbf a,\beta^\nn +\mathbf b,\delta].
$$

Then we define 
$$
\underline \Hom(\bm, \bl) = \prod_\cf^F \Hom_\ck (V^\nn_{\bm^\nn}, V^\nn_{\bl^\nn}),
$$
where $\bl,\bm\in \C^k\times \C^l\times \cp $ are the triples
$$
\bl=(\alpha, \beta, \gamma),
$$
$$
\bm=(\alpha+\mathbf a, \beta+\mathbf b, \delta),
$$
with
$$
\alpha = \prod_\cf \alpha^\nn,
$$
$$
\beta = \prod_\cf \beta^\nn.
$$

\end{definition}

\begin{theorem}
\label{thom}
The bimodule $\underline \Hom(\bm,\bl)$ is well defined and has finite K-type with
$$
\dim\Hom_{\Ind\cc_t}(V_\nu, \underline \Hom(\bm,\bl)) = \sum_{\mathbf c,\mathbf d,\epsilon} c_\nu({\mathbf c,\mathbf d,\gamma/\epsilon}) c_{\bar \nu}({\mathbf c+\mathbf a,\mathbf d+\mathbf b,\delta/\epsilon}),
$$
where the sum is taken over 
$$
\mathbf c\in(-\mathbf a+\Z^k_{\ge 0})\cap \Z^k_{\ge 0},
$$
$$
\mathbf d\in(-\mathbf b+\Z^l_{\ge 0})\cap \Z^l_{\ge 0}.
$$
and $\epsilon \subset \gamma,\delta$ with
$$
|\mathbf c|+|\mathbf d|+|\gamma|-|\epsilon|=|\nu|,
$$
$$
|\mathbf c|+|\mathbf d|+|\mathbf a|+|\mathbf b|+|\delta|-|\epsilon| =|\bar\nu|.
$$ 
Moreover, we have $\underline \End(\bl) = \underline \Hom(\bl, \bl)$.
\end{theorem}
\begin{proof}
Follows from the Corollary \ref{cskew} and is identical to the proof of Theorem \ref{tend}.
\end{proof}

\begin{lemma}\label{rextdefinitionofhom}
Let $\bl\in \C^k\times\C^l\times\cp$ be some triple
$$
\bl=(\alpha,\beta, \gamma).
$$

Suppose $\rm{trdeg}_{\mathbb Q}{\mathbb {Q}}(\alpha_1,\ldots,\alpha_k,\beta_1,\ldots,\beta_l,t) = k+l+1$ (that is, $\{\alpha_i\}_{i\le k}\cup \{\beta_i\}_{i\le l} \cup \{t\}$  can be extended to a transcendence basis of $\mathbb C$ over $\bar{\mathbb Q}$). Then we can always find a collection of partitions $\l^\nn$ with $\ell(\l)\ll n$ and
$$
\l^\nn = [\alpha^\nn,\beta^\nn,\gamma]
$$
with
$$
\alpha_i=\prod_\cf \alpha_i^\nn,
$$
$$
\beta_i = \prod_\cf \beta^\nn_i.
$$

Moreover, we can choose $\alpha^\nn, \beta^\nn$ to be \textbf{nice}.
\end{lemma}

\begin{proof}
Suppose we found \textbf{nice} sequences of partitions $\widetilde\alpha^\nn, \widetilde \beta^\nn$, such that $\text{trdeg}_{\mathbb Q} \mathbb Q(\widetilde\alpha_1,\ldots, \widetilde\alpha_k, \widetilde\beta_1,\ldots,\widetilde\beta_l, t) = k + l + 1 $ for $\widetilde\alpha_i = \prod_\cf \alpha_i^{\nn}, \widetilde\beta_i = \prod_\cf \beta_i^\nn. $ Then there exists an automorphism of $\mathbb C$ over $\bar{\mathbb Q}$ that fixes $t$ and sends $\widetilde\alpha_i$ to $\alpha_i$ and $\widetilde \beta_i$ to $\beta_i$.

Now, to construct $\widetilde\alpha_i, \widetilde \beta_j$ let us define a set
$$
K = \{\phi\in\mathbb C|~\phi=\prod_\cf \phi^\nn \text{ with } \phi^\nn\in \mathbb N, \lim_n \phi^\nn = \infty, \text{ and } \phi^\nn < n/2l\}.
$$
Then the cardinality of $K$ is continuum. Thus, the transcendence degree of $\mathbb Q(\phi)_{\phi\in K}$ cannot be finite. So, we can find $\varepsilon_1, \ldots, \varepsilon_k, \delta_1, \ldots, \delta_l\in K $, such that $\rm{trdeg}_{\mathbb Q} \mathbb Q(\varepsilon_1,\ldots, \varepsilon_k, \delta_1,\ldots, \delta_l, t) = k + l +1$. 
We put $\widetilde \alpha_i - \widetilde \alpha_{i+1} = \varepsilon_i$ and $\widetilde\beta_i - \widetilde\beta_{i+1} = \delta_i$ (where we define $\widetilde\alpha_{k+1}=\widetilde\beta_{l+1}=0$).

The conditions on elements of $K$ imply that $\widetilde \alpha^\nn, \widetilde \beta^\nn$ are \textbf{nice} sequences of partitions (see Definition \ref{deftheta}). Moreover, as $\lim_n \widetilde\alpha_k^\nn = \lim_n \widetilde\beta_l^\nn = \infty$, $\l^\nn = [\widetilde\alpha^\nn, \widetilde\beta^\nn, \gamma]$ is a well-defined partition for any $\gamma$ and for almost all $n$. We have $n -\ell(\l^\nn) = n - (k + \widetilde\beta_1^\nn) = n -k - \delta_1 - \ldots - \delta_l > n - k - l(n/2l) = n/2 - k$. So, $\lim_{n} (n-\ell(\l^\nn)) = \infty$.
\end{proof}

\begin{corollary}
We can define the bimodules 
$$
\underline \Hom(\bm,\bl)
$$
for any $\bl$ as in Lemma \ref{rextdefinitionofhom} and
$$
\bm=(\alpha+\mathbf a,\beta+\mathbf b,\delta)
$$
for some $\mathbf a\in\Z^k, \mathbf b\in\Z^l$ and $\delta\in\cp$.

\end{corollary} 

\begin{hyp}\label{h}
Let $\bl^\nn = (\l^\nn,\bar{\l^\nn}), \bm^\nn = (\mu^\nn,\bar{\mu^\nn})\in \cp\times\cp$ be two sequences of bipartitions, such that both  $\{\l^\nn, \mu^\nn\}$ and $\{\bar{\l^\nn},\bar{\mu^\nn}\}$ satisfy the conditions of Proposition \ref{cgupta}. Then for any $\bn\in\cp\times\cp$ the multiplicity
$$
[V^\nn_{\bl^\nn}\otimes (V^\nn_{\bm^\nn})^*: V^\nn_\bn]
$$
is constant for almost all $n$, thus we can define the bimodule $\underline \Hom(\bm,\bl)$, for $\bl=(\l,\bar\l)$
\begin{align*}
    \l =(\alpha, \beta, \gamma),
    \bar\l =(\bar\alpha, \bar\beta, \bar\gamma)
\end{align*}
with $\alpha\in \C^k, \bar\alpha\in \C^{\bar k}, \beta\in\C^l, \bar\beta\in \C^{\bar l}, \gamma, \bar\gamma\in \cp$. For $\bm = (\mu, \bar\mu)$ with
\begin{align*}
    \mu =(\alpha+\mathbf a, \beta+\mathbf b, \delta),
    \bar\mu =(\bar\alpha+\bar {\mathbf a}, \bar\beta+ \bar {\mathbf b}, \bar\delta),
\end{align*}
where $\mathbf a,\bar{\mathbf a}, \mathbf b ,\bar {\mathbf b}$ are integer sequences and we place the same restriction on $\l,\bar \l$ as in Lemma \ref{rextdefinitionofhom}.
\end{hyp}

\begin{remark}
Defined as a filtered ultraproduct of simple bimodules, the bimodules $\underline{\Hom}(\bm,\bl)$ are clearly simple as well.
\end{remark}

\section{Exponential central characters}

In this section we will compute the central character with which $U(\mathfrak{g}_t)\otimes U(\mathfrak{g}_t)^{op}$ acts on the bimodules $\underline \Hom(\bm,\bl)$. 

\subsection{Definitions and notations}

\begin{definition}
An exponential central character $\chi(z)$ of $U(\mathfrak{g}_t)$ is the generating function 
$$\chi(z)= \frac{1}{e^{z}-1} \sum_{k\ge 0} \frac{\chi(C_k)}{k!}z^k,
$$ where $\chi:Z(U(\mathfrak{g}))\to \mathbbm{k}$ is some central character and $C_k$ are the generators of $Z(U(\mathfrak{g}_t))$ defined in Section \ref{suniv}. 
\end{definition}

\begin{remark}
We can similarly define the generating function for any central character $\chi^\nn$ of $U(\mathfrak{g}_n)$. To do this we need to say what $C_k$ for $k>n$ are. Using the isomorphism
$$
Z(U(\mathfrak{g}_n))\simeq \ck[\h_n^*]^{S_n},
$$
we define $C_k$ to be the element of $Z(U(\mathfrak{g}_n))$ acting on each simple module $L(\nu)$ with the constant
$$
\sum_{i=1}^n \left((\nu_i+\rho_i)^k-\rho_i^k\right).
$$
That is, we can just extend our definition of $C_k$ to all values of $k$.
\end{remark}

\begin{ex}\label{exhcnonzero}
Theorem 3.17 in \cite{U} says that the category $\hc_{\chi,\psi}(\mathfrak{g}_t)$ is non-zero if and only if there exist some complex numbers $b_1,\ldots, b_r$ and $c_1,\cdots, c_s$, such that
$$
\chi(z)-\psi(z) = \sum_{i=1}^r e^{b_iz} - \sum_{j=1}^s e^{c_j z}.
$$
\end{ex}

\subsection{Computation of central characters}

\begin{claim}\label{cl1}
Let $\chi_\bn$ be the central character of $V_\bn\in \cc_t$ for $\bn = (\nu, \bar\nu)\in\cp\times\cp$. We have
$$
\chi_\bn(z) = \frac{1}{e^z-1} \left(\sum_{j=1}^{\ell(\nu)} (e^{\nu_jz} -1) e^{\frac{t+1}{2}z-jz}+\sum_{j=1}^{\ell(\bar\nu)} (e^{-\bar\nu_j z} -1)e^{-\frac{t+1}{2}z+jz}\right).
$$
\end{claim}
\begin{proof}
Let us denote by $\chi_\bn^{(n)}$ the central character with which $U(\mathfrak{g}_n)$ acts on the irreducible module $V_{\bn}^\nn$. Clearly, we have $\chi_\bn = \prod_\cf \chi_\bn^\nn$. Moreover, for big enough $n$
$$\chi_\bn^\nn(C_k) = \sum_{j=1}^{\ell(\nu)}\left( (\nu_j+\rho_j)^k - \rho_j^k\right) + \sum_{j=1}^{\ell(\bar\nu)} \left((-\bar\nu_j +\rho_{n-j+1})^k - \rho_{n-j+1}^k\right).$$

For $GL_n$ we have
$$
\rho_i = \frac{n+1}{2}-i.
$$
In particular,
$$
\rho_{n-i+1} = \frac{n+1}{2} - (n+1) + i = -\frac{n+1}{2} + i.
$$

Now, since $\prod_\cf n = t$, we have
\begin{align*}
 \chi_\bn(C_k)=\prod_\cf \chi^\nn_\bn(C_k) = \sum_{j=1}^{\ell(\nu)}\left( (\nu_j+\tfrac{t+1}{2} - j)^k - (\tfrac{t+1}{2}-j)^k\right) +\\+ \sum_{j=1}^{\ell(\bar\nu)}\left( (-\bar\nu_j -\tfrac{t+1}{2}+j)^k - (-\tfrac{t+1}{2}+j)^k\right).   
\end{align*}

The result follows.
\end{proof}

\begin{remark}
Let $\chi_\nu$ be the central character of the ``polynomial" module $V_\nu$ for $\nu\in \cp$. That is, if $\bn=(\nu, \emptyset)$, we will simply write $\nu$ instead of $\bn$.

The computation in Claim \ref{cl1} showed that if $\bn=(\nu,\bar\nu)$, then

$$
\chi_\bn(z) = \chi_\nu(z) - e^{-z}\chi_{\bar\nu}(-z).
$$
Similarly,
$$
\chi^\nn_\bn(z) = \chi^\nn_\nu(z) - e^{-z} \chi^\nn_{\bar\nu}(-z).
$$
\end{remark}

Let $\l = [\alpha,\beta, \gamma]$ be some partition. We want to express $\chi_\l$ in terms of $\alpha_i,\beta_i,\gamma_i$.

Let us introduce another notation to simplify our computation.

\begin{definition}\label{dqnum}
We put
$$
q= e^z.
$$
We will also make use of the $q$-numbers:
$$
[k]_q= \frac{q^k-1}{q-1}.
$$
They satisfy the following property: if $k>l$ 

\begin{equation}
[k-l]_q = [k]_q-[l]_q q^{k-l}.
\end{equation}
\end{definition}

We will write $\chi(\log(q))$ for the generating function of $\chi$ expressed in terms of $q$.

\begin{ex}\label{exchilam}
For $\l\in\cp$ we have 
$$
\chi_\l(\log(q))= \frac{q^{\frac{t+1}{2}}}{q-1} \sum_{j=1}^{\ell(\l)} (q^{\l_j}-1)q^{-j} = q^{\frac{t+1}{2}}\sum_{j=1}^{\ell(\l)} [\l_j]_q q^{-j}.
$$
\end{ex}

\begin{lemma}\label{ltransposeddiagram}
Let $\l\in\cp$. Then
$$
\sum_{j=1}^{\ell(\l)} [\l_j]_q q^{-j} = \sum_{k=1}^{\ell(\l')}[\l'_k]_q q^{-\l_k'+k-1}.
$$
\end{lemma}
\begin{proof}
Let $J_k = \{j|~\l_j =k\}$. As $\{1,\ldots, \lambda'_k\} = \{j|~\l_j\ge k\}$, we have
$$
J_k=\{\l'_{k+1}+1,\ldots, \l'_k\}.
$$
Now,
$$
\sum_{j=1}^{\ell(\l)} [\l_j]_qq^{-j}=\sum_{k=1}^{\l_1} \sum_{j\in J_k} [k]_q  q^{-j} = \sum_{k=1}^{\l_1} [k]_q(q^{-\l_{k+1}'-1}+\ldots+q^{-\l'_k}) = \sum_{k=1}^{\ell(\l')} [k]_q [\l'_{k}-\l'_{k+1}]_q q^{-\l_k'}.
$$
By $(\ref{dqnum}.1)$, the above is equal to 
$$
\sum_{k=1}^{\ell(\l')} [k]_q([\l_k']_q q^{-\l_k'}-[\l_{k+1}']_q q^{-\l'_{k+1}}) = \sum_{k=1}^{\ell(\l')} [\l_k']_q q^{-\l'_k}([k]_q  - [k-1]_q) = \sum_{k=1}^{\ell(\l')} [\l_k']_q q^{-\l'_k+k-1}.
$$
\end{proof}

\begin{claim}\label{cl2}
Let $\l=[\alpha, \beta, \gamma]$ (in the sense of Definition \ref{defcutdiag}) with $\ell(\alpha) = k, \ell(\beta) = l, \ell(\gamma)=m$. Then
$$
\chi_\l(\log(q)) = q^{\frac{t+1}{2}}( \sum_{j=1}^k [\alpha_j+l]_q q^{-j} + \sum_{j=1}^l [\beta_j-m]_q q^{-\beta_j+j-1-k} + \sum_{j=1}^{m} [\gamma_j+l]_q q^{-k-j}).
$$
\end{claim}
\begin{proof}
We have (by Example \ref{exchilam})
$$
\chi_\l(\log(q))q^{-\frac{t+2}{2}} = \sum_{j=1}^{\ell(\l)} [\l_j]_q q^{-j} = \sum_{j=1}^k [\alpha_j+l]_q q^{-j} + \sum_{j=k+1}^{\ell(\l)}[\l_j]_q q^{-j},
$$
by the definition of $\alpha$. Next, by the definition of $\gamma$, the above is equal to
$$
\sum_{j=1}^k [\alpha_j+l]_q q^{-j} +  q^{-k}(\sum_{j=1}^m [\gamma_j+l]_q q^{-j} + \sum_{j=m+1}^{\ell(\l)-k} [\l_{j+k}]_q q^{-j}).
$$
Lastly, let $\mu$ be the partition obtained from $\l$ by removing the first $k+m$ rows. That is,
$$
\mu_j = \l_{j+k+m}.
$$
Then, by Lemma \ref{ltransposeddiagram} and the definition of $\beta$,
$$
\chi_\l(\log(q))q^{-\frac{t+1}{2}} - \sum_{j=1}^k [\alpha_j+l]_q q^{-j} -  \sum_{j=1}^m [\gamma_j+l]_q q^{-j-k} = q^{-k-m}\sum_{j=1}^{\ell(\l)-k-m} [\l_{j+k+m}]_q q^{-j} =
$$
$$ =  q^{-k-m}\sum_{j=1}^{\ell(\l)-k-m} [\mu_j]_q q^{-j} = 
 q^{-k-m} \sum_{j=1}^l [\mu_j']_q q^{-\mu_j'+j-1} = 
$$
$$
= q^{-k-m} \sum_{j=1}^l [\beta_j - m]_q q^{-\beta_j + m + j -1} = 
 \sum_{j=1}^l [\beta_j - m]_q q^{-\beta_j  + j -1-k}.
$$
\end{proof}

\begin{theorem}
Let $\bl,\bm \in \C^k\times \C^l\times \cp$ be some triples
$$
\bl= (\alpha,\beta,\gamma),
$$
$$
\bm = (\alpha+a, \beta+b,\delta),
$$
 with $\mathbf a\in\Z^k, \mathbf b\in\Z^l$, such that $\rm{trdeg}_{\mathbb Q} \Q(\alpha_1,\ldots, \alpha_k,\beta_1,\ldots,\beta_l, t) = k + l + 1$. 

Then
$$
\underline \Hom(\bm,\bl)\in \hc_{\chi,\psi}(\mathfrak{g}_t),
$$
where
$$
\chi(\log(q)) = q^{\frac{t+1}{2}}( \sum_{j=1}^k [\alpha_j+l]_q q^{-j} + \sum_{j=1}^l [\beta_j-m]_q q^{-\beta_j+j-1-k} + \sum_{j=1}^{m} [\gamma_j+l]_q q^{-k-j})
$$
and
$$
\psi(\log(q)) = q^{\frac{t+1}{2}}( \sum_{j=1}^k [\alpha_j+a_j+l]_q q^{-j} + \sum_{j=1}^l [\beta_j-m+b_j]_q q^{-\beta_j-b_j+j-1-k} + \sum_{j=1}^{m} [\delta_j+l]_q q^{-k-j}).
$$
\end{theorem}
\begin{proof}
By Lemma \ref{rextdefinitionofhom}, we can find two sequences of partitions $\l^\nn, \mu^\nn$, so that
$$
\underline \Hom(\bm, \bl) = \prod_\cf^F \Hom_\ck(V^\nn_{\mu^\nn}, V^\nn_{\l^\nn}).
$$

For each $n$
$$
\Hom_\ck(V^\nn_{\mu^\nn}, V^\nn_{\l^\nn})\in \hc_{\chi^\nn,\psi^\nn},
$$
where
$$
\chi^\nn = \chi_{\l^\nn},
$$
$$
\psi^\nn= \chi_{\mu^\nn}.
$$
The proof of Claim \ref{cl2} works also for central characters of $U(\mathfrak{g}_n)$. So, we know that
$$
\chi^\nn(\log(q)) = q^{\frac{n+1}{2}}( \sum_{j=1}^k [\alpha^\nn_j+l]_q q^{-j} + \sum_{j=1}^l [\beta^\nn_j-m]_q q^{-\beta^\nn_j+j-1-k} + \sum_{j=1}^{m} [\gamma_j+l]_q q^{-k-j})
$$
and
$$
\psi^\nn(\log(q))= q^{\frac{n+1}{2}}( \sum_{j=1}^k [\alpha^\nn_j+a_j+l]_q q^{-j} + \sum_{j=1}^l [\beta^\nn_j-m+b_j]_q q^{-\beta^\nn_j-b_j+j-1-k} + \sum_{j=1}^{m} [\delta_j+l]_q q^{-k-j}).
$$
We get the desired result by taking the ultraproduct.
\end{proof}

\section{Spherical bimodules}
\label{sspher}

In this section we will provide another approach to constructing the bimodules $\underline \End(\bl)$ by looking at the two-sided ideals in $U_\chi$.

\begin{definition}\label{dspherical}
A bimodule $M\in \hc(\mathfrak{g}_t)$ is called spherical if it contains a subobject isomorphic to $\one\in\cc_t$, which generates $M$ as a $(\mathfrak{g}_t,\mathfrak{g}_t)$-bimodule.
\end{definition}

Proposition $3.22.$ in \cite{E} says that any irreducible spherical bimodule is a quotient of $U_\chi$ for some $\chi$. Thus, the description of all irreducible spherical bimodules is equivalent to the description of all maximal two-sided ideals in $U_\chi$. 

It was proved in \cite{E}, Subsection 3.5, that for a generic $\chi$ the algebra $U_\chi$ is simple. However, we have seen that $\underline \End(\bl)$ is a nontrivial quotient of $U(\mathfrak{g})$. Thus, central characters $\chi$ as in Claim \ref{cl2} provide us with an example of non-simple algebras $U_\chi$. 

Our aim is to take the filtered ultraproduct of the annihilators $\mathrm{Ann}(M_n)$ of some $\mathfrak{g}_n$-modules to obtain two-sided ideals in $U(\mathfrak{g}_t)$. However, the problem one immediately encounters when testing this idea is that the ideals $\mathrm{Ann}(M_n)$ could be generated  by elements of degrees that depend on $n$. Hence, in most cases the filtered ultraproduct will turn out to be zero. We would like to understand an upper bound for the degree $k$, such that $F^k\mathrm{Ann}(M_n)$ is nonzero for a given $\mathfrak{g}_n$-module $M$ ($F$ being the PBW-filtration). Unfortunately, this is unknown in general.

However, if $M = V^\nn_\bn$ is a finite-dimensional $\mathfrak{g}_n$-module, one can easily provide such a bound, which we will do below.

\subsection{Annihilators of finite-dimensional representations}

Let us fix the standard basis $\{E_{ij}\}_{i,j\le n}$ of $\gl_n$.

\begin{lemma}\label{lannsim}
For any $m,n\in\N$ and any $i,j,k,l\le n$ the elements
$$
E_{ij}(E_{kl}+\delta_{kl})-E_{il}(E_{kj}+\delta_{kj})
$$
lie in $\mathrm{Ann}(S^mV^\nn)\subset U(\gl_n)$ and the elements
$$
(E_{ij}-\delta_{ij})E_{kl}-(E_{il}-\delta_{il})E_{kj}
$$
lie in $\mathrm{Ann}(S^m(V^\nn)^*)$.
\end{lemma}
\begin{proof}
It follows from the fact that $E_{ij}$ acts as $x_i\frac{\partial}{\partial x_j}$ on $S^mV^\nn$ considered as the space of polynomial functions on $(V^\nn)^*$ of degree $m$. Similarly, $E_{ij}$ acts as $-x^*_j\frac{\partial}{\partial x^*_i}$ on $S^m(V^\nn)^*$.
\end{proof}
\begin{remark}
If all indices $i,j,k,l$ are pairwise distinct for both $S^mV^\nn$ and $S^m(V^\nn)^*$, we get just the corresponding 2-by-2 minor of the matrix $A=(E_{ij})_{ij}\subset U(\gl_n)\otimes Mat_n$.
\end{remark}

\begin{corollary}\label{cann}
For any $\l_1,\ldots,\l_k, \mu_1,\ldots, \mu_l \in\N$ and for any $I=\{i_0,\ldots, i_{p}\}, J=\{j_0,\ldots,j_{p}\}$ sets of $p+1$ pairwise distinct indices, where $p=k+l$, the $(p+1)$-by-$(p+1)$ minor $A_{I,J}$ of the matrix $A=(E_{ij})_{ij}\in U(\gl_n)\otimes Mat_n$ annihilates 
$$
W_{\l,\mu}:=S^{\l_1}V\otimes\ldots\otimes S^{\l_k}V\otimes S^{\mu_1}V^*\otimes\ldots \otimes S^{\mu_l}V^*.
$$
\end{corollary}
\begin{proof}
$U(\gl_n)^{\otimes p}$ acts naturally on $W_{\l,\mu}$ and the action of $E_{ij}\in U(\gl_n)$ on $W_{\l,\mu}$ coincides with the action of $\Delta^{p-1}E_{ij}\subset (U(\gl_n))^{\otimes p}$, where 
$$
\Delta E_{ij}= E_{ij}\otimes 1+1\otimes E_{ij},
$$
i.e. $\Delta$ is the comultiplication on $U(\gl_n)$. 

For $1\le m\le p$ let $E_{ij}^{(m)}\subset (U(\gl_n))^{\otimes p}$ be the elements
$$
E_{ij}^{(m)} := \underbrace{1\otimes\ldots \otimes 1}_{m-1} \otimes E_{ij} \otimes 1 \otimes \ldots \otimes 1.
$$
Then
$$\Delta^{p-1}(E_{ij}) = \sum_{m=1}^{p} E_{ij}^{(m)}. $$

We note that $E_{ij}^{(a)}$ and $E_{kl}^{(b)}$ commute for all $i,j,k,l\le n$ and distinct $a$ and $b$ in $\{1,\ldots, p\}$.

Let $A^{(m)} = (E_{ij}^{(m)})_{ij} \in (U(\gl_n))^{\otimes p}\otimes Mat_n = (U(\gl_n))^{\otimes p}\otimes \End_\ck(V^\nn)$. Then the action of the minor $A_{I,J}$ on $W_{\l,\mu}$ coincides with the action of $(A^{(1)}+\ldots+A^{(p)})_{I,J}$.

Let $B:= \sum_{m=1}^p A^{(m)}$. Then $\Lambda^{p+1} B\in (U(\gl_n))^{\otimes p}\otimes \End_\ck(\Lambda^{p+1} V^\nn)$ is well-defined and acts naturally on $$(U(\gl_n))^{\otimes p}\otimes \Lambda^{p+1} V^\nn$$ (with $(U(\gl_n))^{\otimes p}$ acting on itself via the left multiplication). We have $$B_{I,J} = \langle v_{i_0}\wedge \ldots \wedge v_{i_p}| \Lambda^{p+1}B(v_{j_0}\wedge \ldots\wedge v_{j_p})\rangle, $$
where, given some algebra $R$ and some basis $f_i$ in a free $R$-module $M$, for any $m\in M$ we denote by  $\langle f_i|~m\rangle$ the coefficient of $f_i$ in the basis decomposition of $m$.

Now 
$$
\Lambda^{p+1}B(v_{j_0}\wedge \ldots\wedge v_{j_p}) = \sum_{1\le m_0,\ldots,m_p\le p} A^{(m_0)}(v_{j_0})\wedge\ldots\wedge A^{(m_p)}(v_{j_p}).
$$

In each summand at least one of the matrices $A^{(m)}$ occurs twice. Since matrix elements of distinct $A^{(m)}$ commute, we can move these two to the left. So, up to permutation of indices, each summand is equal to
$$
A^{(m)}v_{j_1}\wedge A^{(m)}v_{j_2}\wedge\ldots = \sum A^{(m)}_{\{k_1,k_2\},\{j_1,j_2\}}v_{k_1}\wedge v_{k_2}\wedge \ldots
$$

Since we only care about the summands with $k_1,k_2\in I$ and $k_1\neq k_2$, we get that by the previous lemma $$A^{(m)}_{\{k_1,k_2\},\{j_1,j_2\}} = E_{k_1 j_1}^{(m)}E_{k_2 j_2}^{(m)} - E_{k_1 j_2}^{(m)}E_{k_2 j_1}^{(m)} $$ acts trivially on $W_{\l,\mu}$. Hence, $A_{I,J}\in \mathrm{Ann}(W_{\l,\mu})$.

\end{proof}

Now let us take some $\bl=(\l,\bar\l)\in \cp\times\cp$. Let $d = d(\l), \bar d = d(\bar\l)$, where $d(\nu)$ denotes the length of the main diagonal of the Young diagram of $\nu$. For $1\le i\le d$ put 
$$
\mu_i= \l_i - d,
$$
$$
\nu_i= \l'_i. 
$$

Similarly define $\bar\mu,\bar\nu.$

Then $V_\bl$ is a submodule in
$$
\bigotimes_{i=1}^{d} S^{\mu_i}V^\nn\otimes  \bigotimes_{i=1}^{d} \Lambda^{\nu_i}V^\nn \otimes \bigotimes_{i=1}^{\bar d} S^{\bar\mu_i}(V^\nn)^*\otimes  \bigotimes_{i=1}^{\bar d} \Lambda^{\bar\nu_i}(V^\nn)^*.
$$

Put $k = \max(d,\bar d)$. By the discussion above $V_\bl$ is a submodule in
$$
R_k= (S^{\bullet}V^\nn)^{\otimes k}\otimes(\Lambda^{\bullet}V^\nn)^{\otimes k}\otimes (S^\bullet (V^\nn)^*)^{\otimes k}\otimes(\Lambda^\bullet (V^\nn)^*)^{\otimes k}.
$$

\begin{lemma}\label{lann} If $R_k$ is as above, then
$$F^{(2k+1)(2k(2k+1)+1)}\mathrm{Ann}(R_{k}) \neq 0.$$
\end{lemma}

\begin{proof}

We have $R_{k}=S^\bullet((V^\nn)^{\oplus k}\oplus (V^\nn)^{*\oplus k})\otimes \Lambda^\bullet((V^\nn)^{\oplus k}\oplus (V^\nn)^{*\oplus k})$. Then $E_{ij}$ acts on the symmetric algebra  as 
$$
\sum_{a=1}^{k} x_{ia}\frac{\partial}{\partial x_{ja}} - \sum_{b=k+1}^{2k} x_{jb}\frac{\partial}{\partial x_{ib}},
$$
where 
$$
S^\bullet (V^{\nn\oplus k}) = \ck[x_{ia}]_{i\le n, a\le k},
$$
$$
S^\bullet((V^\nn)^{*\oplus k})= \ck[x_{ib}]_{i\le n, k< b\le 2k}.
$$
and $x_{ia},x_{ib}$ are even variables.

Similarly, it acts on the exterior algebra as 
$$
\sum_{a=1}^{k} \xi_{ia}\frac{\partial}{\partial \xi_{ja}} - \sum_{b=k+1}^{2k}  \xi_{jb}\frac{\partial}{\partial \xi_{ib}},
$$
where $\xi_{ia},\xi_{ib}$ for $i\le n, a\le k, k<b\le 2k$ are odd variables. 

Let us fix $I=\{i_0,\ldots,i_{2k}\},J=\{j_0,\ldots,j_{2k}\}$: some sets of $2k+1$ pairwise distinct indices with $I\cap J=\emptyset$. For distinct $i,j$ multiplication by $x_{ia}$ (or $\xi_{ia}$) commutes with differentiation $\frac{\partial}{\partial x_{ja}}$ (or $\frac{\partial}{\partial \xi_{ja}}$). Let us denote $$y_{ja}= \frac{\partial}{\partial x_{ja}},$$ $$\eta_{ja}= \frac{\partial}{\partial \xi_{ja}}.$$

Then the representation of the subalgebra $U_{I,J}$ of $U(\mathfrak{g})$ generated by $E_{ij}$ with $i\in I$ and $j\in J$ on $R_{k}$ is given by a homomorphism
$$
\phi:U_{I,J} \to \mathbbm{k}[x_{ia},x_{jb},y_{ib},y_{ja},\xi_{ia},\xi_{jb},\eta_{ib},\eta_{ja}],
$$
$$
E_{ij}\mapsto \sum_{a=1}^{k} x_{ia} y_{ja} - \sum_{b=k+1}^{2k} x_{jb}y_{ib} + \sum_{a=1}^{k} \xi_{ia}\eta_{ja} - \sum_{b=k+1}^{2k}  \xi_{jb}\eta_{ib},
$$
where the algebra on the right-hand-side is a free supercommutative algebra with even generators $x_{ia},x_{jb},y_{ib},y_{ja}$  and odd generators $\xi_{ia},\xi_{jb},\eta_{ib},\eta_{ja}$, and the indices run over the following sets:$$i\in I,j\in J, a\in \{1,\ldots,k\}, b\in\{k+1,\ldots,2k\}.$$

By Corollary \ref{cann}, $A_{I,J}$ (the corresponding $(2k+1)$ by $(2k+1)$ minor) annihilates $S^\cdot(V^{\nn\oplus k}\oplus (V^\nn)^{*\oplus k})$. So, $\phi(A_{I,J})$ is a differential operator that acts by zero on all polynomials in variables $x_{l,a}, x_{l,b}$ for $l\le n,$ and $a\le k, k<b\le 2k$. Therefore, $\phi(A_{I,J})$ lies in the nilradical of our algebra (i.e. it lies in the ideal generated by variables $\xi,\eta$). Moreover, each nonzero monomial in $\phi(A_{I,J})$ must contain equal number of variables $\xi$ and $\eta$, because it is true for $\phi(E_{ij})$. Since the number of odd variables in our algebra is equal to $4k(2k+1)$, we have $\phi(A^{2k(2k+1)+1}_{I,J})=0$.

Finally, since $\deg A_{I,J} = 2k+1$ and $A_{I,J}^{2k(2k+1)+1}\in \mathrm{Ann}(R_k)$ we get that
$$
F^{(2k+1)(2k(2k+1)+1)}\mathrm{Ann} (R_k)\neq 0.
$$

\end{proof} 

\subsection{The bimodules $\underline{\mathrm{Ann}}(\bl)$}

\begin{corollary}
If $\bl^\nn=(\l^\nn,\bar{\l^\nn})$ is a sequence of bipartitions, such that $d(\l^\nn), d(\bar{\l^\nn})$ are bounded, then 
there exists $N\in\N$ with
$$
F^N\mathrm{Ann}(V^\nn_{\bl^\nn})\neq 0
$$
for all $n$.

The filtered ultraproduct
$$
\prod^F_\cf \mathrm{Ann}(V^\nn_{\bl^\nn})
$$
is nonzero and lies in $\Ind\cc_t$.
\end{corollary}
\begin{proof}
The first part of the statement is just a consequence of Lemma \ref{lann}: we can take $k = \max(d(\l^\nn), d(\bar\l^\nn))$ and 
$$
N  = (2k+1)(2k(2k+1)+1).
$$

Now, since for any $l$
$$
F^l \mathrm{Ann}(V^\nn_{\bl^\nn})\subset F^l U(\mathfrak{g}_n)
$$
and, consequently, bimodules $(\mathrm{Ann}(V^\nn_{\bl^\nn}), F)$ satisfy the conditions of Lemma \ref{lwelldefup}, we have
$$
\underline{\mathrm{Ann}}(\bl)= \prod^F_\cf \mathrm{Ann}(V^\nn_{\bl^\nn})\in \Ind\cc_t.
$$

Moreover, $F^N\underline{\mathrm{Ann}}(\bl) = \prod_\cf F^N\mathrm{Ann}(V_{\bl^\nn}) \neq 0.$
\end{proof}

\begin{remark}
Let $\bl^\nn$ be as in Definition \ref{dend} and $\bl = \prod_\cf \bl^\nn$.  Then
$$
\underline \End(\bl) = U(\mathfrak{g}_t)/\underline{\mathrm{Ann}}(\bl),
$$
where 
$$
\underline{\mathrm{Ann}}(\bl) = \prod_\cf^F \mathrm{Ann}(V_{\bl^\nn}).
$$
\end{remark}

\begin{remark}\label{rdiag}
Suppose $\l^\nn$ is a sequence of partitions, such that for all $n$
\begin{itemize}
    \item $d(\l^\nn)<D$ for some $D\in\Z_{\ge 0}$,
    \item $\ell(\l^\nn)\ll n$.
\end{itemize}

 By the properties of the ultrafilter, we have that $d(\l^\nn)$ is constant for almost all $n$. Denote this constant by $d$.

Then for every $i> d$ we have that $\l^\nn_i, (\l^\nn)'_i \le d$ for almost all $n$. Thus, they are constant for almost all $n$.

Now let $k$ be the maximal integer for which $\l^\nn_k$ is not bounded for almost all $n$, or, in other words, the set of $\{n\in\mathbb N|~\l^\nn_k > K\}$ is in $\cf$ for every $K$ (note the similarity with the condition in Definition \ref{deftheta}). By the above, we have $k\le d$. Note that since $\l^\nn$ are partitions, $\l^\nn_i\ge \l^\nn_k$ for every $i\le k$, so if $\l_k^\nn$ are unbounded, then $\l^\nn_i$ are too. We define the constant $l$ in the same way, but for partitions $(\l^\nn)'$.

Let us carry out the ``cutting" procedure for partitions $\l^\nn$ and $k,l$ as above (see Figure \ref{fig1}). We have $\l^\nn = [\alpha^\nn, \beta^\nn, \gamma^\nn]$. Now since $\l_i^\nn$ are constant for almost all $n$ for $i>k$, we have that $\gamma_i = \l^\nn_{i+k} - l$ are constant as well. So, there exists a partition $\gamma$, such that $\l^\nn = [\alpha^\nn, \beta^\nn, \gamma]$ for almost all $n$.

Then, by our definitions of $k$ and $l$, for the sequences of partitions $\alpha^\nn$ and $\beta^\nn$ we have for every $K>0$ that
$$
\{n\in \mathbb N|~\forall~i\le k:~\alpha_i^\nn>K\},\{n\in \mathbb N|~\forall~i\le l:~\beta_i^\nn>K\}\in\cf.
$$

It differs from the conditions we put on $\l^\nn$ in Definition \ref{dboldlambda}. In this, more general case, for a fixed $i\le k$ we can either have that $\{n\in\mathbb N|~ \alpha_i^\nn - \alpha_{i+1}^\nn>K\}\in\cf $ for all $K>0$ or  that $\alpha_i^\nn - \alpha_{i+1}^\nn$ is constant for almost all $n$. And similarly for $\beta^\nn_i- \beta^\nn_{i+1}$ for $i\le l$. Thus, we can replace the data of $\alpha^\nn, \beta^\nn$ with a \textbf{nice} sequence of partitions $A^\nn=(A^\nn_1, \ldots, A^\nn_{p})$ of some length $p\le k$ and a $p$-tuple of non-increasing integer sequences $(\alpha[1], \ldots, \alpha[p])$ with $\alpha[i]\in \mathbb Z^{k_i}$, where $\sum_{i=1}^p k_i = k$. We then recover $\alpha^\nn$ from the conditions that 
$$
\alpha_{i}^\nn = A_j^\nn + \alpha[j]_{s},
$$
if $i = k_1+\ldots + k_{j-1} + s$ (where $1\le s\le k_j$). See Figure \ref{fig2}.

For $\beta^\nn$ we similarly define the data $B^\nn = (B^\nn_1, \ldots, B^\nn_q), (\beta[1],\ldots, \beta[q])$ where $B^\nn$ is a \textbf{nice} sequence of partitions and $\beta[j]\in \mathbb Z^{l_j}$ are $l_j$-tuples of non-increasing sequences with $\sum_{j=1}^r l_j = l$. We have
$$
\beta_{i}^\nn = B_j^\nn + \beta[j]_{s},
$$
if $i = l_1+\ldots + l_{j-1} + s$ (where $1\le s\le l_j$).

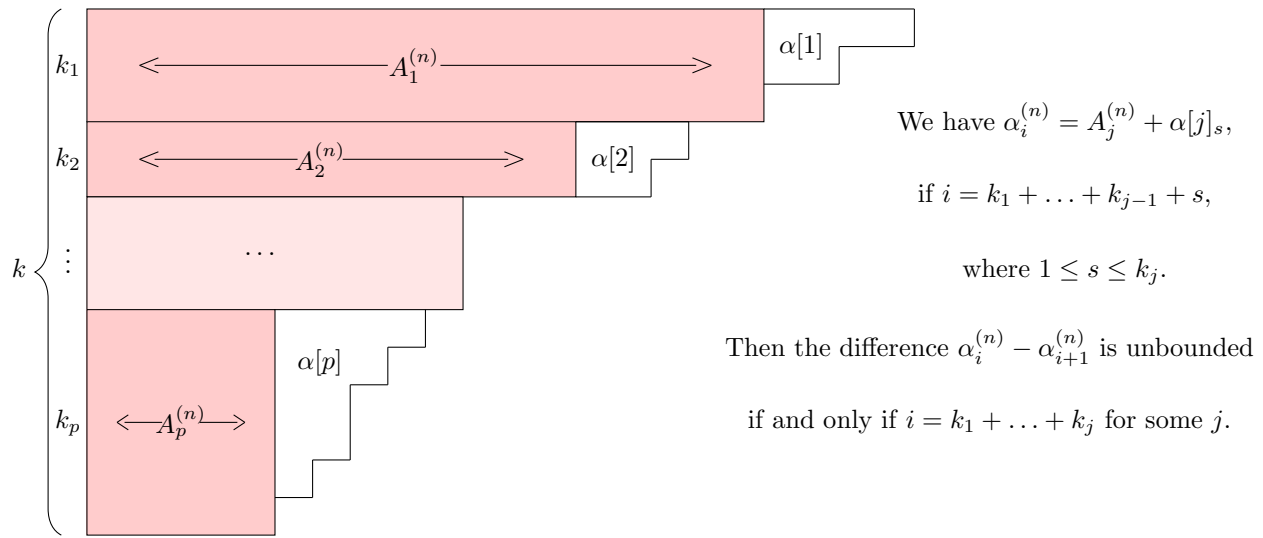
\begin{figure}[h]\caption{Unbounded and bounded differences}\label{fig2}
\begin{tikzpicture}
\filldraw[fill=red!20!white, draw=black] (0,0) rectangle (9,1.5); 
\filldraw[fill=white, draw = white] (0,2) rectangle (1, 1.6);
\draw (9,1.5) -- (11, 1.5);
\draw (11, 1.5) -- (11, 1);
\draw (11, 1) -- (10, 1);
\draw (10,1)-- (10, 0.5);
\draw (10, 0.5) -- (9, 0.5);
\node at (-0.25,0.75){$k_1$};
\node at (4.35, 0.75){$A^\nn_1$};
\draw (0.9, 0.75) -- (4, 0.75);
\draw (4.7, 0.75) -- (8.1, 0.75);
\draw (0.95, 0.85) -- (0.7, 0.75);
\draw (0.7, 0.75) -- (0.95, 0.65);
\draw (8.05, 0.85) -- (8.3, 0.75);
\draw (8.3, 0.75) -- (8.05, 0.65);
\node at (9.5, 1){$\alpha[1]$};

\filldraw[fill=red!20!white, draw=black] (0,0) rectangle (6.5,-1);
\draw (6.5, 0) -- (8, 0);
\draw (8,0) -- (8, -0.5);
\draw (8,-0.5) -- (7.5, -0.5);
\draw (7.5, -0.5) -- (7.5, -1);
\draw (7.5, -1) -- (6.5, -1);
\node at (-0.25, -0.5){$k_2$};
\node at (3.1, -0.5){$A^\nn_2$};
\node at (7, -0.5){$\alpha[2]$};
\draw (0.9, -0.5) -- (2.8, -0.5);
\draw (3.4, -0.5) -- (5.6, -0.5);
\draw (0.95, -0.4) -- (0.7, -0.5);
\draw (0.7, -0.5) -- (0.95, -0.6);
\draw (5.55, -0.4) -- (5.8, -0.5);
\draw (5.8, -0.5) -- (5.55, -0.6);

\filldraw[fill=red!10!white, draw=black] (0,-1) rectangle (5,-2.5);
\node at (2.3,-1.75){$\ldots$};
\node at (-0.25, -1.75){$\vdots$};

\filldraw[fill=red!20!white, draw=black] (0,-2.5) rectangle (2.5, -5.5);
\draw (4.5,-2.5)--(4.5, -3);
\draw (4.5, -3)--(4,-3);
\draw (4,-3)--(4,-3.5);
\draw (4,-3.5)--(3.5,-3.5);
\draw (3.5,-3.5)--(3.5,-4.5);
\draw (3.5,-4.5)--(3,-4.5);
\draw (3,-4.5)--(3,-5);
\draw (3,-5)--(2.5,-5);

\node at (-0.25, -4){$k_p$};
\node at (1.25, -4){$A^\nn_p$};
\node at (3.1, -3.2){$\alpha[p]$};
\draw (1, -4) -- (0.5, -4);
\draw (1.5, -4) -- (2, -4);
\draw (0.55, -3.9) -- (0.4, -4);
\draw (0.4, -4) -- (0.55, -4.1);
\draw (1.95, -3.9) -- (2.1, -4);
\draw (2.1, -4) -- (1.95, -4.1);

\node at (13,0){We have $\alpha_{i}^\nn = A_j^\nn + \alpha[j]_{s}$,};
\node at (13, -1){if $i = k_1+\ldots + k_{j-1} + s$,};
\node at (13, -2){where $1\le s\le k_j$.};
\node at (12, -3){Then the difference $\alpha_i^\nn-\alpha^\nn_{i+1}$ is unbounded};
\node at (12, -4){if and only if $i = k_1+\ldots + k_{j}$ for some $j$.} ;
\draw[decorate,decoration={brace,amplitude=10pt, mirror, raise=4pt}](-0.2, 1.5) -- (-0.2,-5.5);
\node at (-0.9, -1.95){$k$};
\end{tikzpicture}
\end{figure}

 The statement of Proposition \ref{cgupta} is also true for this weaker condition on $\l^\nn$ (i.e. $d(\l^\nn)=d$ for almost all $n$). The generalization of the proof is straightforward, however, the resulting computation is rather bulky.
 
 The appropriate definition of $\mu^\nn$ then requires some conditions on $\mathbf a$ and $\mathbf b$. Namely, we need to make sure that $\alpha[j]_s+a_i\ge \alpha[j]_{s+1} + a_{i+1}$, whenever  $i = k_1+\ldots + k_{j-1} + s$ with $1\le s< k_j$. Similarly, we require that $\beta[j]_s+b_i\ge \beta[j]_{s+1}+a_{i+1}$, whenever $i = l_1+\ldots + l_{j-1}+s$ with $1\le s< l'_j$.
 
 The multiplicities will then depend only on $\gamma, k, l, \mathbf a, \mathbf b, \alpha[j], $ for $1\le j\le m,$ and $\beta[s]$, for $1\le s\le r$. 
\end{remark}

\section{Categorical action of $\mathfrak{sl}_{\Z}$}\label{sslz}

In this section we closely follow the work of Inna Entova-Aizenbud \cite{EA} and define two commuting categorical type $A$ actions on the category $\hc(\gl_t)$.

In the first part of this section we give the basic definitions and restate some of the results from \cite{EA}. 

 We then study the categorical type $A$ action on the subcategory with objects $\underline \Hom(\bm,\bl)$ with fixed $\bm$. It turns out that these categorical actions yield multiple commuting actions of $\ssl_\Z$ on its Grothendieck group (the number of copies of $\ssl_\Z$ acting on it will depend on $\bm$).
 
 In the last part of the section we briefly describe a more general construction to obtain a slightly more interesting $\ssl_\Z$-action. That is, in the light of Remark \ref{rdiag}, we define the bimodules $\underline\Hom(\bm,\bl) = \prod_\cf^F \Hom_\ck (V_{\bm^\nn},V_{\bl^\nn})$, where we require that the lengths of the diagonals of partitions $\l^\nn,\bar\l^\nn, \mu^\nn, \bar\mu^\nn$ are bounded. It is a weaker condition on $\bl^\nn,\bm^\nn$ than  the one posed in Theorem \ref{dhom}, so this construction yields a wider range of bimodules.

\subsection{Basic definitions}

\begin{definition}
The degenerate affine Hecke algebra $aAHA_d$ is the vector space 
$$
\mathbb C[x_1,\ldots, x_d]\otimes\mathbb C[S_d],
$$
with multiplication defined so that $\mathbb C[x_1,\ldots, x_d]\otimes 1$ and $1\otimes \mathbb C[S_d]$ are subalgebras isomorphic to the polynomial algebra $\mathbb C[x_1,\ldots, x_d]$ and the group algebra $\mathbb C[S_d]$ of the symmetric group $S_d$ correspondingly. And with additional relation that
$$
t_j x_k - x_{t_j(k)} t_j = \begin{cases} 1, \text{ if } k = j+1,\\
-1, \text{ if } k = j,\\
0, \text{ otherwise.}
\end{cases}
$$
Here $t_j$ denotes the transposition $(j,j+1)$ in $S_d$.
\end{definition}

The following definition is taken from \cite{EA}, Definition 6.1.1.

\begin{definition}
Let $\mathcal A$ be a Karoubian $\mathbb C$-linear category. A categorical type $A$ action on $\mathcal A$ consists of the data $(F,E,x, \tau)$, where $(E,F)$ are an adjoint pair of (additive) exact endofunctors of $\mathcal A$, $x\in \End(F), \tau\in \End(F^2)$, which satisfy the following conditions:
\begin{itemize}
    \item $F$ is isomorphic to the left adjoint of $E$,
    \item For any $d\ge 2$, the natural transformations $x, \tau$ define an action of the degenerate affine Hecke algebra on $F^d$ by
    \begin{align*}
        dAHA_d = \C[x_1,\ldots, x_d]\otimes \C[S_d]\to \End(F^d)\\
        x_i\mapsto F^{d-i}xF^{i-1}\\
        (i,i+1)\mapsto F^{d-i-1}\tau F^{i-1}.
    \end{align*}
\end{itemize}
\end{definition}

Let $F_c$ be the generalized eigenspace of $x$ corresponding to the eigenvalue $c\in\mathbb C$, i.e. for each $M$ in $ob~\mathcal A$ we define $F_cM = \sum_{k\ge 0} \Ker(x_M - c)^k$ . Then 
    $$
    F = \bigoplus_{c\in\mathbb C} F_c.
    $$
    
Similarly, we get a decomposition $E = \bigoplus_{c\in\C} E_c$, where $E_c$ is left adjoint to $F_c$.

\begin{lemma}
\label{lcategactiononrepglv}
It was shown in \cite{EA},
Section 6.2 that for any rigid symmetric tensor category $\cc$ and any object $V\in \cc$, the category $\Rep_\cc(\gl(V))$ of objects with an action of $\gl(V)$ enjoys a categorical type $A$ action given by the following data:
\begin{itemize}
    \item $F=V\otimes(-), E = V^*\otimes (-)$;
    \item $x_M\in\End(V\otimes M) \cong \Hom(V\otimes V^*\otimes M, M)$ is the action map $\gl(V)\otimes M\to M$;
    \item $\tau = \sigma_{V,V}\times Id \in \End(F^2)$, where $\sigma_{-,-}$ is the symmetric braiding on $\cc$.
\end{itemize}
\end{lemma}
\begin{proof}
See \cite{EA}, Section 6.2. 
\end{proof}

Let us quickly review a few notions associated to $\ssl_\Z$ and its representations. Recall that $\ssl_\Z$ is generated by the elements $f_c= E_{c+1,c}$ and $e_c= E_{c,c+1}$, $c\in\Z$. We denote the tautological representation of $\ssl_\Z$ by $\C^\Z$. 

Let us introduce some notations. For a cell $(i,j)$ (that is, a square with vertices $(i,j), (i-1,j), \\ (i, j-1), (i-1,j-1)$) in the Young diagram of a partition $\l$ we denote its content by $ct(\square(i,j)) = j - i.$ More generally, if $\bl$ is a bipartition, we can define a Young diagram of $[\bl]_n = \sum \l_ie_i-\sum \bar\l_i e_{n-i+1} = \sum a_i e_i \in \h_n^*$ by adding cells with coordinates $(i, j)$ for $-\bar \l_1 + 1 \le j\le a_i$ (see Figure \ref{fig3}). Note that when $n\ge \ell(\l)+\ell(\bar \l)$ we can uniquely recover the bipartition $\bl$ from the Young diagram of $[\bl]_n$.

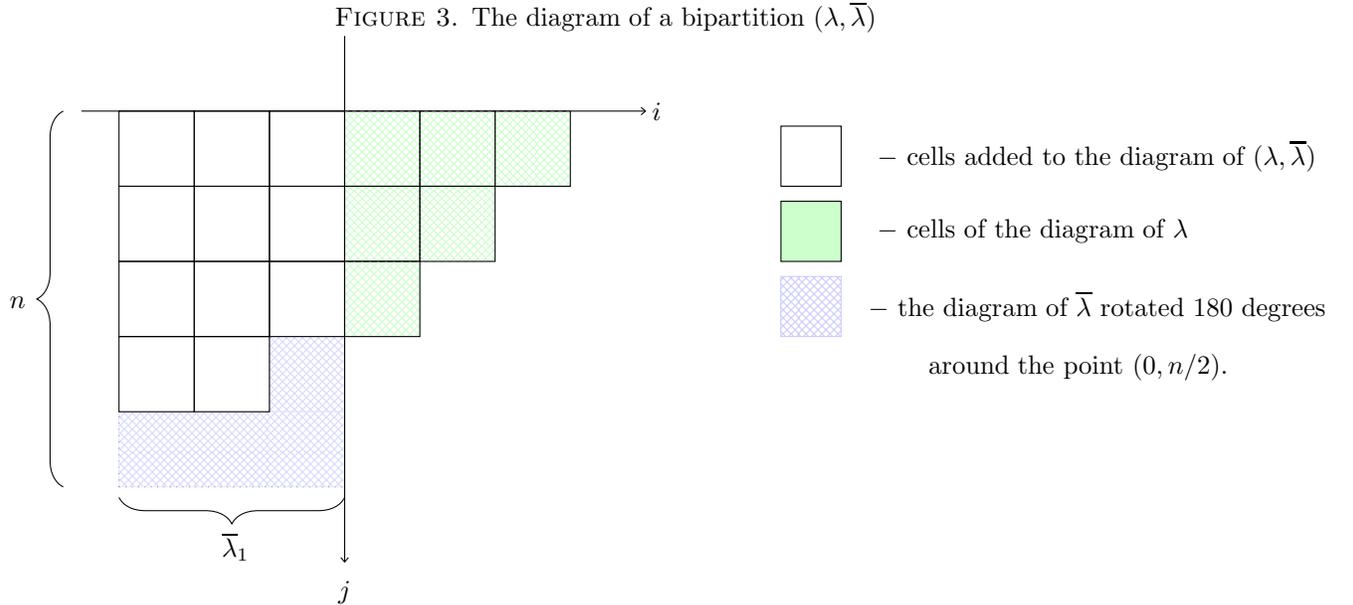
\begin{figure}[h]\caption{The diagram of a bipartition $(\l,\bar\l)$}\label{fig3}
    \begin{tikzpicture}
    
    \draw[pattern = crosshatch,  pattern color = blue!15!white!, draw = blue!7!white!] (-3,-4) rectangle (0,-5);
    \draw[pattern = crosshatch,  pattern color = blue!15!white!, draw = blue!10!white!] (-1,-3) rectangle (0,-4);
    \draw (0,-6)--(0,1);
    \draw (0,-6)--(-0.05, -5.94);
    \draw (0,-6)--(0.05, -5.94);
    \draw (-3.5,0)--(4,0);
    \draw (4,0)--(3.94,0.05);
    \draw (4,0)--(3.94,-0.05);
    \node at (0,-6.4){$j$};
    \node at (4.15,0){$i$};
    \draw[pattern = crosshatch,  pattern color = green!20!white!, draw = black] (0,0) rectangle (1,-1);
    \draw[pattern = crosshatch,  pattern color = green!20!white!, draw = black] (0,-1) rectangle (1,-2);
    \draw[pattern = crosshatch,  pattern color = green!20!white!, draw = black] (0,-2) rectangle (1,-3);
    \draw[pattern = crosshatch,  pattern color = green!20!white!, draw = black] (1,0) rectangle (2,-1);
    \draw[pattern = crosshatch,  pattern color = green!20!white!, draw = black] (2,0) rectangle (3,-1);
    \draw[pattern = crosshatch,  pattern color = green!20!white!, draw = black] (1,-1) rectangle (2,-2);
    
    \draw[decorate,decoration={brace,amplitude=10pt, mirror, raise=4pt}](-3.6, 0) -- (-3.6,-5);
    \draw[decorate,decoration={brace,amplitude=10pt, mirror, raise=4pt}](-3, -5) -- (0,-5);
    \node at (-4.35,-2.55){$n$};
    \node at (-1.45, -5.8){$\bar\l_1$};
    
    \draw (-3,0) rectangle (-2,-1);
    \draw (-2,0) rectangle (-1,-1);
    \draw (-1,0) rectangle (0,-1);
    \draw (-3,-1) rectangle (-2,-2);
    \draw (-2,-1) rectangle (-1,-2);
    \draw (-1,-1) rectangle (0,-2);
    \draw (-3,-2) rectangle (-2,-3);
    \draw (-2,-2) rectangle (-1,-3);
    \draw (-1,-2) rectangle (0,-3);
    \draw (-3,-3) rectangle (-2,-4);
    \draw (-2,-3) rectangle (-1,-4);
    
    \draw (5.8,-0.2) rectangle (6.6,-1);
    \node at (10,-0.6){$-$ cells added to the diagram of $(\l,\bar\l)$};
    \filldraw[fill= green!20!white!] (5.8,-1.2) rectangle (6.6,-2);
    \node at (9.15,-1.6){$-$ cells of the diagram of $\l$};
    \draw[pattern = crosshatch,  pattern color = blue!20!white!, draw = blue!20!white!] (5.8,-2.2) rectangle (6.6,-3);
    \node at (10,-2.6){$-$ the diagram of $\bar\l$ rotated $180$ degrees};
    \node at (9.75, -3.4){ around the point $(0,n/2)$.};
    
    \draw[dotted][draw = black!40!white] (-3,-5)--(0,-5);
    \draw[dotted][draw = black!40!white] (-3,-5)--(-3,-4);
    
    \end{tikzpicture}
\end{figure}

Let us denote by $\l+\square_c$ (resp. $\l-\square_c$) the set of (bi)partitions obtained by adding (resp. removing) a cell of content $c$ to (resp. from) the Young diagram of $\l$. This set consists of at most one (bi)partition. 

Let us now consider some representations of $\ssl_\Z$. First of all, let us look at the module $\C^\Z$ with a basis $v_c, c\in \Z$. We have
$$
f_c v_b = \delta_{c,b} v_{c+1},
$$
$$
e_c v_b = \delta_{c+1, b} v_c.
$$

Now, the module  $\Lambda^n \C^\Z$ has a basis $v_{i_1}\wedge \ldots\wedge v_{i_n}$ indexed by decreasing sequences $i_1>\ldots >i_n$. For any integral dominant weight $\nu$ of $\gl_n$ let us denote by $v_\nu$ the vector corresponding to the sequence $\nu_1>\nu_2-1>\ldots>\nu_n-n+1$. Thus, $\Lambda^n \C^\Z$ has a basis indexed by dominant integral weights of $\gl_n$, or, equivalently, by bipartitions $\bl$ with $n\ge \ell(\l)+\ell(\bar\l)$.  We get that
$$
f_c v_\nu  = v_{\nu+\square_c},
$$
$$
e_c v_\nu = v_{\nu-\square_c},
$$
where $v_{\nu\pm \square_c}=0$ whenever $\nu\pm\square_c$ is empty. 

We recall that the Fock module $\mathfrak F$ is the vector space with a basis consisting of infinite wedges $v_{i_0}\wedge v_{i_{-1}}\wedge v_{i_{-2}}\wedge \ldots$, where $i_0>i_{-1}>i_{-2}>\ldots$ and $i_{-s}=-s$ for $s\mathfrak{g} 0$. For any partition $\nu$ we define
$$
v_\nu = v_{\nu_1}\wedge v_{\nu_2-1} \wedge v_{\nu_3-2}\wedge\ldots
$$
Thus, $\mathfrak F$ has a basis indexed by partitions of arbitrary size. We have
$$
f_c v_\nu  = v_{\nu+\square_c},
$$
$$
e_c v_\nu = v_{\nu-\square_c}.
$$

Now, for any $\ssl_\Z$-module $M$ denote by $M^\vee$ its twist by the automorphism of $\ssl_\Z$ sending $f_c$ to $e_{c}$ and vice versa. Let $M^\tau$ denote the twist of $M$ by the automorphism of $\ssl_\Z$ sending $f_c$ to $f_{-c}$ and $e_c$ to $e_{-v}$. Note that $M^\vee, M^\tau$ have the same underlying vector spaces as $M$. When $M = \Lambda^n\C^\Z$ (resp. $\mathfrak F$) and $\nu $ is an integral dominant weight of $\gl_n$ (resp. a partition) we have for $v_\nu\in M^\vee$ that
$$
f_c v_\nu  = v_{\nu-\square_{c}},
$$
$$
e_c v_\nu = v_{\nu+\square_{c}}.
$$

The categorical type $A$ action induces an action of $\ssl_\Z$ on the complexified Grothendieck group whenever all eigenvalues of $x$ are integer. Let us consider a few examples.

\begin{ex}
All objects of $\cc_n$ enjoy a natural action of $\gl(V)=\gl_n$. So, we can consider the categorical type $A$ action on $\cc_n$ given by the tuple $(F, E, x, \tau)$ as in Lemma \ref{lcategactiononrepglv}.  We have for any integral dominant weight $\bl$
$$
F(V_\bl) = \bigoplus_{\bm\in \bl+\square} V_\bm.
$$
It was proved in \cite{EA} that if $\bm\in\bl+\square_{c}$, then $V_\bm$ is the eigenspace of $x_{V_\bl}$ with  eigenvalue $c$. Then we have
$$
F = \bigoplus_{c\in \Z} F_c
$$
with
$$
F_c(V_\bl) = \bigoplus_{\bm\in \bl+\square_c} V_\bm.
$$

This decomposition induces an action of $\ssl_\Z$ on the complexified Grothendieck group of $\cc_n$. By direct comparison of the formulas, it is easy to deduce that $\C\tens{\Z}Gr(\cc_n)$ is isomorphic to $\Lambda^n\C^\Z$ as an $\ssl_\Z$-module. 
\end{ex}

\begin{ex}
Instead of $\cc_n$, we can consider the category $\cc_t$, whose objects enjoy the natural action of $\gl(V) = \gl_t$. For any bipartition $\bl$ and large enough $n$ the set of bipartitions obtained from $[\bl]_n$ by adding/removing a cell doesn't depend on $n$. Moreover, the content of the added/removed cell depends linearly on $n$, therefore we can extend this procedure to bipartitions considered as weights of $\gl_t$. We get that adding a cell of content $c$ to a bipartition $\bl=(\l,\bar\l)$ is the same as either adding a cell of content $c$ to $\l$ (and leaving $\bar\l$ intact) or removing a cell of content $t - c$ from $\bar\l$ (and leaving $\l$ intact). We can split the operator $x$ into the sum of two other operators 
$$
x = x'+x'',
$$
where the eigenvalues of $x'$ are integer and the eigenvalues of $x''$ lie in $t+\Z$. The data $(E,F, x', \tau)$ and $(E,F,x''- t, \tau)$ define two categorical type $A$ actions on $\cc_t$.  As an $\ssl_\Z\oplus \ssl_\Z$-module $\C\tens{\Z}Gr(\cc_t)$ is isomorphic to $\mathfrak F\otimes (\mathfrak F^\tau)^\vee$ (see \cite{EA} Section 8).
\end{ex}

\subsection{Categorical type $A$ action on Harish-Chandra bimodules}

Lemma \ref{lcategactiononrepglv} allows us to define two categorical type $A$ actions on $\hc(\gl_t)$ by considering its objects as left (or right) $\gl_t$-modules. Since these actions are clearly symmetrical, we will stick to studying only the left action.

\begin{definition}
Let $\mathcal D_\bm$ be the Serre subcategory of $\hc(\gl_t)$ generated by $\underline \Hom(\bm, \bl)$ with $\bm$ fixed.
\end{definition}

 Recall that $\bl = (\l, \bar \l)$ and $\l, \bar\l$ denote the tuples $(\alpha, \beta, \gamma),  (\bar\alpha, \bar\beta, \bar\gamma)$, where $ \gamma, \bar\gamma\in \cp$,
\begin{align*}
    \alpha = \prod_\cf \alpha^\nn,\\
    \beta = \prod_\cf \beta^\nn,\\
    \bar\alpha = \prod_\cf \bar\alpha^\nn,\\
    \bar\beta = \prod_\cf \bar\beta^\nn,
\end{align*}
$\alpha^\nn,\beta^\nn,
\bar\alpha^\nn,\bar\beta^\nn$ are \textbf{nice} sequences of partitions, and we require that $\ell(\alpha^\nn), \ell(\beta^\nn), \ell(\bar \alpha^\nn), \ell(\bar\beta^\nn)$ are constant for almost all $n$.

We get that $\bl$ is in some sense the ultraproduct of bipartitions $\bl^\nn$. The conditions on these bipartitions are such that the set $\bl^\nn\pm \square$ has the same number of elements for almost all $n$. Moreover, we can list all the elements of this set. For any $\bm\in\bl^\nn\pm\square$ with
$$
\bm^\nn = ([\sigma^\nn, \tau^\nn, \epsilon],[\bar\sigma^\nn,\bar\tau^\nn,\bar\epsilon])
$$
we have (for almost all $n$) that either
\begin{align}
    \sigma^\nn_i = \alpha^\nn_i\pm 1,\text{ or}\label{eqfirstcond}\\
    \bar\sigma^\nn_i = \bar\alpha^\nn_i\mp 1,\text{ or}\label{eqsecondcond}\\
    \tau^\nn_i = \beta^\nn_i\pm 1, \text{ or}\label{eqthirdcond}\\
    \bar\tau^\nn_i = \bar \beta^\nn_i \mp 1, \text{ or}\label{eqfourthcond}\\
    \epsilon \in \gamma\pm\square, \text{ or}\label{eqfifthcond}\\
    \bar\epsilon\in \bar\gamma\mp\square,\label{eqlastcond}
\end{align}
and only one of the conditions holds for exactly one $i$. The change of sign in (\ref{eqsecondcond}), (\ref{eqfourthcond}) and (\ref{eqlastcond}) reflects the fact that removing a cell from $\bar\alpha, \bar\beta $ or $\bar\epsilon$ results in adding a cell to $\bl$.  Note that because our sequences of partitions are \textbf{nice}, we can add/remove a cell to/from any row of $\alpha^\nn, \beta^\nn, \bar\alpha^\nn,\bar\beta^\nn$  (for almost all $n$) and obtain a partition. The content of the added/removed cell depends linearly on $\alpha_i^\nn, \beta_i^\nn, \bar\alpha_i^\nn, \bar\beta_i^\nn$, and $n$, so we can extend this notion algebraically to $\Rep(GL_t)$ and tuples $\bl$. 

So, let $\bl\pm \square$ denote the set of tuples $\bm$ obtained from $\bl$ by adding/removing a cell. That is, $\bm=\prod_\cf \bm^\nn$, where each $\bm^\nn\in \bl^\nn\pm\square$ satisfies one and the same of the conditions \ref{eqfirstcond}-\ref{eqlastcond} with the same value $i$ for almost all $n$. We have 
\begin{align}\label{eqfofhom}
F(\underline\Hom(\bm, \bl)) = \bigoplus_{\bn\in\bl+\square} \underline\Hom(\bm, \bn).
\end{align}

Each $\underline\Hom(\bm, \bn)$ is the eigenspace of $x_{\underline\Hom(\bm, \bl)}$ with eigenvalue $c$ if $\bn\in \bl+\square_{c}$. 

\begin{lemma}\label{leigenv}
Let $\bn\in \bl+\square$. Then $\bn^\nn$ satisfies one of the conditions (\ref{eqfirstcond})--(\ref{eqlastcond}) for almost all $n$. Let $c(\bn)$ be the eigenvalue of $x_{\underline\Hom(\bm, \bl)}$ on $\underline\Hom(\bn, \bl)$. Then
$$
c(\bn) = \begin{cases}
\alpha_i+\ell(\beta)+1-i, \text{ if } \bn^\nn \text{ satisfies (\ref{eqfirstcond})} ,\\
-\bar \alpha_i-\ell(\bar\beta)-t+i, \text{ if } \bn^\nn \text{ satisfies (\ref{eqsecondcond})},\\
-\beta_i - \ell(\alpha)-1+i, \text{ if } \bn^\nn \text{ satisfies (\ref{eqthirdcond})},\\
\bar\beta_i+\ell(\bar\alpha)-t-i, \text{ if } \bn^\nn \text{ satisfies (\ref{eqfourthcond})},\\
c'+\ell(\beta)-\ell(\alpha), \text{ if } \bn^\nn \text{ satisfies (\ref{eqfifthcond}) and }\epsilon\in \gamma+\square_{c'},\\
-c'-\ell(\bar\beta)+\ell(\bar\alpha)-t, \text{ if } \bn^\nn \text{ satisfies \ref{eqlastcond}) and }\bar\epsilon\in \bar\gamma-\square_{c'}.
\end{cases}
$$
\end{lemma}
\begin{proof}
It follows from the calculation of the content of the added cell for $\bn^\nn$ and then taking the ultraproduct of the obtained values.
\end{proof}

\begin{definition}
We say $\bn\in\bl\pm\square_c$ if $\bn\in\bl\pm\square$ and $c(\bn)=c$.
\end{definition}

Let us, following Lemma \ref{rextdefinitionofhom}, require that  $\rm{trdeg}_{\mathbb Q}\mathbb Q(\alpha_{i_1},\beta_{j_1},\bar\alpha_{i_2},\bar\beta_{j_2}, t)_{i_1\le \ell(\alpha), i_2\le \ell(\bar\alpha), j_1\le\ell(\beta), j_2\le\ell(\bar\beta)} = \ell(\alpha)+\ell(\beta)+\ell(\bar\alpha)+\ell(\bar\beta)+1$. Then the numbers $$\{\alpha_{i_1}, -\bar\alpha_{i_2}-t, -\beta_{j_1}, \bar\beta_{j_2}-t, \ell(\beta)-\ell(\alpha), -\ell(\bar\beta)+\ell(\bar\alpha)-t\}_{i_1\le \ell(\alpha), i_2\le \ell(\bar\alpha), j_1\le\ell(\beta), j_2\le\ell(\bar\beta)}$$ lie in separate $\Z$-cosets of $\C$. For any coset $r+\Z\in \C/\Z$ we can define the operator $x^{(r)}$, such that 
$$
x = \sum_{r+\Z\in\C/\Z} x^{(r)}
$$
and if $N$ is an eigenspace of $x_M$ with eigenvalue $c$, then 
$$
x_M^{(r)}|_N = \begin{cases} c, \text{ if } c \in r+\Z,\\ 0, \text{ otherwise.} \end{cases}
$$

Equation \ref{eqfofhom} implies that the categorical type $A$ action given by $(F,E,x,\tau)$ restricts to the category $\mathcal D_\bm$.  Moreover, the tuples $(F,E,x^{(r)}-r, \tau)$ for cosets $r+\Z$ and some coset representatives $r\in\C$ induce  well-defined commuting categorical $\ssl_\Z$-actions. 

Let $\bm=(\mu,\bar\mu)$ with
\begin{align*}
    \mu=(\alpha,\beta,\gamma),\\
    \bar\mu=(\bar\alpha,\bar\beta,\bar\gamma).
\end{align*}
Then for any $\bl$, such that $\underline\Hom(\bm, \bl)\in \mathcal D_\bm$ we have that $\bl=(\l,\bar\l)$ and
\begin{align*}
    \l=(\alpha+\mathbf a, \beta+\mathbf b, \delta),\\
    \bar\l = (\bar\alpha+\bar{\mathbf a}, \bar\beta+\bar{\mathbf b}, \bar\delta)
\end{align*}
for some integer sequences $\mathbf a, \mathbf b, \bar{\mathbf a}, \bar{\mathbf b}$ and some $\delta, \bar\delta\in\cp$.

The calculation in Lemma \ref{leigenv} shows that $(F,E,x^{(r)}-r, \tau)$ induces a trivial $\ssl_\Z$-action unless $r$ satisfies one of the following conditions for some $i$:
\begin{align}
    r=\alpha_i + \ell(\beta) - i, \text{ or}\\
    r =-\bar \alpha_i-t-\ell(\bar\beta) + i, \text{ or}\\
    r=-\beta_i-\ell(\alpha)+i, \text{ or}\\
    r=\bar\beta_i-t+\ell(\bar\alpha)-i, \text{ or}\\
    r = \ell(\beta)-\ell(\alpha), \text{ or}\\
    r= -\ell(\bar\beta)+\ell(\bar\alpha)-t.
\end{align}

In all of the above cases the eigenvalues of $x^{(r)}-r$ are integers and the decomposition
$$
F=\bigoplus_{i\in \Z} F_i
$$
gives the $\ssl_\Z$-action on the complexified Grothendieck group $\C\tens{\Z}Gr(\mathcal D_\bm)$. Let us denote by $\ssl_\Z^{(r)}$ the copy of $\ssl_\Z$, which acts on $\C\tens{\Z}Gr(\mathcal D_\bm)$ via $(F,E,x^{(r)}-r, \tau)$.

\begin{theorem} \label{tgrmod1}
 Let $S_1 = \{\alpha_i + \ell(\beta) - i\}, S_2=\{-\bar \alpha_i~-~t~-~\ell(\bar\beta)~+~i\}, S_3 = \{-\beta_i-\ell(\alpha)+i\}, S_4 = \{\bar\beta_i-t+\ell(\bar\alpha)-i\}$,  i.e. $S_k$ is the set of values of $r$ satisfying one of the conditions (7.8.$k$).\\ Then as a $\left(\bigoplus_{r\in S_1\cup S_2\cup S_3\cup S_4}\ssl_\Z^{(r)}\right)\oplus \ssl_\Z^{(\ell(\beta)-\ell(\alpha))}\oplus \ssl_\Z^{(-\ell(\bar\beta)+\ell(\bar\alpha)-t)}$-module 
 $$
 \C\tens{\Z}Gr(\mathcal D_\bm) \simeq \left(\bigotimes_{r\in S_1} \C^\Z\otimes \bigotimes_{r\in S_2} ((\C^\Z)^\tau)^\vee\otimes \bigotimes_{r\in S_3} (\C^\Z)^\tau\otimes \bigotimes_{r\in S_4} (\C^\Z)^\vee\right)\otimes  \mathfrak F\otimes (\mathfrak F^\tau)^\vee.
 $$

\end{theorem}
\begin{proof}
The tensor product on the right hand side has a basis $$\bigotimes_{r\in S_1\cup S_2\cup S_3\cup S_4} v_{i_r}^{(r)}\otimes v_\delta^{(\ell(\beta)-\ell(\alpha))}\otimes v_{\bar\delta}^{(-\ell(\bar\beta)+\ell(\bar\alpha)-t)}.$$ The isomorphism takes this basis element to the class of $\underline\Hom(\bm, \bl)$ with $\bl=(\l,\bar\l)$,
\begin{align*}
    \l=(\alpha+\mathbf a, \beta+\mathbf b, \delta),\\
    \bar\l = (\bar\alpha+\bar{\mathbf a}, \bar\beta+\bar{\mathbf b}, \bar\delta),\\
    \mathbf a = (a_1,\ldots, a_{\ell(\alpha)}), \mathbf b = (b_1,\ldots, b_{\ell(\beta)}),\\
    \bar{\mathbf a} = (\bar a_1,\ldots,\bar a_{\ell(\bar\alpha)}), \bar{\mathbf b} = (\bar b_1, \ldots, \bar b_{\ell(\bar \beta)}).
\end{align*}
such that 
\begin{align*}
    a_j = i_{\alpha_i+\ell(\beta)-j}-1,\\
    \bar a_j = i_{-\bar\alpha_i-t-\ell(\bar \beta)+j}-1,\\
    b_j = i_{-\beta_i - \ell(\alpha)  + j}-1,\\
    \bar b_j = i_{\bar\beta_i - t+\ell(\bar\alpha)-j}-1.
\end{align*}
\end{proof}

\subsection{A more general construction}

Recall that Remark \ref{rdiag} implied that there is a wider range of Harish-Chandra bimodules of finite K-type than was constructed in Section \ref{s3}. That is, we could take an ultraproduct of bimodules $\Hom_\ck(V_{\bm^\nn},V_{\bl^\nn})$ with partitions $\l^\nn,\overline \l^\nn,\mu^\nn,\bar\mu^\nn$ having a bounded diagonal length. This construction is more general but requires a slightly bulkier combinatorial data. We will now describe this construction in more detail.

Suppose we have a sequence of partitions $\l^\nn$ with bounded length of the diagonal. Then for almost all $n$ we have $d(\l^\nn)=d$ for some fixed $d$. We define $k$ to be the maximal index $i$ for which the sequence $\l_i^\nn$ is not bounded for almost all $n$. Similarly, $l$ is defined as the maximal $j$ for which the sequence $(\l\nn)'_j$ is not bounded for almost all $n$. Recall that to  $\l^\nn = [\alpha^\nn, \beta^\nn, \gamma]$, where $\ell(\alpha^\nn) = k, \ell(\beta^\nn) =l$ for almost all $n$, we assigned the data of \textbf{nice} sequences of partitions $A^\nn = (A_1^\nn, \ldots,  A_p^\nn), B^\nn = (B_1^\nn, \ldots B_q^\nn)$ of some lengths $p\le k, q\le l$, a $p$-tuple $(\alpha[1],\ldots,\alpha[p])$ and a $q$-tuple $(\beta[1],\ldots, \beta[q])$ of non-increasing integer sequences with $\alpha[j]\in \mathbb Z^{k_j}, \beta[m]\in \mathbb Z^{l_m}$, where
$$
\sum_{j=1}^{p} k_j = k, ~\sum_{m=1}^q l_m = l.
$$

 Thus, $\l = \prod_\cf \l^\nn$ results in the following combinatorial data: 
 \begin{itemize}
     \item $A=(A_1,\ldots, A_p) \in \mathbb C^p$ with $A = \prod_\cf A^\nn$, where $A^\nn$ is a \textbf{nice} sequence of partitions;
     \item $B = (B_1, \ldots B_q)\in\mathbb C^q$ with $B = \prod_\cf B^\nn$, where $B^\nn$ is a \textbf{nice} sequence of partitions;
     \item a $p$-tuple of non-increasing integer sequences $(\alpha[1], \ldots, \alpha[p])$ with $\alpha[i]\in \mathbb Z^{k_i}$, where $\sum_{i=1}^p k_i = k$;
     \item a $q$-tuple of non-increasing integer sequences $\beta[1], \ldots, \beta[q]$ with $\beta[i]\in\mathbb Z^{l_i}$, where $\sum_{i=1}^q l_i = l$;
     \item a partition $\gamma$.
 \end{itemize}

We then have 
\begin{itemize}
    \item $\alpha_i = A_j + \alpha[j]_m$, if  $i = k_1+\ldots+k_{j-1}+m$ with $1\le m\le k_j$
    \item $\beta_i =  B_j + \beta[j]_m$, if  $i = l_1+\ldots + l_{j-1} + m$ with $m\le l_j$,
    \item $\l^\nn = [\alpha^\nn, \beta^\nn, \gamma]$.
\end{itemize}

Let us denote by $\cp_k'$ the set of non-increasing integer sequences of length $k$. Let $\cp'= \bigcup_k \cp'_k$. For $\alpha\in \cp'$ we say $\ell(\alpha) = k$ if $\alpha\in \cp'_k$.
   
\begin{remark}
In light of Lemma \ref{rextdefinitionofhom}, we can start with the data of $\l = (A,B, (\alpha[1],\ldots,\alpha[p]),(\beta[1],\ldots,\beta[q]),\gamma)$ with $A=(A_1,\ldots, A_p)\in\mathbb C^p, B=(B_1,\ldots, B_q)\in \mathbb C^q$, and $\alpha[i], \beta[j]\in \cp'$. If we now assume that $\rm{trdeg}_{\mathbb Q}\mathbb Q(A_1,\ldots, A_p, B_1, \ldots, B_q, t) = p+q+1$, then we can find a sequence partitions $\l^\nn$ for which $\l = \prod_\cf \l^\nn$ in an appropriate sense.
\end{remark}

Define $\mathbf A = (A_1,\alpha[1],\ldots, A_p,\alpha[p])\in (\C\times \cp')^p$ and put $\mathbf A^1 = (A_1,\ldots , A_p)\in \C^p$ and $\mathbf A^2 = (\alpha[1],\ldots, \alpha[p])\in(\cp')^p$.

When $k_i= l_i =1$ for all $i$ and $\mathbf A^2 = \mathbf B^2 = 0$,  the datum of $\l = (\mathbf A, \mathbf B, \gamma)$ is the same as in Lemma \ref{rextdefinitionofhom}. That is, $\l = (\alpha, \beta, \gamma)$ with $\alpha = \mathbf A^1, \beta = \mathbf B^1$.

So, let $\mu$ be the tuple $(\mathbf A,  \mathbf B, \gamma)$, $\bar\mu = (\bar{\mathbf A}, \bar{\mathbf B}, \bar \gamma)$ and $\bm=(\mu,\bar \mu)$. Let 
$$
\l=(\mathbf A+\mathbbm a, \mathbf B+\mathbbm b, \delta),
$$
$$
\bar\l = (\bar{\mathbf A}+\bar{\mathbbm a},\bar{\mathbf B}+\bar{\mathbbm b}, \bar \delta)
$$
with $\mathbbm a = (\mathbbm a[1], \ldots, \mathbbm a[p])\in (\cp')^p$ and $\ell(\mathbbm a[i]) = \ell(\alpha[i])$ and similar conditions on $\bar{\mathbbm a}, \mathbbm b, \bar{\mathbbm b}$, where the notation $\mathbf A+\mathbbm a$ means that 
$$
(\mathbf A+\mathbbm a)^1=\mathbf A^1
$$
and 
$$
(\mathbf A+\mathbbm a)^2 = \mathbf A^2 + \mathbbm a = (\alpha[1]+\mathbbm a_1,\ldots, \alpha[p]+\mathbbm a_p).
$$
Then we can define the bimodules $\underline\Hom(\bm, \bl) = \prod_\cf^F \Hom_\ck (V_{\bm^\nn}, V_{\bl^\nn})$. This is just a slight generalization of the construction in Section \ref{s3}. 

Given $\l=(\mathbf A, \mathbf B, \gamma), \bar \l = (\bar{\mathbf A}, \bar{\mathbf B}, \delta)$, $\bl = (\l, \bar \l)$ as above, we can still define the sets $\bl\pm\square_c$ (similarly to what was done in the previous part for $\l=(\alpha, \beta, \gamma)$). Let  $\bl\pm\square$ denote the set of all bipartitions obtained from $\bl$ by adding or removing a cell from its Young diagram. Because  $\mathbf A, \bar{\mathbf A}, \mathbf B, \bar{\mathbf B}$ are defined as ultraproducts of \textbf{nice} sequences of partitions, for almost all $n$ the number of bipartitions in $\bl^\nn\pm \square$ is constant and depends only on  $\alpha[i], \beta[j], \gamma, \bar\alpha[i], \bar\beta[j], \delta$, that is, we obtain them by adding/removing cells from the Young diagrams of (bi)partitions $\alpha[i], \beta[j], \gamma, \bar\alpha[i], \bar\beta[j], \delta$. We construct the elements of $\bl+\square$ by taking the ultraproduct of bipartitions in $\bl^\nn+\square$ obtained by applying the same operation of the form ``add a cell to $\theta$" for some $\theta\in\{\alpha[i],\beta[j],\gamma\}$ or ``remove a cell from $\theta$" for some $\theta\in\{\bar \alpha[i],\bar\beta[j],\delta\}$. Similarly, elements of $\bl-\square$ are ultraproducts of bipartitions in $\bl^\nn-\square$ obtained by applying either the operation of the form  ``remove a cell from $\theta$" for some $\theta\in\{\alpha[i],\beta[j],\gamma\}$ or ``add a cell to $\theta$" for some $\theta\in\{\bar \alpha[i],\bar\beta[j],\delta\}$.
We obtain the possible values of $c$ by taking the ultraproduct of contents of added/removed cells in the Young diagrams of $\bl^\nn$.

\begin{lemma}\label{lpossiblecontents}
The sets $\bl\pm\square_c$ can be non-empty only if one of the following holds for some $j$:

\begin{itemize}
    \item $c = A_j+l-k_1-\ldots-k_{j-1}+m$ for some $m\in\mathbb Z$; 
    
    In this case $\bl\pm \square_c$ is non-empty if and only if $\alpha[j]\pm\square_{m}$ is. 
    
    \item $c=-B_j-k+l_1+\ldots+l_{j-1}+m$ for some  $m\in\mathbb Z$;
    
    In this case $\bl\pm\square_c$ is non-empty if and only if $\beta[j]\pm \square_{-m}$ is.

    \item $c = k-l+m$ for some $m\in \Z$;
    
    In this case $\bl\pm \square_c$ is non-empty if and only if $\gamma\pm\square_{m}$ is.
    
    \item $c = - \bar A_j-t-\bar L+\bar k_1+\ldots \bar k_{j-1}+m$ for some $m\in\Z$;
    
    In this case $\bl\pm \square_c$ is non-empty if and only if $\bar\alpha[j]\mp \square_{-m}$ is .
    
    \item $c=\bar B_j-t+\bar k -\bar l_1-\ldots-\bar l_{j-1} +m$ for some  $m\in\mathbb Z$;
    
    In this case $\bl\pm\square_c$ is non-empty if and only if $\bar \beta[j]\mp \square_{m}$ is.
    
    \item $c = -t+\bar k-\bar l+m$ for some $m\in\Z$.
    
    In this case  $\bl\pm \square_c$ is non-empty if and only if $\delta\mp\square_{-m}$ is.
    
\end{itemize}
\end{lemma}

 Let us fix some $\bm=(\mu,\bar\mu)$ with
\begin{align*}
    \mu= (\mathbf A, \mathbf B, \gamma),\\
    \bar\mu = (\bar{\mathbf A},\bar{\mathbf B}, \bar\gamma).
\end{align*}
Define $\mathcal D_\bm$ to be the Serre subcategory of the category of Harish-Chandra bimodules generated by $\underline \Hom(\bm,\bl)$, where $\bl = (\l,\bar\l)$ runs over all tuples
\begin{align}
    \l = (\mathbf A+\mathbbm a, \mathbf B + \mathbbm b, \delta),\\
    \bar\l = (\bar{\mathbf A}+\bar{\mathbbm a},\bar{\mathbf B}+\bar{\mathbbm b}, \bar\delta)
\end{align}
for some $\delta, \bar\delta, \mathbbm a, \mathbbm b, \bar{\mathbbm a}, \bar{\mathbbm b}$. 

As in the previous part we put
$$
x = \sum_{r+\Z\in \C/\Z} x^{(r)}.
$$

Then the previous calculation shows that the action of $\ssl_\Z^{(r)}$ on $\C\tens{\Z}Gr(\mathcal D_\bm)$ is trivial unless for some $j$ one of the following conditions holds:
\begin{align}
    r = A_j+l-k_1-\ldots-k_{j-1}, \text{ or}\\
    r = -B_j - k+l_1+\ldots+l_{j-1}, \text{ or}\\
    r = k-l, \text{ or}\\
    r = -\bar A_j -t -\bar l + \bar k_1+\ldots+\bar k_{j-1}, \text{ or}\\
    r = \bar B_j - t + \bar k - \bar l_1-\ldots-\bar l_{j-1}, \text{ or}\\
    r = -t + \bar k - \bar l.
\end{align}

\begin{theorem} \label{tgrmod2}
 Let
 \begin{align*}
 S_1 = \{A_j + l - k_1 - \ldots - k_{j-1}\}_{j=1}^p,\\
 S_2=\{-B_j-k+l_1+\ldots+l_{j-1}\}_{j=1}^q,\\
 S_3 = \{-\bar A_j-t-\bar l+\bar k_1+\ldots \bar k_{j-1}\}_{j=1}^{\bar p},\\
 S_4 = \{\bar B_j - t+\bar k - \bar l_1 - \ldots  - \bar l_{j-1}\}_{j=1}^{\bar q}.
 \end{align*}
 Then as a $\left(\bigoplus_{r\in S_1}\ssl_\Z^{(r)}\right)\oplus \left( \bigoplus_{r\in S_2}\ssl_\Z^{(r)}\right)\oplus\left(\bigoplus_{r\in S_3}\ssl_\Z^{(r)}\right)\oplus\left(\bigoplus_{r\in S_4}\ssl_\Z^{(r)}\right)\oplus \ssl_\Z^{(K-L)}\oplus \ssl_\Z^{(-t+\bar K - \bar L)}$-module 
 $$
 \C\tens{\Z}Gr(\mathcal D_\bm) \simeq \left( \bigotimes_{j=1}^p \Lambda^{k_j}\C^\Z\right)\otimes \left( \bigotimes_{j=1}^q (\Lambda^{l_j}\C^\Z)^\tau\right)\otimes\left( \bigotimes_{j=1}^{\bar p} ((\Lambda^{\bar k_j}\C^\Z)^\tau)^\vee\right)\otimes \left( \bigotimes_{j =1}^{\bar q} (\Lambda^{\bar l_j}\C^\Z)^\vee\right)\otimes  \mathfrak F\otimes (\mathfrak F^\tau)^\vee.
 $$

\end{theorem}
\begin{proof}
The tensor product on the right hand side has a basis $$\bigotimes_{r\in S_1\cup S_2\cup S_3\cup S_4} v_{\nu_r}^{(r)}\otimes v_\delta^{(\ell(\beta)-\ell(\alpha))}\otimes v_{\bar\delta}^{(-\ell(\bar\beta)+\ell(\bar\alpha)-t)}.$$ The isomorphism takes this basis element to the class of $\underline\Hom(\bm, \bl)$ with $\bl=(\l,\bar\l)$,
$$
\l=(\mathbf A+\mathbbm a, \mathbf B+\mathbbm b, \delta),
$$
$$
\bar\l = (\bar{\mathbf A}+\bar{\mathbbm a},\bar{\mathbf B}+\bar{\mathbbm b}, \bar \delta)
$$
\begin{align*}
    \mathbbm a = (\mathbbm a[1],\ldots, \mathbbm a[p]), \mathbbm b = (\mathbbm b[1],\ldots, \mathbbm b[q]),\\
    \bar{\mathbbm a} = (\bar{\mathbbm a}[1],\ldots,\bar{\mathbbm a}[\bar p]), \bar{\mathbbm b} = (\bar{\mathbbm b}[1], \ldots, \bar{\mathbbm b}[\bar q]).
\end{align*}
such that 
\begin{align*}
    \mathbbm a[j] = \nu_{A_j+l-k_1-\ldots -k_{j-1}},\\
    \mathbbm b[j] = \nu_{-B_j-k+l_1+\ldots+l_{j-1}} \\
    \bar{\mathbbm a}[j] = \nu_{-\bar A_j-t-\bar l + \bar k_1 + \ldots+\bar k_{j-1}},\\
    \bar{\mathbbm b}[j] = \nu_{\bar B_j - t + \bar k - \bar l_1 - \ldots - \bar l_{j-1}}.
\end{align*}
\end{proof}

\section{Appendix by Serina Hu and Alexandra Utiralova}

In this section we will generalize all the constructions from Sections $3-5$ for the categories $\Rep(O_t)$ and $\Rep(Sp_t)$.

So, we let $\cc_t = \Rep(O_t)$ or $\Rep(Sp_t)$ and $G_n = O_{2n+1}$ or $Sp_{2n}$ correspondingly.

\subsection{Finite dimensional bimodules}

We would like to construct Harish-Chandra bimodules in $\Ind\cc_t$ generalizing the finite dimensional bimodules
$$
\Hom_\ck (V^\nn_\mu, V^\nn_\l) \simeq V^\nn_\l\otimes V^\nn_\mu.
$$

\begin{claim}\label{clking}\cite{Ki}
If $\nu,\l,\mu\in \cp$, then the multiplicity
\begin{align}\label{fmult}
\dim \Hom_{G_n}(V^\nn_\nu,V^\nn_\l\otimes V^\nn_\mu)
\text{ is equal to }
\sum_{\eta,\omega,\xi} c^\l_{\eta,\omega} c^\mu_{\eta, \xi} c^\nu_{\omega, \xi},
\end{align}
where for three partitions $\l,\mu,\nu$ we denote by $c^\l_{\mu,\nu}$ the corresponding Littlewood-Richardon coefficient. 
\end{claim}

We will now proceed in the exact same manner as for the type $A$ case.  We will impose the same conditions on $\l,\mu$ as in Proposition \ref{cgupta}. That is, we let $\l^\nn$ be a sequence of partitions as in Definition \ref{dboldlambda}. Recall that it means that we fix $k, l\in \mathbb Z_{\ge 0}, \gamma\in\mathcal P$ and assume that 
\begin{itemize}
    \item $\ell(\l)\ll n$ (in the sense of Definition \ref{defll});
    
    \item $\l^\nn =[\alpha^\nn,\beta^\nn,\gamma]$ (in the sense of Definition \ref{defcutdiag}) and the sequences $\alpha^\nn, \beta^\nn$ are \textbf{nice} (in the sense of Definition \ref{deftheta}).
\end{itemize}

Let us fix some $\mathbf a=(a_1,\ldots,a_k)\in\Z^k$ and $\mathbf b=(b_1,\ldots,b_l)\in\Z^l$. We denote $|\mathbf a|= \sum_{i=1}^k a_i$ and $|\mathbf b|=\sum_{j=1}^l b_j.$ Then we fix some $\delta\in\mathcal P$ and put
$$
\mu^\nn = [\alpha^\nn+a,\beta^\nn +b,\delta],
$$
which is a well-defined partition  for almost all $n$ as $\alpha^\nn$ and $\beta^\nn$ are \textbf{nice}.

\begin{theorem}
Let $\l^\nn$ and $\mu^\nn$ be as above, then for any $\nu\in\mathcal P$ the multiplicity of $V_\nu^\nn$ in $\Hom_\ck(V^\nn_{\mu^\nn}, V^\nn_{\l^\nn})$ is constant for almost all $n$.
\end{theorem}

\begin{proof}

By Claim \ref{clking}, the desired multiplicity is equal to
$$
\sum_{\eta,\omega,\xi} c^{\l^\nn}_{\eta,\omega} c^{\mu^\nn}_{\eta, \xi} c^\nu_{\omega, \xi}.
$$

As in Proposition \ref{cgupta}, we claim that the sum in \ref{fmult} depends only on $k,l, \mathbf a, \mathbf b, \gamma, \delta.$ Though the sum is over all
partitions $ \eta,\omega,\xi $, we see that for a term to be nonzero, we must have $\eta\subset\l^\nn, \mu^\nn$.
There is a fixed (not depending on $n$) set of pairs $(\omega, \xi)$ so that
$$
c^\nu_{\omega, \xi} \neq 0,
$$
since $\nu $ is fixed and $|\omega|+|\xi|=|\nu|$. Thus, we will focus on $c^{\l^\nn}_{\eta,\omega}, c^{\mu^\nn}_{\eta,\xi}$. As in the proof of Proposition \ref{cgupta}, we again want to classify all possible $\eta=[\sigma, \tau,\epsilon]\subset \l^\nn, \mu^\nn$. Using the same notation as in Definition \ref{defskew}, we establish that
$$
c^{\l^\nn}_{\eta,\omega} = c_\omega(\mathbf c, \mathbf d, \gamma/\epsilon)
$$
and
$$
c^{\mu^\nn}_{\eta, \xi} = c_\xi(\mathbf c+\mathbf a, \mathbf d+\mathbf b, \delta/\epsilon)
$$
for some
\begin{align*}
    \mathbf c\in (-\mathbf a+\Z^k_{\ge 0})\cap \Z^k_{\ge 0},\\
    \mathbf d\in (-\mathbf b+\Z^l_{\ge 0})\cap \Z^l_{\ge 0}
\end{align*}
and $\epsilon\subset \gamma, \delta$.

So,
$$
\dim\Hom_{G_n}(V_\nu^\nn, V^\nn_{\l^\nn}\otimes V^\nn_{\mu^\nn}) = \sum_{\mathbf c,\mathbf d,\epsilon} c_\omega(\mathbf c, \mathbf d, \gamma/\epsilon)c_\xi(\mathbf c+\mathbf a, \mathbf d+\mathbf b, \delta/\epsilon)c^\nu_{\omega, \xi},
$$
where the sum is taken over all $\mathbf c,\mathbf d,\epsilon$ as above. 
\end{proof}

\begin{corollary}
\label{chomosp}
For $\l^\nn, \mu^\nn$ as above we define the triples
$$
\l = (\alpha, \beta, \gamma) = (\prod_\cf \alpha^\nn, \prod_\cf \beta^\nn, \gamma)
$$
and
$$
\mu = (\alpha+\mathbf a, \beta+\mathbf b, \delta)
$$
with $\alpha\in\C^k, \beta\in \C^k$.

Then, similarly to type A case, 
$$
\underline \Hom(\mu,\l) = \prod_\cf^F \Hom_\ck(V^\nn_{\mu^\nn},V^\nn_{\l^\nn})
$$
(where $F$ is a filtration on $\cp$ by finite subsets, e.g. $F^k\cp =\{\l|~|\l|\le k\}$) is a well-defined object of $\Ind\cc_t$ and is a Harish-Chandra bimodule of finite K-type.

Clearly, Lemma \ref{rextdefinitionofhom} applies in this case.
\end{corollary}

\subsection{Central characters}

Let $C_{2k}$ be the central elements defined in Section \ref{suniv}.  

\begin{definition}
Given a central character $\chi:U(\mathfrak{g}_t)\to \C$ we define the exponential central character to be the generating function
$$
\chi(z)= \frac{1}{e^{\frac{z}{2}}-e^{-\frac{z}{2}}} \sum_{k} \frac{\chi(C_{2k})}{(2k)!} z^{2k}.
$$
\end{definition}

The proof of Theorem 3.17 in \cite{U} (stated in Example \ref{exhcnonzero}) can be easily generalized to the case of the categories $\Rep(O_t)$ and $\Rep(Sp_t)$. One obtains the following result: 

\begin{claim}
If $\mathfrak{g}_t = \mathfrak o_t$ or $\mathfrak{sp}_t$, the category $\hc_{\chi,\psi}(\mathfrak{g}_t)$ is non-zero if and only if there exist some complex numbers $b_1,\ldots, b_r$, such that
$$
\chi(z)-\psi(z) = \frac{1}{e^{\frac{z}{2}}-e^{-\frac{z}{2}}} \sum_{i=1}^r \left( \mathrm{cosh}((b_i+1)z) - \mathrm{cosh}(b_iz)\right) = \sum_{i=1}^r\mathrm{sinh}\left(\frac{(2b_i+1)z}{2}\right).
$$
\end{claim}

\begin{lemma}
Let $\nu\in \cp$ and let $V_\nu$ be the corresponding simple object in $\cc_t$. Then $Z(U(\mathfrak{g}_t))$ acts on $V_\nu$ via the character $\chi_\nu(z)$ with
\begin{align*}
\chi_\nu(z) = \frac{1}{e^{\frac{z}{2}}-e^{-\frac{z}{2}}} \sum_{i=1}^{\ell(\nu)}\left(\cosh\left(\nu_iz+\tfrac{t}{2}z-iz\right)-\cosh\left(\tfrac{t}{2}z -iz\right)\right)
\end{align*}

\end{lemma}
\begin{proof}
Let us calculate the vector $\rho\in E$ explicitly. Our choice of simple roots will be standard: in type $B$ simple roots are
$$
\{e_i-e_{i+1}\}_{i<n}\cup\{e_n\},
$$
and in type $C$ we have:
$$
\{e_i-e_{i+1}\}_{i<n}\cup\{2e_n\}.
$$
Thus, in type $B$
$$
\rho = (n-\tfrac{1}{2}, n-\tfrac{3}{2},\ldots, \tfrac{3}{2}, \tfrac{1}{2}),
$$
and in type $C$:
$$
\rho =  (n, n-1,\ldots, 2, 1).
$$

In the case $G_n = O_{2n+1}$, we get
$$
\chi^\nn_\nu(C_{2k})= \sum_{i=1}^{\ell(\nu)} \left(\nu_i + \tfrac{2n+1-2i}{2}\right)^{2k}-\left(\tfrac{2n+1-2i}{2}\right)^{2k}.
$$

Thus,
$$
\chi^\nn(z) = \sum_{i=1}^{\ell(\nu)}\cosh\left(\nu_i z+\tfrac{2n+1-2i}{2}z\right)-\cosh\left(\tfrac{2n+1-2i}{2}z\right). 
$$

Now, since $t=\prod_\cf (2n+1)$, 
$$
\chi(z) = \frac{1}{e^{\frac{z}{2}}-e^{-\frac{z}{2}}} \sum_{i=1}^{\ell(\nu)}\left(\cosh\left(\nu_i z+\tfrac{t}{2}z-i z\right)-\cosh\left(\tfrac{t}{2}z -iz\right)\right).
$$

In the case $G_n= Sp_{2n}$:
$$
\chi^\nn_\nu(C_{2k})=\sum_{i=1}^{\ell(\nu)}\left(\nu_i+(n-i+1)-1\right)^{2k}-\left((n-i+1)-1\right)^{2k}.
$$
So, since $t=\prod_\cf 2n$, we get
$$
\chi_\nu(z) = \frac{1}{e^{\frac{z}{2}}-e^{-\frac{z}{2}}} \sum_{i=1}^{\ell(\nu)}\left(\cosh\left(\nu_iz+\tfrac{t}{2}z-iz\right)-\cosh\left(\tfrac{t}{2}z -iz\right)\right).
$$

\end{proof}

\begin{remark}
We showed that
$$
\chi_\nu(z) = \frac{1}{2(e^{\frac{z}{2}}-e^{-\frac{z}{2}}) }\sum_{i=1}^{\ell(\nu)}\left(e^{\nu_iz+\frac{t}{2}z-iz}-e^{\frac{t}{2}z-iz} + e^{-\nu_iz-\frac{t}{2}z+iz} - e^{-\frac{t}{2}z+iz}\right)
$$
(by definition of $\cosh$). Thus, if we define
$$
\widetilde\chi_\nu(z) =  \frac{1}{e^{\frac{z}{2}}-e^{-\frac{z}{2}}} \sum_{i=1}^{\ell(\nu)} e^{\frac{t}{2}z-iz}\left( e^{\nu_iz}-1 \right)
$$
we get
$$
\chi_\nu(z) = \frac{1}{2}(\widetilde \chi_\nu(z)-\widetilde\chi_\nu(-z)).
$$
\end{remark}

We see that the expression for $\widetilde \chi_\nu(z)$ is very similar to the exponential central character of the simple $GL_t$ module $V_{(\nu,\emptyset)}$. In particular, for $\l=[\alpha,\beta,\gamma]$ we can use Claim \ref{cl2} directly to express $\widetilde \chi_\l(z)$ in terms $\alpha_i,\beta_i,\gamma_i$. 

\begin{claim}
Let $\l=[\alpha, \beta, \gamma]$  with $\ell(\alpha) = k, \ell(\beta) = l, \ell(\gamma)=m$ and $q = e^z$. Then
$$
\widetilde\chi_\l(\log(q)) = q^{\frac{t+1}{2}}( \sum_{j=1}^k [\alpha_j+l]_q q^{-j} + \sum_{j=1}^l [\beta_j-m]_q q^{-\beta_j+j-1-k} + \sum_{j=1}^{m} [\gamma_j+l]_q q^{-k-j}).
$$
\end{claim}
\begin{corollary}\label{ccentrcharnottypea}
Let $\l=(\alpha, \beta, \gamma)$ and $\mu=(\alpha+\mathbf a, \beta+\mathbf b,\delta)$. The Harish-Chandra bimodule $\underline \Hom(\mu,\l)$ constructed in Corollary \ref{chomosp} lies inside the category $\hc_{\chi,\psi}$ with
$$
2\chi(\log(q)) = \widetilde\chi(\log(q))-\widetilde\chi(-\log(q));~2\psi(\log(q)) = \widetilde\psi(\log(q))-\widetilde\psi(-\log(q))
$$
and
$$
\widetilde\chi(\log(q)) = q^{\frac{t+1}{2}}( \sum_{j=1}^k [\alpha_j+l]_q q^{-j} + \sum_{j=1}^l [\beta_j-m]_q q^{-\beta_j+j-1-k} + \sum_{j=1}^{\ell(\gamma)} [\gamma_j+l]_q q^{-k-j}),
$$
$$
\widetilde\psi(\log(q))= q^{\frac{t+1}{2}}( \sum_{j=1}^k [\alpha_j+l+a_j]_q q^{-j} + \sum_{j=1}^l [\beta_j+b_j]_q q^{-\beta_j-b_j+j-1-k} + \sum_{j=1}^{\ell(\delta)} [\delta_j+l]_q q^{-k-j})
$$
\end{corollary}

\subsection{Spherical bimodules}

In this subsection we will construct the bimodules $\underline \End(\l)= \underline\Hom(\l,\l)$ as a quotient of $U_\chi$ with $\chi$ as in Corollary \ref{ccentrcharnottypea}. 

We want to construct a subbimodule of $U_\chi$ which we will denote $\underline{ \mathrm{Ann}}(\l)$ and that fits into a short exact sequence
$$
0\to \underline{\mathrm{Ann}} (\l) \to U_\chi\to \underline \End(\l) \to 0.
$$
Our plan is to take the filtered ultraproduct of the two-sided ideals $\mathrm{Ann}(V^\nn_{\l^\nn})$ with respect to the PBW-filtration and show that it is a well-defined object of $\Ind(\cc_t)$.

For the ease of notation we will assume further that $G_n$ is isomorphic either to $O_{2n}$  or to $Sp_{2n}$ and is embedded naturally into $GL_{2n}=GL_{V^\nn}$. Let $Q$ and $\omega$ be the symmetric and the skew-symmetric bilinear forms preserved by $G_n$ in each case. For the sake of computation we put
$$
    Q(v_i, v_j) = \delta_{n+i,j};
$$
$$
    \omega(v_i, v_j) = \begin{cases} \delta_{n+i,j}, \text{ if } i\le n,\\
    -\delta_{i, n+j}, \text{ otherwise},
    \end{cases}
$$
where $v_i$ is a basis of $V$ indexed by numbers $\{1,\ldots,2n\}$ modulo $2n$.

\begin{remark}
All the following computations work also for $G_n= O_{2n+1}$, where we define the symmetric bilinear form $Q$ in the same way as above and add the condition that
$$
Q(v_{2n+1}, v_i) = \delta_{i, 2n+1}.
$$
\end{remark}

The forms $Q,\omega$ provide an isomorphism between $\gl_{V^\nn}$ and $V^\nn\otimes V^\nn$ given by
$$
\widetilde Q: E_{ij}\mapsto v_i\otimes v_{n+j};
$$
$$
\widetilde \omega: E_{ij} \mapsto\begin{cases} -v_i\otimes v_{n+j}, \text{ if } j\le n;\\
v_i\otimes v_{n+j}, \text{ otherwise}.
\end{cases}
$$

The images of $\mathfrak {so}_{(V^\nn,Q)}$ and $\mathfrak {sp}_{(V^\nn,\omega)}$ under these maps are isomorphic to $\Lambda^2 V^\nn$ and  $S^2V^\nn$ correspondingly. Let $\epsilon = -1$  if $G_n= O_{2n}$ and  $\epsilon = 1$  if $G_n = Sp_{2n}$. Put $a_{ij} = v_i\otimes v_j + \epsilon v_j\otimes v_i$. Then $\{a_{ij}\}_{i,j\le 2n}$ span $\mathfrak{g}_n$. 

\begin{definition}
Define $A = \sum a_{ij} \otimes E_{ij} \in U(\mathfrak{g}_n)\otimes Mat_{2n}$ to be a matrix with coefficients in $U(\mathfrak{g}_n)$.
\end{definition}

\begin{lemma}\label{lannsimnontypea} (The analogue of Lemma \ref{lannsim}).
Let $I = \{i_1, i_2, i_3\}$ and $J = \{j_1, j_2, j_3\}$ be any two three-element subsets of $\{1,\ldots, n\}$. Then the minor $A_{I,J}$ annihilates $S^mV^\nn$ for all $m$.
\end{lemma}
\begin{proof}
As in the proof of Lemma \ref{lannsim} we want to consider the image of elements $a_{ij}$ in the algebra of differential operators on $(V^\nn)^*$. We get that $a_{ij}$ acts on $S^m V^\nn$ as $x_i\frac{\partial}{\partial x_{n+j}} + \epsilon x_j\frac{\partial}{\partial x_{n+i}}$
if $i,j\le n$.

Since $x_i$ and $\frac{\partial}{\partial x_{n+j}}$ commute for all $i,j\le n$, we denote $\frac{\partial}{\partial x_{n+j}}$ by $y_j$ and consider the image of $A_{I,J}$ in the commutative algebra $\ck[x_i,y_j]_{i, j\in I\cup J}$.

It is left to compute
$$
\det \begin{pmatrix}
x_{i_1}y_{j_1}+\epsilon x_{j_1} y_{i_1}& x_{i_1}y_{j_2}+\epsilon x_{j_2} y_{i_1}& x_{i_1}y_{j_3}+\epsilon x_{j_3} y_{i_1}\\
x_{i_2}y_{j_1}+\epsilon x_{j_1} y_{i_2}& x_{i_2}y_{j_2}+\epsilon x_{j_2} y_{i_2}& x_{i_2}y_{j_3}+\epsilon x_{j_3} y_{i_2}\\
x_{i_3}y_{j_1}+\epsilon x_{j_1} y_{i_3}& x_{i_3}y_{j_2}+\epsilon x_{j_2} y_{i_3}& x_{i_3}y_{j_3}+\epsilon x_{j_3} y_{i_3}
\end{pmatrix}
$$
and verify that the resulting expression is indeed equal to zero in $k[x_i, y_j]$.

We note that
$$
\begin{pmatrix}
x_{i_1}y_{j_1}+\epsilon x_{j_1} y_{i_1}& x_{i_1}y_{j_2}+\epsilon x_{j_2} y_{i_1}& x_{i_1}y_{j_3}+\epsilon x_{j_3} y_{i_1}\\
x_{i_2}y_{j_1}+\epsilon x_{j_1} y_{i_2}& x_{i_2}y_{j_2}+\epsilon x_{j_2} y_{i_2}& x_{i_2}y_{j_3}+\epsilon x_{j_3} y_{i_2}\\
x_{i_3}y_{j_1}+\epsilon x_{j_1} y_{i_3}& x_{i_3}y_{j_2}+\epsilon x_{j_2} y_{i_3}& x_{i_3}y_{j_3}+\epsilon x_{j_3} y_{i_3}
\end{pmatrix}
 = \begin{pmatrix}
x_{i_1}& y_{i_1}\\
x_{i_2}& y_{i_2}\\
x_{i_3}& y_{i_3}
\end{pmatrix}  \begin{pmatrix}
y_{j_1}& y_{j_2}& y_{j_3}\\
\epsilon x_{j_1}& \epsilon x_{j_2} & \epsilon x_{j_3}
\end{pmatrix},
$$
so its rank is at most two.
\end{proof}

\begin{corollary}\label{cannnontypea}
(The analogue of Corollary \ref{cann}). Let $I=\{i_1,\ldots,i_{2k+1}\}, J=\{j_1,\ldots, j_{2k+1}\}$ be any two $2k+1$-element subsets of $\{1,\ldots, n\}$. For any $\l_1,\ldots, \l_{k}\in \mathbb N$ the minor $A_{I,J}$ annihilates 
$$
W_\l = S^{\l_1}V^\nn\otimes \ldots \otimes S^{\l_k}V^\nn.
$$
\end{corollary}
\begin{proof}
The proof is analogous to the proof of Corollary \ref{cann}.

$U(\mathfrak{g}_n)^{\otimes k}$ acts naturally on $W_{\l}$ and the action of $a_{ij}\in U(\mathfrak{g}_n)$ on $W_{\l}$ coincides with the action of $\Delta^{k-1}(a_{ij})\subset (U(\mathfrak{g}_n))^{\otimes k}$, where 
$$
\Delta(a_{ij})= a_{ij}\otimes 1+1\otimes a_{ij},
$$
i.e. $\Delta$ is the comultiplication on $U(\mathfrak{g}_n)$. 

For $m = 1, \ldots, k$, let $a_{ij}^{(m)}\subset (U(\mathfrak{g}_n))^{\otimes k}$ be the elements
$$
a_{ij}^{(m)} := \underbrace{1\otimes\ldots \otimes 1}_{m-1} \otimes~
a_{ij} \otimes 1 \otimes \ldots \otimes 1.
$$
Then
$$\Delta^{k-1}(a_{ij}) = \sum_{m=1}^{k} a_{ij}^{(m)}. $$

We note that $a_{ij}^{(m)}$ and $a_{kl}^{(n)}$ commute for all $i,j,k,l$ and distinct $m$ and $n$.

Let $A^{(m)} = (a_{ij}^{(m)})_{ij} \in (U(\mathfrak{g}_n))^{\otimes k}\otimes Mat_{2n}$. Then the action of the minor $A_{I,J}$ on $W_{\l}$ coincides with the action of $(A^{(1)}+\ldots+A^{(k)})_{I,J}$.

Let $B:= \sum_{m=1}^k A^{(m)}$. Then $\Lambda^{2k+1} B$ is well-defined and acts naturally on $$(U(\mathfrak{g}_n))^{\otimes k}\otimes \Lambda^{2k+1} V^\nn.$$ We have $$B_{I,J} = \langle v_{i_1}\wedge \ldots \wedge v_{i_{2k+1}}| B(v_{j_1}\wedge \ldots\wedge v_{j_{2k+1}})\rangle. $$
Now 
$$
B(v_{j_1}\wedge \ldots\wedge v_{j_{2k+1}}) = \sum_{1\le m_1,\ldots,m_{2k+1}\le k} A^{(m_1)}(v_{j_1})\wedge\ldots\wedge A^{(m_{2k+1})}(v_{j_{2k+1}}).
$$

In each summand at least one of the matrices $A^{(m)}$ occurs thrice. Since matrix elements of distinct $A^{(m)}$ commute, we can move these three to the left. So, up to permutation of indices, each summand is equal to
$$
A^{(m)}v_{j_1}\wedge A^{(m)} v_{j_2}\wedge A^{(m)} v_{j_3}\wedge\ldots = \sum A^{(m)}_{\{i_1,i_2,i_3\},\{j_1,j_2, j_3\}}v_{i_1}\wedge v_{i_2}\wedge v_{i_3} \wedge \ldots
$$

We get that by Lemma \ref{lannsimnontypea}, $A^{(m)}_{\{i_1,i_2, i_3\},\{j_1,j_2. j_3\}} $ acts trivially on $W_{\l}$. Hence, $A_{I,J}\in \mathrm{Ann}(W_{\l})$.

\end{proof}

Now let us take some $\l \in \cp$ and let $d = d(\l)$. For $1\le i\le d$ put 
$$
\mu_i= \l_i - d,
$$
$$
\nu_i= \l'_i. 
$$

Then $V^\nn_\l$ is a submodule in $\mathbb S_\l$ and hence in
$$
\bigotimes_{i=1}^{d} S^{\mu_i}V^\nn\otimes  \bigotimes_{i=1}^{d} \Lambda^{\nu_i}V^\nn 
$$
 Thus, $V^\nn_\l$ is a submodule in
$$
R_d= (S^{\bullet}V^\nn)^{\otimes d}\otimes(\Lambda^{\bullet}V^\nn)^{\otimes d}.
$$

\begin{lemma}\label{lannnontypea}(The analogue of Lemma \ref{lann}).  If $R_d$ is as above, then
$$F^{(2d+1)(d(2d+1)+1)}\mathrm{Ann}(R_{d}) \neq 0$$
\end{lemma}

\begin{proof}

We have $R_d = S^\bullet(V^{\nn\oplus d})\otimes \Lambda^\bullet(V^{\nn\oplus d})$.
 The elements $a_{ij}$ with $i,j\le n$ act on the symmetric algebra  as 
$$
\sum_{a=1}^{d} x_{i,a}\frac{\partial}{\partial x_{(n+j),a}} +\epsilon  x_{j,a}\frac{\partial}{\partial x_{(n+i),a}},
$$
where $x_{i,a}$ are even variables.

Similarly, they act on the exterior algebra as 
$$
\sum_{a=1}^{k} \xi_{i,a}\frac{\partial}{\partial \xi_{(n+j),a}} +\epsilon  \xi_{j,a}\frac{\partial}{\partial \xi_{(n+i),a}},
$$
where $\xi_{i,a}$ are odd variables. 

Let us take some subsets $I=\{i_1,\ldots,i_{2d+1}\},J=\{j_1,\ldots,j_{2d+1}\}$ inside  $\{1,\ldots, n\}$ with $I\cap J=\emptyset$. For any $i,j\le n$ multiplication by $x_{i,a}$ (or $\xi_{i,a}$) commutes with differentiation $\frac{\partial}{\partial x_{(n+j),a}}$ (or $\frac{\partial}{\partial \xi_{(n+j),a}}$). Let us denote $$y_{j,a}= \frac{\partial}{\partial x_{(n+j),a}},$$ $$\eta_{j,a}= \frac{\partial}{\partial \xi_{(n+j),a}}.$$

Then the representation of the subalgebra $U_{I,J}$ of $U(\mathfrak{g}_n)$ generated by $a_{ij}$ with $i\in I$ and $j\in J$ on $R_{d}$ is given by a homomorphism
$$
\phi:U_{I,J} \to \mathbbm{k}[x_{i,a},y_{i,a},\xi_{i,a},\eta_{i,a}],
$$
$$
a_{ij}\mapsto \sum_{a=1}^{d} (x_{i,a} y_{j,a} + \epsilon x_{j,a}y_{i, a}) + \sum_{a=1}^{d} (\xi_{i,a}\eta_{j,a} +\epsilon   \xi_{j,a}\eta_{i,a}),
$$
where the algebra on the right-hand-side is a free supercommutative algebra with even generators $x_{i,a},y_{i,a}$  and odd generators $\xi_{i,a},\eta_{i,a}$, and the indices run over the following sets:$$i\in I\cup  J, a\in \{1,\ldots,d\}.$$

By Corollary \ref{cannnontypea}, $A_{I,J}$ (the corresponding $(2d+1)$ by $(2d+1)$ minor) annihilates $S^\cdot(V^{\oplus d})$. Therefore, $\phi(A_{I,J})$ lies in the nilradical of our algebra (i.e. in the ideal generated by the variables $\xi,\eta$). Moreover, each nonzero monomial in $\phi(A_{I,J})$ must contain equal number of variables $\xi$ and $\eta$, because it is true for $\phi(a_{ij})$. We note that the number of odd variables in our algebra is equal to $2N=4d(2d+1)$. Thus, $\phi(A^{N+1}_{I,J})=0$.

Finally, since $\deg A_{I,J} = 2d+1$ and $A_{I,J}^{2d(2d+1)+1}\in \mathrm{Ann}(R_d)$ we get that
$$
F^{(2d+1)(2d(2d+1)+1)}R_d\neq 0.
$$
\end{proof} 

\begin{corollary}
If $\l^\nn$ is a sequence of partitions such that $d(\l^\nn)$ are bounded, then there exists $N\in\N$ with
$$
F^N\mathrm{Ann}(V^\nn_{\l^\nn})\neq 0
$$
for almost all $n$.
Moreover, the filtered ultraproduct
$$
\prod^F_\cf \mathrm{Ann}(V^\nn_{\l^\nn})
$$
is nonzero and lies in $\Ind\cc_t$.
\end{corollary}
\begin{proof}
The first part of the statement is just a consequence of Lemma \ref{lannnontypea}: we can take $d = \max(d(\l^\nn))$ and 
$$
N  = (2d+1)(2d(2d+1)+1).
$$

Now, since for any $l$
$$
F^l \mathrm{Ann}(V^\nn_{\l^\nn})\subset F^l U(\mathfrak{g}_n)
$$
and since, consequently, $(\mathrm{Ann}(V^\nn_\l), F)$ satisfy the conditions of Lemma \ref{lwelldefup},
$$
\underline{\mathrm{Ann}}(\l)= \prod^F_\cf \mathrm{Ann}(V^\nn_{\l^\nn})\in \Ind\cc_t.
$$

Moreover, $F^N\underline{\mathrm{Ann}}(\l) = \prod_\cf F^N\mathrm{Ann}(V^\nn_{\l^\nn}) \neq 0.$
\end{proof}

\begin{remark}
Let $\l$ be as in Corollary \ref{chomosp} and $\l = \prod_\cf \l^\nn$.  Then
$$
\underline \End(\l) = U(\mathfrak{g}_t)/\underline{\mathrm{Ann}}(\l),
$$
where 
$$
\underline{\mathrm{Ann}}(\l) = \prod_\cf^F \mathrm{Ann}(V^\nn_{\l^\nn})
$$
and
$$
\underline \End(\l)= \underline \Hom(\l,\l).
$$
\end{remark}

$$
$$

\end{document}